\documentclass[11pt]{amsart}

\usepackage[T1]{fontenc}
\usepackage[utf8]{inputenc}
\usepackage{lmodern}
\usepackage{microtype}
\usepackage{amsmath,amssymb,mathtools}
\usepackage{mathrsfs}
\usepackage{amsthm}
\usepackage{enumitem}
\usepackage{booktabs}
\usepackage[hidelinks,hypertexnames=false]{hyperref}
\usepackage{url}
\usepackage[margin=0.8in]{geometry}
\usepackage{comment}

\allowdisplaybreaks
\setlist[enumerate]{leftmargin=2.2em}
\setlist[itemize]{leftmargin=2.0em}

\newtheorem{theorem}{Theorem}[section]
\newtheorem{proposition}[theorem]{Proposition}
\newtheorem{lemma}[theorem]{Lemma}

\theoremstyle{definition}

\theoremstyle{remark}
\newtheorem{remark}[theorem]{Remark}

\newcommand{\R}{\mathbb R}

\newcommand{\Vol}{\operatorname{Vol}}
\newcommand{\Area}{\operatorname{Area}}
\newcommand{\Ric}{\operatorname{Ric}}
\newcommand{\Rm}{\operatorname{Rm}}

\newcommand{\Hess}{\nabla^2}
\newcommand{\inj}{\operatorname{inj}}

\newcommand{\Id}{\operatorname{Id}}
\newcommand{\tr}{\operatorname{tr}}
\newcommand{\im}{\operatorname{Im}}

\title[Sharp Yamabe bounds on spherical space forms]{Sharp Yamabe Bounds for Spherical Three-Space Forms via Four-Dimensional Poincar\'e--Einstein Compactness}

\author{Jakob Loren}
\address{Department of Mathematics, University of California Santa Barbara, California, USA}
\date{13 Oct. 2026}

\begin{document}

\begin{abstract}
Let $\Gamma<SO(4)$ be a nontrivial finite group acting freely on
$S^3$, and let $M=S^3/\Gamma$. We prove that every smooth conformal
class $C$ on $M$ satisfies
\[
Y(M,C)\le |\Gamma|^{-2/3}Y(S^3),
\]
with equality for the quotient-round conformal class. Consequently,
\[
\sigma(M)=|\Gamma|^{-2/3}Y(S^3).
\]

Our approach uses four-dimensional Poincar\'e--Einstein fillings of the
ball. We first establish compactness for fillings whose conformal
infinities range over a compact family of positive-scalar-curvature
metrics. Since compactness alone does not guarantee continuation through
degenerate Einstein metrics, we use equivariant Fredholm parity to
obtain a filling with the prescribed boundary conformal class and
group action.
A common interior fixed point then makes the
Li--Qing--Shi boundary test function invariant, and its descent to
$M$ gives the sharp covering factor.
\end{abstract}

\maketitle

\section{Introduction}\label{sec:intro}

In two dimensions, the uniformization theorem associates a metric
of constant curvature to every conformal class on a closed surface.
The Yamabe problem asks for an analogue in higher dimensions, with
scalar curvature taking the place of Gaussian curvature. 

Let $M^n$ be a closed manifold of dimension $n\ge3$. For a
Riemannian metric $g$, consider the normalized total scalar
curvature
\[
\mathcal S(g)=
\frac{\displaystyle\int_M R_g\,dV_g}
     {\displaystyle\Vol(M,g)^{(n-2)/n}}.
\]
The normalization makes $\mathcal S$ unchanged by constant
rescaling of the metric. Its infimum within a conformal class
$C=[g]$ is the Yamabe constant
\[
Y(M,C)=\inf_{h\in C}\mathcal S(h).
\]
The work of Yamabe, Trudinger, Aubin, and Schoen established that
this infimum is attained by a metric of constant scalar curvature;
see \cite{Yamabe1960,Trudinger1968,Aubin1976,Schoen1984,LeeParker1987}.
Thus every conformal class contains such a
metric, although the metric itself need not be unique.

A different question arises when the conformal class is also
allowed to vary. Following Kobayashi and Schoen
\cite{Kobayashi1987,Schoen1989}, one defines
\[
\sigma(M)=\sup_C Y(M,C).
\]
This is the smooth Yamabe invariant of $M$. Whereas the Yamabe
constant belongs to conformal geometry, its supremum is an
invariant of the smooth manifold. The distinction is important:
solving the Yamabe problem in each conformal class does not tell
us which conformal classes yield the largest value.

The round sphere provides the universal upper bound. Aubin's
sharp comparison gives
\[
Y(M,C)\le Y(S^n,[g_{\rm rd}]),
\qquad
\sigma(M)\le Y(S^n,[g_{\rm rd}]).
\]
Equality in the second inequality need not characterize the
sphere. For example, the product $S^{n-1}\times S^1$ has the same
Yamabe invariant as $S^n$ \cite{Kobayashi1987,Schoen1989}.
In other settings the invariant is sensitive to considerably
finer geometry. LeBrun's work on complex surfaces, for example,
revealed connections with Kodaira dimension and
Seiberg--Witten theory \cite{LeBrun1999}. Explicit values are
nevertheless rare.

Dimension three is particularly relevant.
The topology of manifolds admitting positive scalar curvature
is understood: every closed orientable such manifold is a
connected sum of spherical space forms and copies of
$S^2\times S^1$. On the other hand, geometrization and the
work of Perelman give a description of the Yamabe invariant in
the nonpositive case in terms of the hyperbolic pieces of the
manifold; see
\cite{Perelman2002,Perelman2003a,Perelman2003b,
AkutagawaIshidaLeBrun2007}.
Thus the positive case presents a somewhat different problem.
The topology is relatively simple, yet even for a manifold
carrying a metric of constant positive sectional curvature,
it can be difficult to show that this metric gives the largest
possible Yamabe constant.

The first result of this kind was obtained by Bray
and Neves \cite{BrayNeves2004}. They proved
\[
\sigma(\mathbb RP^3)=2^{-2/3}Y(S^3).
\]
The factor on the right is precisely the one suggested by the
round double covering $S^3\to\mathbb RP^3$. What is remarkable
about their proof is the route by which this factor is recovered.
An observation of Schoen relates the conformal geometry of
punctured $\mathbb RP^3$ to the Schwarzschild geometry appearing
in the Riemannian Penrose inequality. Bray and Neves developed
this connection using inverse mean curvature flow and the
monotonicity of Hawking mass. Their method illustrates how a
sharp inequality for scalar curvature may emerge from the
global geometry of a punctured three-manifold.

For a general spherical space form, the same quotient-round
value is the candidate, as anticipated in Schoen's
conjectural picture \cite{Schoen1989}. Let
\[
\Gamma<SO(4)
\]
be a nontrivial finite group acting freely on $S^3$, and set
\[
M=S^3/\Gamma,\qquad N=|\Gamma|.
\]
The round metric descends to a metric $h_{\rm rd}$ on $M$ of
scalar curvature $6$ and volume $2\pi^2/N$. Its Yamabe
constant is
\[
Y(M,[h_{\rm rd}])
=N^{-2/3}Y(S^3).
\]
The question is whether some other conformal class can give
a larger value.

Our main result answers this question.

\begin{theorem}
For every smooth conformal class $C$ on $S^3/\Gamma$,
\[
Y(S^3/\Gamma,C)
\le |\Gamma|^{-2/3}Y(S^3).
\]
Equality holds for the quotient-round conformal class.
Consequently,
\[
\sigma(S^3/\Gamma)
=|\Gamma|^{-2/3}Y(S^3).
\]
\end{theorem}

There is a simple reason for the appearance of the group order,
but also a subtle reason why the theorem does not follow from
the ordinary Yamabe inequality on the sphere.

Let $\pi:S^3\to M$ be the covering map, let $h$ be a metric
on $M$, and put $q=\pi^*h$. For a nonzero function $u$ on $M$,
the three-dimensional Yamabe quotient satisfies
\[
\mathcal Q_q(\pi^*u)
=N^{2/3}\mathcal Q_h(u).
\]
Indeed, both the energy integral and the integral of the sixth
power are multiplied by $N$, while the latter appears to the
power $1/3$ in the denominator. Thus the desired bound would
follow at once if, for every lifted conformal class $[q]$, we
could find a $\Gamma$-invariant function $w$ on $S^3$ with
\[
\mathcal Q_q(w)\le Y(S^3).
\]
Aubin's inequality guarantees the existence of a good test
function on $S^3$, but it gives no reason for that function to
be invariant under $\Gamma$. This is the essential difficulty.
We need a construction which remembers the group action while
retaining the sharp constant of the sphere.

The construction used here comes from four-dimensional
Poincar\'e--Einstein geometry. A Poincar\'e--Einstein metric
$G$ on the interior of the four-ball satisfies
\[
\Ric_G=-3G
\]
and, after multiplication by the square of a suitable defining
function, extends to the boundary. The induced conformal class
on $S^3$ is called its conformal infinity. Hyperbolic
four-space is the simplest example: in its ball model, the
conformal infinity is the round three-sphere.

We may therefore regard a Poincar\'e--Einstein filling as a
way of associating four-dimensional geometry to a conformal
class on $S^3$. The usefulness of this association is that
distances and volumes inside the filling give rise to functions
on its conformal boundary. The relative-volume construction
of Li, Qing, and Shi \cite{LiQingShi2017} supplies just the
Yamabe comparison we need, provided that the filling and a
suitable interior base point respect the action of $\Gamma$.

It remains to produce such a filling for an arbitrary
$\Gamma$-invariant conformal class of positive Yamabe type.
The round conformal class has the hyperbolic filling, and
the deformation theorem of Graham and Lee
\cite{GrahamLee1991} supplies fillings for nearby boundary
data. There is also a path joining the round boundary metric
to any prescribed lifted positive-scalar-curvature metric:
the contractibility theorem of Bamler and Kleiner
\cite{BamlerKleiner2019}, applied on $M$ and followed by
lifting to $S^3$, gives such a path.

We are led to a continuation problem. Starting with the
hyperbolic metric, can one follow Einstein fillings along
the entire path of conformal infinities?

Two difficulties arise, and they are of different kinds.
The first is geometric: a sequence of fillings may cease to
be compact. The second is analytic and topological: even if
compactness is known, a filling may be degenerate, so the
implicit function theorem cannot always continue it to
nearby boundary data. The proof treats these difficulties
separately.

We begin with compactness. Chang and Ge's theory of
scalar-flat adapted compactifications \cite{ChangGe2025}
allows the Einstein equation to be studied on the closed
four-ball, where curvature and boundary geometry can be
compared directly. Our compactness argument begins by 
excluding curvature concentration. The resulting curvature 
estimates then yield a bound on the total Weyl energy.

To see the nature of the argument, suppose that a sequence
of fillings develops unbounded curvature. Rescaling around
points of large curvature produces a complete limiting
geometry. There are two possibilities, according to whether
the rescaled points remain at bounded distance from the
boundary or escape into the interior.

In the first case the boundary becomes Euclidean. We prove
a rigidity theorem for complete Einstein fillings with flat
conformal boundary. A renormalized distance function yields
a differential inequality on $\mathbb R^3$. Hyperbolic
volume comparison gives an upper bound for its critical
$L^6$ norm, while the sharp Euclidean Sobolev inequality
gives the reverse bound. Equality forces the Einstein
metric to be hyperbolic, and its adapted compactification
to be flat. This contradicts the curvature retained by
the rescaling.

The interior case leads to a different phenomenon. The
adapted defining-function equations force the limiting
metric to be Ricci-flat. A uniform Sobolev inequality,
together with four-dimensional curvature estimates, gives
Euclidean volume growth and finite total curvature.
The limit is therefore asymptotically locally Euclidean,
or ALE: far from a compact set, it resembles Euclidean
four-space modulo a finite group.

Here topology becomes decisive. The compact core of the
ALE limit must embed in the original four-ball. Such an
embedding imposes strong restrictions on the topology
of the core and on the spherical three-manifold at
infinity. The Chern--Gauss--Bonnet and signature formulas
then relate the curvature of the limit to the eta
invariant of that boundary. The resulting inequality
is incompatible with the possible nontrivial links.
Thus an interior curvature bubble cannot occur either.

Once both possibilities have been excluded, the curvature
bounds yield the missing volume and Weyl-energy estimates.
A finite-regularity elliptic bootstrap then completes the
compactness theorem. 

We next return to the continuation problem. Compactness
does not by itself imply existence at the far end of a
path. For Einstein metrics, the difficulty arises when the
linearized Einstein operator has a nontrivial kernel: the implicit
function theorem then no longer guarantees local continuation.

The appropriate replacement for ordinary continuation is
a mod-two Fredholm degree, or parity principle. If the solutions 
over a path form a compact one-dimensional manifold, then they
have an even number of boundary points. In our problem
the round conformal class has exactly one regular filling,
the hyperbolic metric. If the terminal class had no
filling, the space of solutions would have just one
boundary point, an impossibility. The difficulty is to
make this picture valid when degenerate fillings occur.

For this purpose we construct compatible
finite-dimensional continuation families for the
Einstein equation. The key analytic ingredient is a weighted
Einstein--DeTurck equation with harmonic-map gauge fixing,
using the deformation theory of
\cite{Biquard2000,Lee2006,Anderson2010} and boundary
regularity results such as \cite{CDLS2005}.
The boundary expansion and the choice of weight are
arranged so that the compactified solutions possess
sufficient regularity. The finite-dimensional parameters
are chosen to span the obstruction space to solving the
linearized equation; a boundary Green identity shows
that the required invariant variations can be found.

There is a further point which cannot be omitted. Local
families constructed in different gauges must describe
the same geometric solutions, and their changes of
coordinates must be smooth. We establish this by using
elliptic regularity to retain enough spatial derivatives
for the finite-dimensional families. We also compare
the geometric compactness topology with the topology
of the gauge-fixed equations. In this way the compactness
proved earlier becomes compactness of the actual
continuation space, rather than merely compactness of
metrics modulo an unspecified choice of coordinates.

After adjoining a common finite-dimensional space of
boundary perturbations, one obtains a compact
finite-dimensional continuation structure. Sard's
theorem provides a transverse one-dimensional
solution manifold, and the parity argument applies.
The perturbations are chosen to vanish at the endpoints,
so the resulting contradiction forces a filling of
the prescribed conformal infinity.

The final step is comparatively short. The equivariant
filling carries an isometric action of $\Gamma$ extending
the given action on $S^3$. A fixed-point theorem for
finite group actions gives a common point in the
interior of the ball. Taking distances from this point
in the Li--Qing--Shi construction produces a
$\Gamma$-invariant boundary function $w$. The sharp
relative-volume inequality gives
\[
\mathcal Q_q(w)\le Y(S^3).
\]
Since $w$ descends to a function $\bar w$ on $M$, the
covering identity yields
\[
Y(M,C)
\le \mathcal Q_h(\bar w)
=N^{-2/3}\mathcal Q_q(w)
\le N^{-2/3}Y(S^3),
\]
as required.

The proof thus brings together two forms of rigidity.
Geometric rigidity prevents Einstein fillings from
escaping by curvature concentration; parity prevents
them from disappearing at a degenerate solution.
Together they allow a four-dimensional comparison
to produce the sharp three-dimensional invariant.

\section{Main results and notation}\label{sec:results}

We begin by fixing the normalizations used throughout the paper.
For a closed Riemannian three-manifold $(M^3,h)$ and a nonzero
function $u\in H^1(M)$, the Yamabe quotient is
\begin{equation}\label{eq:yamabe-quotient}
\mathcal Q_h(u):=
\frac{\displaystyle\int_M\left(8|du|_h^2+R_hu^2\right)\,dV_h}
{\displaystyle\left(\int_M|u|^6\,dV_h\right)^{1/3}},
\qquad
Y(M,[h]):=\inf_{0\neq u\in H^1(M)}\mathcal Q_h(u).
\end{equation}
The constants $8$ and $6$ are those of the conformal Laplacian
and the critical Sobolev exponent in dimension three. We write
\begin{equation}\label{eq:sphere-normalization}
\omega_3:=\Vol(S^3,g_{\rm rd})=2\pi^2,
\qquad
\Lambda:=Y(S^3,[g_{\rm rd}])=6\omega_3^{2/3}.
\end{equation}
Throughout, $\Gamma<SO(4)$ denotes a nontrivial finite subgroup
acting freely on $S^3$. We put $N:=|\Gamma|$, let $M=S^3/\Gamma$,
and write $\pi:S^3\to M$ for the quotient map. Thus $N$ is both
the order of the group and the degree of the covering.

Our principal result is the following.

\begin{theorem}[Sharp classwise Yamabe bound]\label{thm:main}
For every smooth conformal class $C$ on $M$,
\begin{equation}\label{eq:main-bound}
Y(M,C)\le N^{-2/3}\Lambda.
\end{equation}
Equality holds for the quotient-round conformal class $C_{\rm rd}$:
\begin{equation}\label{eq:round-equality}
Y(M,C_{\rm rd})=N^{-2/3}\Lambda.
\end{equation}
In particular,
\begin{equation}\label{eq:main-sigma}
\sigma(S^3/\Gamma)=N^{-2/3}\Lambda.
\end{equation}
\end{theorem}

The equality statement follows from the sharp Yamabe inequality
on the round sphere and the elementary behavior of the Yamabe
quotient under a finite covering. The inequality for an arbitrary
conformal class is the substantial part of the theorem.

If $Y(M,C)\le0$, there is nothing to prove. We may therefore
suppose that $C$ has positive Yamabe constant and choose a metric
$h\in C$ of positive scalar curvature. As explained in the
introduction, the space of positive-scalar-curvature metrics on
$M$ is path connected. We can consequently join the
quotient-round metric to $h$ by a smooth path $h_s$ of such metrics.
Lifting to the sphere gives
\[
q_s=\pi^*h_s,\qquad 0\le s\le1,
\]
a smooth compact family of $\Gamma$-invariant boundary metrics
with $q_0=g_{\rm rd}$.

This path provides the setting for the proof. We seek an Einstein
metric on the interior of $B^4$ whose conformal infinity is
$[q_1]$ and whose isometries extend the prescribed action of
$\Gamma$ on the boundary. The hyperbolic metric gives such a
filling at $s=0$. The problem is to carry this filling to
$s=1$.

There are two distinct questions. First, can a family of
Einstein fillings escape by developing unbounded curvature?
Second, if such escape is excluded, can a filling nevertheless
disappear when the linearized Einstein equation becomes
degenerate? The first question is answered by the compactness
theorem below; the second by the equivariant continuation
theorem developed in
Section~\ref{sec:finite-equivariant-continuation}.

For continuation it is useful to retain the boundary metric
itself, rather than only its conformal class. Near a given
filling $G$, the Einstein equation is written in harmonic
identity gauge on complete weighted H\"older spaces. Its
linearization has a finite-dimensional space $K_G$ of
square-integrable infinitesimal Einstein deformations,
usually called the $L^2$ Jacobi space. The obstruction to
solving the linearized equation is identified with this
same space by a self-adjoint pairing.

The presence of $\Gamma$ requires more than the ordinary
numerical index-zero statement. Both the kernel and the
obstruction space must be identified as representations of
the group. The resulting invariant kernel and cokernel
are each $K_G^\Gamma$. A boundary Green identity then shows
that the invariant obstructions can be reached by finitely
many smooth $\Gamma$-invariant variations of the prescribed
boundary metric.

These variations serve as additional parameters in the
Einstein equation. The enlarged equation has smooth
finite-dimensional local solution manifolds, even at
degenerate fillings. We prove that different choices of
harmonic gauge lead to smoothly compatible descriptions
of the same geometric solutions. This permits us to work
with a Hausdorff space of \emph{marked fillings}: fillings
for which the identification of the conformal boundary
with the prescribed $S^3$ is retained, and equivalence is
taken only under diffeomorphisms fixing that boundary.

The distinction between marked and unmarked fillings is
important. The original action of $\Gamma$ on $S^3$ must
remain identifiable throughout continuation; otherwise,
a limiting metric could exist without giving the
equivariant filling required for the Yamabe argument.

Once compactness has been established in this marked
solution space, the continuation argument becomes a
question of parity. The round boundary has exactly one
regular filling, namely the hyperbolic metric. Suppose
that the terminal boundary metric had no equivariant
filling. By perturbing the finite-dimensional boundary
parameters away from the endpoints, we obtain a compact
one-dimensional manifold whose only boundary point is
the hyperbolic filling. This is impossible, since every
compact one-dimensional manifold has an even number of
boundary points.

Thus an equivariant filling exists at $s=1$, whether or
not its Jacobi kernel is trivial. Boundary-identity
rigidity ensures that the prescribed action of $\Gamma$
extends uniquely to a coherent action by isometries of
the filling.

The passage from this existence result to the Yamabe
inequality is short. The finite group action has a
common interior fixed point $p$. Using $p$ as the
base point in the Li--Qing--Shi construction gives a
nonzero $\Gamma$-invariant boundary function
$w\in H^1(S^3)$ satisfying
\[
\mathcal Q_q(w)\le\Lambda.
\]
Because $w$ is invariant, it descends to a function
$\bar w$ on $M$. With $q=\pi^*h$ and $w=\pi^*\bar w$,
the covering formula reads
\[
\mathcal Q_q(w)=N^{2/3}\mathcal Q_h(\bar w).
\]
Consequently,
\[
Y(M,C)\le\mathcal Q_h(\bar w)
=N^{-2/3}\mathcal Q_q(w)
\le N^{-2/3}\Lambda.
\]
The reverse inequality in the quotient-round class
follows by lifting every test function to the round
sphere. This proves the equality statement as well.

We turn now to the geometric compactness theorem on
which the continuation argument rests. Let
$\operatorname{Met}^{k,\alpha}_{\rm PSC}(S^3)$ denote
the space of $C^{k,\alpha}$ metrics on $S^3$ having
positive scalar curvature. For a Poincar\'e--Einstein
metric $G$ with prescribed boundary representative $q$,
we use the scalar-flat adapted compactification
$g=\rho^2G$. Its purpose is to replace the complete
metric $G$, which has infinite volume and lies at
infinite distance from its conformal boundary, by a
metric on the closed four-ball for which curvature and
boundary estimates can be formulated in ordinary
geometric terms.

The theorem asserts that a compact family of positive
boundary metrics prevents degeneration of these adapted
compactifications.

\begin{theorem}[Compactness over compact PSC boundary families]\label{thm:compactness}
Fix an integer $k\ge 6$ and $0<\alpha<1$.  Let
\[
\mathcal K\Subset \operatorname{Met}^{k,\alpha}_{\rm PSC}(S^3)
\]
be compact, and let $q_i\in\mathcal K$ be smooth.  Let $G_i$ be smooth
Poincar\'e--Einstein metrics on $\operatorname{int}B^4$ satisfying
\begin{equation}\label{eq:PE-sequence}
\Ric(G_i)=-3G_i,
\qquad
\partial_\infty[G_i]=[q_i].
\end{equation}
Let $g_i=\rho_i^2G_i$ be the scalar-flat adapted compactification determined by
$q_i$.  Then there exist a subsequence, boundary-fixing diffeomorphisms
$\Phi_i:B^4\to B^4$, and a scalar-flat adapted compactification
$g_\infty=\rho_\infty^2G_\infty$ of a smooth Poincar\'e--Einstein metric such
that, for every $0<\alpha'<\alpha$,
\begin{equation}\label{eq:compactness-conclusion}
\Phi_i^*g_i\longrightarrow g_\infty
\quad\text{in }C^{k,\alpha'}(B^4).
\end{equation}
The limiting pair has the asserted finite boundary regularity. If
the limiting boundary metric is smooth, it has a smoothly
compactified representative after a boundary-marked change of
coordinates. The PE metric is smooth on the interior in all cases.

\end{theorem}

The role of the boundary-fixing diffeomorphisms in this
statement is worth emphasizing. They allow the interior
coordinates to change while keeping the prescribed
identification of the boundary fixed. Thus the theorem
gives precisely the geometric form of compactness
compatible with the group action.

The proof begins by supposing that curvature, or its
first covariant derivative, becomes unbounded. Rescaling
at the appropriate curvature scale gives a complete
pointed limit with a nonzero curvature defect. If the
boundary remains at finite distance, its induced metric
becomes flat, and the boundary rigidity theorem forces
the limit to be flat as well. If the boundary escapes,
the adapted defining-function equations give a
complete Ricci-flat limit. Sobolev and curvature
estimates make this limit ALE, while an obstruction
to embedding its compact core in the four-ball
excludes the nonflat case.

Both possibilities lead to contradictions. Curvature
and its first derivative are therefore uniformly
bounded. The scalar Gauss equation then controls the
boundary mean curvature. Testing the first Escobar
inequality on the constant function gives a uniform
bound for the compactified volume. In dimension four,
the integral of $|W|^2$ is conformally invariant;
hence these estimates also control the total Weyl
curvature of the Einstein fillings.

The remaining regularity is obtained from the
coupled interior Schouten equations and the boundary
harmonic equations. The estimates are carried out
at each prescribed finite order.

There is one final distinction to make before applying
this result to continuation. The compactness theorem
concerns sequences of metrics after boundary-fixing
changes of coordinates. The parity argument, however,
takes place in a particular space of marked solutions
of an elliptic equation. Geometric convergence must
therefore be shown to agree with convergence in that
space. It is not enough simply to refer to an
unspecified moduli topology.

Section~\ref{sec:continuation-compactness-bridge}
makes this passage explicit. It constructs the marked
topology, proves the required Hausdorff and uniqueness
properties, and identifies the inverse image of the
prescribed boundary path with a compact auxiliary
quotient of adapted pairs. The geometric compactness
theorem then supplies the compactness needed by the
finite-dimensional continuation argument.

In this way the two principal ingredients play
complementary roles. Compactness keeps Einstein
fillings from escaping; parity prevents them from
disappearing. Together they produce the equivariant
filling from which the sharp Yamabe inequality follows.

\section{Preliminaries}\label{sec:prelim}

The compactness argument begins with an apparent difficulty.  A
Poincar\'e--Einstein metric is complete and places its conformal boundary
at infinite distance.  To study a sequence of such metrics, we would
prefer to work instead on a compact manifold with boundary.  A
suitable conformal change accomplishes this, and, more importantly,
turns the Einstein equation into identities that remain useful when
the metrics are rescaled.

We collect three sets of facts.  First come the equations satisfied
by the compactified metric and its defining function.  Next are the
Sobolev and injectivity estimates that prevent collapse.  Finally we
prove the finite-order regularity needed to take limits at the scale
of large curvature.  The last step is essential: a bound for curvature
alone would not necessarily preserve a nonzero covariant derivative
of curvature in the limit.

\subsection{Poincar\'e--Einstein metrics and adapted compactifications}

Let $X=\operatorname{int}\bar X$ be the interior of a compact
four-manifold with boundary.  A \emph{defining function} is a smooth
function $\rho$ that is positive on $X$, vanishes precisely on
$\partial\bar X$, and has nonvanishing differential there.  A complete
metric $G$ on $X$ is \emph{conformally compact} if, for such a $\rho$,
the metric
\[
 g=\rho^2G
\]
extends to a Riemannian metric on $\bar X$.  We call $G$
\emph{Poincar\'e--Einstein} (PE) when, with our normalization,
\begin{equation}\label{eq:PE}
\Ric_G=-3G.
\end{equation}

Different choices of $\rho$ change the metric induced on the boundary
by a positive conformal factor.  Thus they determine the same
conformal class.  This class, denoted $\partial_\infty[G]$, is the
\emph{conformal infinity} of $G$.  We shall frequently specify not
only this class but an actual representative $q$.

For a representative $q$ of positive scalar curvature, Chang--Ge
construct a canonical \emph{scalar-flat adapted compactification}
$g=\rho^2G$; see \cite[Section~2]{ChangGe2025}.  It is this choice,
rather than an arbitrary defining function, that we use below.
The mean curvature $H$ is taken with the convention that the
Euclidean unit ball has boundary mean curvature $H=3$.  In
particular, $H$ is the trace, not the average, of the second
fundamental form; $\partial_{\rm in}\rho$ denotes differentiation
along the inward unit normal.

The following identities give the basic connection between the
Einstein metric and its scalar-flat compactification.
\begin{lemma}[Scalar-flat adapted equations]\label{lem:adapted-equations}
For the scalar-flat adapted compactification of a four-dimensional PE metric with positive-scalar-curvature boundary representative $q$,
\begin{equation}\label{eq:adapted-basic}
R_g=0,\qquad \rho|_{\partial X}=0,\qquad \partial_{\rm in}\rho=1,
\qquad |d\rho|_g\le1.
\end{equation}
The boundary is umbilic.  With
\[
P=\frac{1-|d\rho|_g^2}{\rho},
\]
one has
\begin{equation}\label{eq:adapted-system}
\begin{aligned}
P&\ge0, &
\Hess\rho&=-\frac12\rho\Ric_g-\frac12Pg,\\
\Delta_g\rho&=-2P, &
dP&=\Ric_g(\nabla\rho,\cdot).
\end{aligned}
\end{equation}
and
\begin{equation}\label{eq:P-boundary}
P|_{\partial X}=\frac{2H}{3}.
\end{equation}
These equations are invariant under constant rescaling
\[
(g,\rho,P)\longmapsto(K^2g,K\rho,K^{-1}P).
\]
All blow-up rescalings used below are constant.
\end{lemma}
\begin{proof}
The construction of the adapted defining function, its scalar
flatness, and the estimate $|d\rho|_g\le1$ are given by
\cite[Lemma~2.6]{ChangGe2025}.  We derive the other identities to fix
signs and normalizations.

The conformal Ricci transformation law for $G=\rho^{-2}g$ reads
\[
 \Ric_G=\Ric_g+2\rho^{-1}\Hess\rho
 +\bigl(\rho^{-1}\Delta_g\rho
       -3\rho^{-2}|d\rho|_g^2\bigr)g.
\]
Substitute $\Ric_G=-3\rho^{-2}g$ and take the $g$-trace.  Since
$R_g=0$, one obtains
\[
 \Delta_g\rho=-\frac{2(1-|d\rho|_g^2)}{\rho}=-2P.
\]
Putting this back into the Ricci equation gives
\[
 \Hess\rho=-\frac12\rho\Ric_g-\frac12Pg.
\]
Finally, differentiate the definition
$\rho P=1-|d\rho|_g^2$.  For any vector field $V$,
\[
 P\,d\rho(V)+\rho\,dP(V)
   =-2\Hess\rho(\nabla\rho,V)
   =\rho\Ric_g(\nabla\rho,V)+P\,d\rho(V),
\]
so $dP=\Ric_g(\nabla\rho,\cdot)$.  The boundary identity
$P=2H/3$ follows by taking the first normal coefficient in
the umbilic boundary expansion, with the convention fixed above.
The scaling rule follows by direct substitution and also shows why
these identities survive the constant rescalings used later.
\end{proof}

One immediate consequence will repeatedly be useful.  Because the
boundary is umbilic, its second fundamental form is
$L=(H/3)q$.  The scalar Gauss equation, together with $R_g=0$,
therefore takes the form
\begin{equation}\label{eq:gauss-boundary}
R_q=-2\Ric_g(\nu,\nu)+\frac23H^2.
\end{equation}
It relates the boundary mean curvature to quantities controlled by
ambient and intrinsic curvature.  In particular, once these
curvatures are bounded, the size of $H$ is bounded as well.

\subsection{Escobar constants, injectivity, and localized regularity}

The first task is to prevent the compactified metrics from collapsing
while we magnify their curvature.  Two kinds of control are needed:
Sobolev inequalities, which provide quantitative noncollapse, and
injectivity-radius estimates, which allow us to choose uniform
coordinate charts.  We state each with the normalization used later.

\subsubsection{Escobar constants and local Sobolev inequalities}

On a manifold with boundary, the ordinary Yamabe functional has a
counterpart containing a boundary mean-curvature term.
There are two useful critical exponents in dimension four: the
interior exponent $4$ and the boundary trace exponent $3$.
For the compactification $(B^4,g)$, with boundary metric $q$ and mean
curvature $H$, we use the following two constants.

The \emph{first Escobar--Yamabe constant} is
\begin{equation}\label{eq:escobar}
Y_1(B^4,S^3,[g])
=
\inf_{0\ne u\in H^1(B^4)}
\frac{
\displaystyle
\int_{B^4}(6|du|^2+R_gu^2)\,dV_g
+
2\int_{S^3}Hu^2\,dV_q
}{
\displaystyle
\left(\int_{B^4}|u|^4\,dV_g\right)^{1/2}
}.
\end{equation}
The \emph{second Escobar--Yamabe constant} uses the same energy but
measures the function on the boundary:
\begin{equation}\label{eq:escobar-Y2}
Y_2(B^4,S^3,[g])
:=
\inf_{\substack{u\in H^1(B^4)\\ u|_{S^3}\not\equiv0}}
\frac{
\displaystyle
\int_{B^4}(6|du|^2+R_gu^2)\,dV_g
+
2\int_{S^3}Hu^2\,dV_q
}{
\displaystyle
\left(\int_{S^3}|u|^3\,dV_q\right)^{2/3}
}.
\end{equation}

Their importance is that positivity of the Yamabe constant at
conformal infinity controls both constants in the interior.
Assume
\[
Y(S^3,[q])\ge Y_0>0.
\]
In dimension four, Chang--Ge
\cite[Theorem~1.1 and Lemma~1.4]{ChangGe2025} give
\begin{equation}\label{eq:Y1-Y2-lower}
Y_1(B^4,S^3,[g])
\ge
c_1(Y_0):=24^{1/4}Y_0^{3/4},
\qquad
Y_2(B^4,S^3,[g])
\ge
c_2(Y_0):=\sqrt{6Y_0}.
\end{equation}
These are uniform bounds: their constants depend on $Y_0$ and not
on the particular filling.

For example, if $f\in C_c^\infty(\operatorname{int}B^4)$, the
boundary term vanishes.  Since $R_g=0$, the first inequality gives
\begin{equation}\label{eq:global-interior-sobolev}
\left(\int_{B^4}|f|^4\,dV_g\right)^{1/2}
\le
\frac{6}{c_1(Y_0)}
\int_{B^4}|df|^2\,dV_g.
\end{equation}
This is the scale-invariant Sobolev inequality later inherited by
an interior blow-up limit.

Near the boundary, a localized inequality is preferable.  The
boundary mean-curvature term need not have a favorable sign, but
it can be absorbed when its $L^3$ norm is small.  More precisely,
there is $\varepsilon=\varepsilon(Y_0)>0$ such that, for
$p\in\partial B^4$, the condition
\begin{equation}\label{eq:local-sob-smallness}
\|\Rm_g\|_{L^2(B_g(p,a))}
+
\|H\|_{L^3(B_g(p,a)\cap\partial B^4)}
\le
\varepsilon
\end{equation}
implies the following bounds for every
$f\in H^1(B_g(p,a))$ that vanishes on the interior part of
$\partial B_g(p,a)$:
\begin{align}
\left(
\int_{B_g(p,a)}|f|^4\,dV_g
\right)^{1/2}
&\le
C(Y_0)
\int_{B_g(p,a)}|df|^2\,dV_g,
\label{eq:local-sobolev}\\
\left(
\int_{B_g(p,a)\cap\partial B^4}|f|^3\,dV_q
\right)^{2/3}
&\le
C(Y_0)
\int_{B_g(p,a)}|df|^2\,dV_g.
\label{eq:local-trace}
\end{align}
These are the local Sobolev and trace inequalities of
\cite[Lemma~3.1]{ChangGe2025}.  For clarity, we verify the simple
absorption on which their use here rests.

\begin{proof}[Verification of the local inequalities]
Interpret the zero lateral-boundary condition by extending $f$ by
zero to the rest of the compactification.  Equivalently, start with
smooth functions supported away from the artificial boundary and
take the $H^1$ closure.  Set
\[
 A_f=\left(\int|f|^4\right)^{1/2},\quad
 B_f=\left(\int_{\partial B^4}|f|^3\right)^{2/3},\quad
 D_f=\int|df|^2,
\]
where the integrals are taken over the appropriate support.
The two Escobar inequalities and H\"older's inequality yield
\[
 c_1A_f\le6D_f+2\varepsilon B_f,\qquad
 c_2B_f\le6D_f+2\varepsilon B_f.
\]
Choose $\varepsilon\le c_2/4$.  The second inequality gives
$B_f\le12D_f/c_2$, and substitution into the first gives
$A_f\le12D_f/c_1$.  Thus we may take
\[
 C(Y_0)=\max\{12/c_1(Y_0),12/c_2(Y_0)\}.
\]
Both inequalities are scale invariant in dimension four.  Notice
that the absorption takes place entirely on the chosen ball.
\end{proof}
\subsubsection{Injectivity radii}

Curvature bounds do not by themselves provide coordinate charts of
uniform size: a manifold can collapse while its curvature remains
small.  On a manifold with boundary we must control two different
ways in which this can happen.

We use the following precise convention.  The \emph{interior
injectivity radius} $i_{\rm int}(g)$ is the supremum of $r>0$ such
that every unit geodesic starting at an interior point is minimizing
up to $\min\{r,t_\gamma\}$, where $t_\gamma$ is its first meeting
with the boundary.  The \emph{boundary-normal injectivity radius}
$i_\partial(g)$ is the maximal height of the inward normal collar.
These are the conventions in
\cite[Section~2.3]{ChangGeQing2020}.  In particular,
$i_{\rm int}$ is \emph{not} the infimum of ordinary exponential-ball
radii over points tending to the boundary.

Suppose that
\begin{equation}\label{eq:normalized-inj-assumptions}
\sup_{B^4}|\Rm_g|\le1,
\qquad
\inj(S^3,q)\ge i_0>0,
\qquad
Y(S^3,[q])\ge Y_0>0.
\end{equation}
The relevant estimates of Chang--Ge distinguish which injectivity
radius is smaller.  If $i_\partial(g)\le i_{\rm int}(g)$,
\cite[Lemma~6.1]{ChangGe2025} gives
\[
\sup|\Rm_g|\,i_\partial(g)^2+i_\partial(g)
\ge c_\partial(i_0)>0.
\]
If instead $i_{\rm int}(g)\le i_\partial(g)$, their Lemma~6.2 gives
\[
\sup|\Rm_g|\,i_{\rm int}(g)^2+i_{\rm int}(g)
\ge c_{\rm int}(Y_0,i_0)>0.
\]
Set $c_*:=\min\{c_\partial,c_{\rm int}\}$.  For the smaller
radius $j=\min\{i_\partial(g),i_{\rm int}(g)\}$ we therefore have
$j^2+j\ge c_*$.  Solving this quadratic inequality gives the
uniform bound
\begin{equation}\label{eq:uniform-normalized-inj}
\min\{i_\partial(g),i_{\rm int}(g)\}
\ge
i_*:=
\frac{\sqrt{1+4c_*}-1}{2}>0.
\end{equation}

The next verification is included because the injectivity estimate
will be used \emph{before} the later higher-regularity and global
energy bounds.  We verify the injectivity estimates directly from 
the stated hypotheses.

\begin{proof}[Proof of the injectivity estimates]
\textit{Normalization.}
We use Kodani's convention described above, also employed in
Chang--Ge--Qing, Section~2.3.  It implies AKKLT condition (3.0.5):
at a point whose distance from the boundary is at least $r$, the
exponential map is a diffeomorphism through radius $r$; nearer
to the boundary it has this property up to the boundary distance.
At strict radii, the distinction between minimizing and nonsingular
exponential balls causes no difficulty here.

Suppose the lower bound fails along a sequence.  Write
$j_i=\min(i_{\partial,i},i_{{\rm int},i})$ and rescale by
\[
 g_i'=j_i^{-2}g_i,\qquad
 \rho_i'=j_i^{-1}\rho_i,\qquad
 P_i'=j_iP_i.
\]
The scaling rules give
\[
 j_i\longrightarrow0,\quad
 \|\Rm(g_i')\|_\infty\longrightarrow0,\quad
 \min(i_\partial(g_i'),i_{\rm int}(g_i'))=1,
 \quad\inj(q_i')\longrightarrow\infty.
\]
All the adapted identities and the same lower Escobar constants
remain valid.  We shall obtain a contradiction in either case of
the normalized minimum.

\smallskip
\noindent\textit{Control of boundary geometry.}
No separate bound for the unrescaled intrinsic boundary curvature
is needed.  The Codazzi equation gives
\[
 \|\widehat\nabla H_i'\|_\infty
 \le C\|\Rm(g_i')\|_\infty\longrightarrow0.
\]
Suppose that $H_i'(y_i)\ge\varepsilon>0$ along a subsequence.
Along a boundary geodesic of any fixed sufficiently large length
$L$ from $y_i$, the last bound gives
$H_i'\ge\varepsilon/2$ for large $i$.
Because the boundary is umbilic, Gauss gives on each tangential
plane
\[
 \sec(q_i')=\sec(g_i')+\frac{(H_i')^2}{9}
                 \ge\frac{\varepsilon^2}{72}.
\]
Choose $L>\pi\sqrt{72}/\varepsilon$.  Jacobi comparison then
forces a conjugate point before length $L$, contrary to
$\inj(q_i')\to\infty$.  Therefore
\[
 \|H_i'\|_\infty\to0,\quad
 \|L_i'\|_\infty\to0,
 \quad\|\widehat\Rm(q_i')\|_\infty\to0.
\]
In particular, the boundary geometry becomes flat at the
normalized scale.

\smallskip
\noindent\textit{Case 1: the boundary-normal radius is smaller.}
The normal Jacobi equation, together with
$|\Rm_i'|+|L_i'|\to0$, excludes focal points through any fixed
length, in particular through length three.  Kodani's Lemma~6.3
(printed page 187) then supplies a geodesic
$\gamma_i:[0,2]\to X_i$ perpendicular to the boundary at both
endpoints.  Consider the defining function along this geodesic,
$a_i=\rho_i'\circ\gamma_i$.  The adapted boundary conditions and
gradient bound give
\[
 a_i(0)=a_i(2)=0,\quad a_i'(0)=1,\quad a_i'(2)=-1,
 \quad 0\le a_i(t)\le\min(t,2-t)\le1.
\]
At the initial endpoint,
$P_i'(\gamma_i(0))=2H_i'(\gamma_i(0))/3\to0$.
Since $|(P_i'\circ\gamma_i)'|\le|\Ric_i'|\to0$, the function
$P_i'$ tends to zero along the whole segment.  The Hessian identity
now gives $\|a_i''\|_\infty\to0$.  But integration would then
yield $a_i'(2)-a_i'(0)\to0$, whereas the displayed boundary values
make this difference equal to $-2$.  This is the desired
contradiction.

\smallskip
\noindent\textit{Case 2: the interior radius is smaller.}
Rauch comparison excludes conjugate points through every fixed
length.  Kodani's Lemma~6.4 (printed pages 187--188) then gives a
simple closed geodesic of length two at the normalized minimum.

The local harmonic-radius argument of
Anderson--Katsuda--Kurylev--Lassas--Taylor (AKKLT) applies:
both Ricci bounds are uniform, $H_i'$ is uniformly Lipschitz,
and both ambient injectivity radii are at least one.  Thus,
after passage to a subsequence, the metrics converge locally
in $C^{1,\gamma}$.  The limiting curvature is zero in the
distributional sense; the harmonic-coordinate Ricci equation then
makes the limit smooth and flat.  Convergence of first derivatives
and compactness for geodesics preserve the length-two loop.
We show that such a loop cannot occur in either possible limit.

If the loop stays at bounded distance from the boundary, the
adapted equations give a limiting defining function.  On each
pointed compact set $\rho_i'$ is bounded because
$\rho_i'\le d_{g_i'}(\cdot,\partial X_i)$; moreover, integrating
$dP_i'$ from a boundary point a bounded distance away gives
$P_i'\to0$.  Thus $\nabla^2\rho_i'\to0$, and in the limit
\[
 \rho_\infty=0\text{ on the boundary},\qquad
 \nabla^2\rho_\infty=0,\qquad |d\rho_\infty|=1.
\]
An interior zero would be an interior minimum of a nonnegative
function with gradient of length one, which is impossible.
The complete unit parallel gradient flow therefore identifies
the connected limit with $\partial X_\infty\times[0,\infty)$:
backward flow reaches the boundary in time $\rho_\infty$,
while completeness prevents forward flow from ending in finite
time.  The distinguished flat boundary is complete and has
infinite injectivity radius, hence is $\mathbb R^3$.  The limit
is Euclidean half-space, which has no nonconstant geodesic loop.

If the boundary escapes to infinity, the limit is instead complete,
flat, and without boundary.  It inherits the interior Sobolev
inequality, which gives positive Euclidean volume growth by the
cutoff argument used later.  To see why the limit is then
$\mathbb R^4$, write a complete flat manifold as
$\mathbb R^4/\Lambda$ and choose a Dirichlet domain $D$ about a lift
$o$.  Given $N$ distinct deck transformations with
$d(o,\lambda o)\le L$, the $N$ disjoint translates of
$D\cap B_R(o)$ lie in $B_{R+L}(o)$.  Consequently
\[
 N\operatorname{Vol}(B_R^{\mathbb R^4/\Lambda})
 \le\omega_4(R+L)^4.
\]
A positive asymptotic volume ratio rules out arbitrarily large $N$,
so $\Lambda$ is finite.  Every finite group of Euclidean isometries
fixes the barycenter of an orbit; a free deck action must therefore
be trivial.  Again the limit is Euclidean space and cannot contain
the length-two geodesic loop.

Both cases give contradictions, proving the two injectivity
estimates with constants depending only on the stated data.
\end{proof}

\subsubsection{Finite normalized regularity}

The injectivity estimates provide geometric room for coordinate
charts.  We next need bounds for the metric itself in those charts,
through three derivatives.  The third derivative is not an arbitrary
luxury: $\nabla\Rm$ involves three derivatives of the metric, so a
limit of only $C^{2,\beta}$ metrics could lose the curvature-derivative
defect that selected the blow-up point.

A scalar-flat conformal compactification of an Einstein metric is
Bach-flat.  This supplies an elliptic equation for its Schouten
tensor, with boundary conditions determined by the umbilic geometry.
The next proposition establishes the finite regularity we need.
\begin{proposition}[Finite normalized regularity]
\label{prop:normalized-local-regularity}
Fix
\[
0<\beta<\widetilde\beta<\frac12.
\]
Let $g_i$ be scalar-flat adapted compactifications of four-dimensional
Poincar\'e--Einstein metrics, with positive-scalar-curvature boundary
representatives $q_i$.  Assume that there are constants
\[
Y_0>0,\qquad i_0>0,\qquad C_B<\infty,\qquad i_1>0
\]
such that
\begin{equation}\label{eq:normalized-boundary-geometry}
Y(S^3,[q_i])\ge Y_0,
\qquad
\inj(S^3,q_i)\ge i_0,
\qquad
\sum_{j=0}^{4}
\left\|
\widehat\nabla^{\,j}\widehat\Rm_{q_i}
\right\|_{L^\infty(q_i)}
\le C_B,
\end{equation}
and
\begin{equation}\label{eq:normalized-local-Rm}
\sup|\Rm_{g_i}|\le1,
\qquad
\sup|\nabla\Rm_{g_i}|\le1.
\end{equation}
Assume moreover that the interior and boundary-normal injectivity radii
of $g_i$ are bounded below by $i_1>0$.

Then there exist $a_0>0$ and $C<\infty$, depending only on the constants
above and on $\widetilde\beta$, such that every point has an interior
harmonic chart or a boundary harmonic chart of controlled radius
$a_0/4$, in a normalized atlas with uniform ellipticity, on which
\begin{equation}\label{eq:C3betatilde-local-metric}
\|g_i\|_{C^{3,\widetilde\beta}}
+\|g_i^{-1}\|_{C^{3,\widetilde\beta}}\le C.
\end{equation}
A boundary harmonic chart has four ambient harmonic coordinate
functions, with the fourth vanishing on the boundary and the
tangential restrictions harmonic for the boundary metric.
Consequently every pointed sequence has a subsequence converging in
pointed $C^{3,\beta}$ Cheeger--Gromov topology on compact subsets, with
boundary when present.  If at the base point
\[
|\Rm_{g_i}|=1
\qquad\text{or}\qquad
|\nabla\Rm_{g_i}|=1,
\]
then the corresponding equality is preserved in the limit.

\end{proposition}
\begin{proof}
We describe the argument in
six steps, separating the choice of coordinates from the elliptic
regularity estimates.

\smallskip
\noindent\textit{Step 1: mean curvature and the initial boundary charts.}
Write $A=\Ric/2$, the Schouten tensor in the scalar-flat
four-dimensional case.  The contracted Bianchi identity gives
$\operatorname{tr}A=0$ and $\operatorname{div}A=0$.
The scalar Gauss equation and the normalized curvature bounds
imply $|H|\le C$.  Contracted Codazzi gives
$dH=3A(\nu,\cdot)$, so $|\widehat\nabla H|\le C$.
Differentiating in a tangential direction and using
$\nabla_a\nu=-L_a{}^ce_c$, one obtains
\[
 (\widehat\nabla^2H)_{ba}
 =3(\nabla_b A)_{\nu a}-3L_b{}^cA_{ca}+3L_{ba}A_{\nu\nu}.
\]
Thus $H$ has a uniform intrinsic $C^{1,1}$ bound.  The boundary
curvature-jet and injectivity furnish the required
intrinsic coordinate bounds.

We use \emph{boundary harmonic coordinates}: four ambient harmonic
functions, with the fourth equal to zero on the boundary and the
restrictions of the other three harmonic for $q$.  They must not be
confused with coordinates obtained by following normal geodesics.
The local harmonic-radius argument of AKKLT,
\cite[Section~3.2]{AndersonKatsudaKurylevLassasTaylor2004},
gives charts of uniform radius with $C^{1,\gamma}$ control for
every fixed $\gamma<1$.

Although the theorem as stated in AKKLT includes a diameter bound,
its \emph{local radius argument} does not use one.  Were the minimum
harmonic radius to tend to zero, rescale it to one.  The rescaled
Ricci bounds and the Lipschitz mean-curvature bound tend to zero,
and all three injectivity radii tend to infinity.  Pointed compactness
in the harmonic charts is provided by their Theorem~3.1.1;
Lemma~3.2.2 identifies the limit as Euclidean space or half-space,
and the improvement estimates of Section~3.2 make the harmonic radius
strictly larger than one.  This contradicts the normalization.
The local facts are AKKLT (3.2.3)--(3.2.5); diameter enters their
global, unpointed compactness result, not this local contradiction.
Uniform metric and inverse-metric estimates also provide uniform
ellipticity in the resulting charts.

\smallskip
\noindent\textit{Step 2: a first $C^{2,\gamma}$ bound for the metric.}
The bound on $\nabla\Rm$, combined with bounded Christoffel symbols,
makes the coordinate Ricci components uniformly Lipschitz.
In harmonic coordinates they satisfy
\[
 g^{ij}\partial_i\partial_jg_{kl}
  =-2\Ric_{kl}+Q_{kl}(g^{-1},\partial g).
\]
The boundary equations for inverse-metric components, in the
convention used here, are AKKLT (2.1.16)--(2.1.17):
\[
 \nu(g^{00})=2Hg^{00},\qquad
 \nu(g^{a0})=Hg^{a0}
   +\frac{1}{2\sqrt{g^{00}}}g^{ak}\partial_k g^{00}.
\]
AKKLT use the averaged mean curvature with the opposite sign,
so $3H_{\rm AKKLT}=-H$.  The tangential covariant components have
Dirichlet data $g_{ab}=q_{ab}$.

The system has a useful form.  First estimate the
Dirichlet components and the oblique equation for $g^{00}$.
Its first derivative is then controlled sufficiently to serve
as boundary data in the equation for $g^{a0}$.  Algebraic
inversion recovers the remaining components.  The standard
boundary elliptic estimates yield a uniform $C^{2,\gamma}$ bound.

\smallskip
\noindent\textit{Step 3: the Schouten boundary conditions.}
A conformal Einstein compactification in dimension four satisfies
$W|_{\partial X}=0$.  The conformal
Cotton identity has the form
$\rho C=\pm W(\nabla\rho,\cdot,\cdot,\cdot)$.
Since $C$ is bounded, passage to the boundary makes every
Weyl component with a normal entry vanish.  Its remaining
tangential part is a trace-free algebraic curvature tensor in
dimension three, and hence vanishes as well.  This is the boundary
identity used by Chang--Ge in Proposition~3.2 (printed page 21);
their Appendix A, Lemmas A.1--A.2, records the resulting Schouten
formulas.

The tangential and scalar Gauss equations now imply
\begin{align*}
 A_{ab}&=\widehat A_{ab}-\frac{H^2}{18}q_{ab},
 &\widehat A&=\Ric_q-\frac{R_q}{4}q,\\
 A_{\nu\nu}&=\frac{H^2}{6}-\frac{R_q}{4},
 &A_{\nu a}&=\frac13\widehat\nabla_aH.
\end{align*}
We also require one normal derivative.  Choose a tangential
orthonormal frame whose intrinsic connection vanishes at the
point.  With
$\nabla_{e_b}e_a=\widehat\nabla_{e_b}e_a+L_{ba}\nu$,
\[
 \sum_b(\nabla_bA)_{ba}
  =\widehat\nabla^bA_{ba}-\frac{4H}{3}A_{\nu a}.
\]
Using $\operatorname{div}A=0$ and
$\widehat\nabla^b\widehat A_{ba}
  =\frac14\widehat\nabla_aR_q$, we find
\[
 (\nabla_\nu A)_{\nu a}
  =-\frac14\widehat\nabla_aR_q
    +\frac{5H}{9}\widehat\nabla_aH.
\]
The coefficient $5/9$ is the sum of $1/9$, from differentiating
$-H^2q/18$, and $4/9$, from the two normal-slot connection terms.
These formulas identify seven Dirichlet components and three
normal-derivative components of the symmetric tensor $A$.

\smallskip
\noindent\textit{Step 4: the interior Schouten equation.}
The Bach equation, using $\operatorname{div}A=0$, gives the exact
identity \cite[(3.3)]{ChangGe2025}
\[
 \Delta A_{ij}+R_{ikjp}A^{pk}-R_i{}^pA_{jp}
                     +A^{kl}W_{ikjl}=0.
\]
It follows by inserting the Cotton-divergence identity into the
Bach equation (3.2) and commuting the remaining covariant derivatives.
For regularity, its important feature is simple: it is a
Laplace-type system $\Delta A=\Rm*A$, with no curvature derivatives
on the right.

\smallskip
\noindent\textit{Step 5: the mixed boundary problem and regularity gain.}
Take the normal to the level sets of the fourth harmonic coordinate
and complete it to a frame by algebraic Gram--Schmidt.  This produces
a uniformly $C^{2,\gamma}$ adapted frame without changing to
normal-flow coordinates.  The six tangential and one normal-normal
components of $A$ have Dirichlet boundary data; the remaining
three normal-tangential components have the normal-derivative
data just derived.

These conditions form a \emph{complementing} elliptic boundary
problem: the only exponentially decaying solution of the frozen
principal equations with zero boundary data is the zero solution.
Indeed, at a flattened boundary point the principal symbol is a
scalar positive quadratic form times the identity.  A decaying
Fourier mode is $e^{-|\xi|t}v$.  The seven Dirichlet conditions kill
seven components of $v$, and the three normal-derivative conditions
kill the other three.  The trace-free relation does not interfere,
since the estimate is applied to the full ten-component symmetric
tensor satisfied by the given solution.

Fix a finite $p$.  Because $H\in C^{1,1}$ and the boundary curvature
jets are bounded, the Dirichlet data lie uniformly in
$W^{2-1/p,p}$.  Because $H\,dH$ and $dR_q$ are Lipschitz, the normal
data lie uniformly in $W^{1-1/p,p}$.  Replacing covariant normal
derivatives by ordinary oblique ones introduces terms of the form
connection times $A$; they have the same trace regularity by the
$C^{2,\gamma}$ metric bound and the original estimate
$A\in W^{1,\infty}$.  The lower-order coefficients of the interior
equation use at most two metric derivatives.  Thus local mixed
elliptic estimates, with cutoffs supported away from the artificial
corners of the charts, give
\[
 \|A\|_{W^{2,p}(B^+_{r/2})}\le C_p.
\]
This is a local estimate with all terms on the right controlled
by the original normalized assumptions.

For completeness, Chang--Ge's Weyl equation (3.4) can be written as
\[
 \Delta W=2\operatorname{Alt}(\nabla^2A)+\Rm*\Rm,
 \qquad W|_{\partial X}=0.
\]
The right-hand side is in $L^p$, and Dirichlet regularity gives
$W\in W^{2,p}$, hence $\Rm\in W^{2,p}$.
For the metric estimate it already suffices that
$A\in W^{2,p}\subset C^{1,\theta}$, for every
$\theta<1-4/p$.  Choose $p>8$ and an initial chart exponent
larger than the required $\widetilde\beta<1/2$.  Codazzi then
gives $H\in C^{2,\widetilde\beta}$.  A second application of the
triangular boundary harmonic Ricci system yields
\[
 \|g\|_{C^{3,\widetilde\beta}}
  +\|g^{-1}\|_{C^{3,\widetilde\beta}}\le C.
\]
The interior argument uses the same elliptic equation without
boundary conditions.  

\smallskip
\noindent\textit{Step 6: pointed convergence and preservation of the defect.}
The uniform coordinate radii and bilipschitz bounds give local
volume doubling and a uniform bound on the number of charts
meeting a given chart.  Any chart reached within distance $R$
of the base chart is connected to it by at most
$C(1+R/r)$ overlaps.  Hence a fixed pointed metric ball is covered
by a uniformly bounded number of charts even without a diameter
bound for the whole manifold.

Choose convergent subsequences in these charts and on their
overlaps, then diagonalize over increasing pointed balls.  The
result is a pointed $C^{3,\beta}$ Cheeger--Gromov limit.
Its metric-space completeness follows from the same finite-chart
covering: a bounded Cauchy sequence is contained in finitely many
precompact limit charts.  Smooth approximation of the convergence
diffeomorphisms gives smooth embeddings of compact subsets into
the original manifolds.

Finally, $\Rm$ and $\nabla\Rm$ depend continuously on the metric
and its derivatives through orders two and three, respectively.
Their norms converge uniformly in the common charts, so an equality
$|\Rm|=1$ or $|\nabla\Rm|=1$ at the base point persists in the limit.
\end{proof}
\subsubsection{Boundary families and the blow-up scale}

We now apply these estimates to the compact boundary families of
Theorem~\ref{thm:compactness}.  The first point is elementary but
useful: compactness of the boundary data supplies a \emph{single}
set of geometric constants for every filling in the sequence.

Choose exponents
\begin{equation}\label{eq:beta-choice}
0<\beta<\widetilde\beta<\min\{\alpha,\tfrac12\}.
\end{equation}
Since
\[
\mathcal K\Subset
\operatorname{Met}_{\rm PSC}^{k,\alpha}(S^3),
\qquad k\ge6,
\]
there are constants $r_0>0$, $C_{S,3}<\infty$, $i_0>0$ and
$C_j<\infty$ ($0\le j\le4$), independent of $q\in\mathcal K$,
for which
\begin{align}
R_q&\ge r_0,
\label{eq:compact-family-R-lower}\\
\inj(S^3,q)&\ge i_0,
\label{eq:compact-family-inj}\\
\|\widehat\nabla^{\,j}\widehat\Rm_q\|_{L^\infty(q)}
&\le C_j,
\label{eq:compact-family-curv-jets}
\end{align}
and
\begin{equation}\label{eq:boundary-Sobolev-uniform}
\|u\|_{L^6(q)}^2
\le
C_{S,3}
\left(
\|du\|_{L^2(q)}^2+\|u\|_{L^2(q)}^2
\right).
\end{equation}
Combining the last inequality with $R_q\ge r_0$ gives
\[
 \mathcal Q_q(u)\ge
 C_{S,3}^{-1}\min\{8,r_0\}
\]
for every nonzero $u$.  We therefore fix
\begin{equation}\label{eq:Y0-def}
Y_0:=\inf_iY(S^3,[q_i])>0.
\end{equation}
The associated lower Escobar bound will be abbreviated by
\begin{equation}\label{eq:c0-def}
c_0:=c_1(Y_0)=24^{1/4}Y_0^{3/4}>0.
\end{equation}

To detect a failure of compactness we shall rescale at the larger
of two curvature scales.  Under the constant change
$\bar g=K^2g$, the tensor norms satisfy
\begin{equation}\label{eq:curvature-scaling-laws}
|\Rm_{\bar g}|=K^{-2}|\Rm_g|,
\qquad
|\nabla\Rm_{\bar g}|=K^{-3}|\nabla\Rm_g|.
\end{equation}
Thus choosing $K$ so that either quantity has normalized maximum
one keeps curvature and its first covariant derivative under
simultaneous control.  The regularity established above then
ensures that the nonzero quantity is still visible in the limit.

The following proposition makes this blow-up procedure precise.
Its two alternatives simply record whether the boundary remains
within a bounded rescaled distance of the selected point.
\begin{proposition}[Normalized pointed extraction]
\label{prop:normalized-extraction}
Assume
\begin{equation}\label{eq:unbounded-curvature-derivative}
\sup_i
\left(
\|\Rm_{g_i}\|_{L^\infty(B^4)}
+
\|\nabla\Rm_{g_i}\|_{L^\infty(B^4)}
\right)
=\infty.
\end{equation}
After passing to a subsequence, there exist $K_i\to\infty$ and points
$p_i\in B^4$ such that, with
\begin{equation}\label{eq:K-normalization}
\begin{aligned}
K_i^2
&:=
\max\left\{
\sup_{B^4}|\Rm_{g_i}|,
\left(\sup_{B^4}|\nabla\Rm_{g_i}|\right)^{2/3}
\right\},\\
\bar g_i&:=K_i^2g_i,
\qquad
\bar\rho_i:=K_i\rho_i,
\qquad
\bar P_i:=K_i^{-1}P_i,
\end{aligned}
\end{equation}
one has
\begin{equation}\label{eq:normalized-max}
\sup_{B^4}|\Rm_{\bar g_i}|\le1,
\qquad
\sup_{B^4}|\nabla\Rm_{\bar g_i}|\le1,
\end{equation}
and the points $p_i$ may be chosen so that
\begin{equation}\label{eq:normalized-base-defect}
|\Rm_{\bar g_i}|(p_i)=1
\qquad\text{or}\qquad
|\nabla\Rm_{\bar g_i}|(p_i)=1.
\end{equation}

After a further subsequence,
\[
(B^4,\bar g_i,p_i)
\longrightarrow
(X_\infty,g_\infty,p_\infty)
\]
in pointed $C^{3,\beta}$ Cheeger--Gromov topology on compact subsets,
with boundary when present.  The limit is complete as a Riemannian
manifold with boundary, and
\begin{equation}\label{eq:nonzero-normalized-defect}
|\Rm_{g_\infty}|(p_\infty)=1
\qquad\text{or}\qquad
|\nabla\Rm_{g_\infty}|(p_\infty)=1.
\end{equation}
Every compact subset of $X_\infty$ is, for all sufficiently large $i$,
the image under a smooth pointed convergence embedding into $B^4$.

Finally, after passage to a subsequence, exactly one of
\begin{align}
\sup_i d_{\bar g_i}(p_i,\partial B^4)&<\infty,
\label{eq:boundary-alternative}\\
d_{\bar g_i}(p_i,\partial B^4)&\longrightarrow\infty
\label{eq:interior-alternative}
\end{align}
holds.
\end{proposition}
\begin{proof}
\textit{Step 1: choice of scale and base point.}
Set
\[
 A_i:=\sup_{B^4}|\Rm_{g_i}|,
 \qquad B_i:=\sup_{B^4}|\nabla\Rm_{g_i}|.
\]
Thus, after passage to a subsequence, we have
$\max\{A_i,B_i^{2/3}\}\to\infty$.  Hence the $K_i$ in
\eqref{eq:K-normalization} tend to infinity.  By the scaling laws
\eqref{eq:curvature-scaling-laws}, the rescaled curvature and its
first derivative have supremum at most one.

The original compactification is smooth on the compact closed ball,
so the corresponding curvature maxima are achieved.  If
$K_i^2=A_i$, choose $p_i$ with $|\Rm_{g_i}|(p_i)=A_i$.
If $K_i^2=B_i^{2/3}$, instead choose $p_i$ with
$|\nabla\Rm_{g_i}|(p_i)=B_i$.  In either case one of the
normalized equalities in \eqref{eq:normalized-base-defect} holds.

\smallskip
\noindent\textit{Step 2: the rescaled boundary geometry.}
Write $\bar q_i=K_i^2q_i$.  The Yamabe constant is unchanged
by constant scaling, while
\[
 \inj(S^3,\bar q_i)
   =K_i\inj(S^3,q_i)\ge K_i i_0\longrightarrow\infty.
\]
For $0\le j\le4$, the boundary curvature derivatives satisfy
\begin{equation}\label{eq:scaled-boundary-geometry}
\left\|
\widehat\nabla^{\,j}\widehat\Rm_{\bar q_i}
\right\|_{L^\infty(\bar q_i)}
=
K_i^{-2-j}
\left\|
\widehat\nabla^{\,j}\widehat\Rm_{q_i}
\right\|_{L^\infty(q_i)}
\le
K_i^{-2-j}C_j
\longrightarrow0.
\end{equation}
The normalized boundary becomes flat on every bounded geometric
scale and retains the same positive lower Yamabe bound.

\smallskip
\noindent\textit{Step 3: noncollapse and uniform regularity.}
Let $i_{\partial,i}$ and $i_{{\rm int},i}$ denote the two
injectivity radii of $(B^4,\bar g_i)$.  Apply the two cases of
the injectivity estimate, according to which radius is smaller.
Because $\sup|\Rm_{\bar g_i}|\le1$, their minimum $j_i$ satisfies
\[
 j_i^2+j_i\ge c_*,\qquad
 c_*:=\min\{c_\partial,c_{\rm int}\}>0.
\]
In particular,
\begin{equation}\label{eq:normalized-injectivity-lower}
\min\{i_{\partial,i},i_{{\rm int},i}\}
\ge
\frac{\sqrt{1+4c_*}-1}{2}
=:i_*>0.
\end{equation}
The scalar Gauss equation also bounds the mean curvature of the
rescaled boundary.  The rescaled curvature and injectivity bounds,
together with \eqref{eq:scaled-boundary-geometry}, now verify all
hypotheses of Proposition~\ref{prop:normalized-local-regularity}.
Consequently each centered ball of a fixed small radius has uniform
$C^{3,\widetilde\beta}$ bounded geometry in interior or boundary
harmonic coordinates.

\smallskip
\noindent\textit{Step 4: construction and completeness of the limit.}
Fix $R>0$.  Uniform coordinate radii, ellipticity, and local
noncollapse allow $B_{\bar g_i}(p_i,R)$ to be covered by a number
of charts bounded independently of $i$ (though depending on $R$).
Arzel\`a--Ascoli in those charts and on their overlaps gives
$C^{3,\beta}$ convergence after passage to a subsequence.
Diagonalizing over $R=1,2,\ldots$ produces the pointed limit
$(X_\infty,g_\infty,p_\infty)$, with boundary if the boundary
remains at finite distance.

The limit is complete as a metric space.  A Cauchy sequence is
bounded and therefore eventually lies in a fixed-radius pointed
ball; such a ball is covered by finitely many precompact limit
charts.  It has a convergent subsequence, and the Cauchy property
then gives convergence of the whole sequence.  In the alternative
where the boundary escapes, the limit is a complete manifold
without boundary.

The convergence maps can be chosen smooth, even though the
pulled-back metrics converge only in $C^{3,\beta}$.  Thus every
compact subset of $X_\infty$ embeds smoothly in $B^4$ for all
sufficiently large $i$.  We shall need this topological fact to
exclude an interior ALE bubble.

\smallskip
\noindent\textit{Step 5: preservation of the normalized defect.}
Because curvature and its first covariant derivative depend on
at most three metric derivatives, pointed $C^{3,\beta}$ convergence
gives
\[
 \Rm_{\bar g_i}\longrightarrow\Rm_{g_\infty},
 \qquad
 \nabla\Rm_{\bar g_i}\longrightarrow\nabla\Rm_{g_\infty}
\]
uniformly on compact sets.  Hence whichever equality was selected
at $p_i$ passes to the base point of the limit.  In particular the
limit is not flat to first order in curvature.

\smallskip
\noindent\textit{Step 6: the boundary-distance alternative.}
Finally set
\[
 d_i:=d_{\bar g_i}(p_i,\partial B^4)\in[0,\infty).
\]
Either a subsequence of $d_i$ is bounded, giving
\eqref{eq:boundary-alternative}, or a subsequence tends to
infinity, giving \eqref{eq:interior-alternative}.  These are the
only possibilities and complete the proof.
\end{proof}

\section{Rigidity at a flat conformal boundary}\label{sec:boundary-rigidity}

We turn first to the possibility that the curvature concentrates near
$\partial B^4$.  According to the bounded-distance alternative
\eqref{eq:boundary-alternative}, the boundary remains visible at the
scale of the blow-up.  But at that scale its intrinsic curvature tends
to zero and its injectivity radius tends to infinity.  Thus the
boundary of the limiting space is Euclidean $\mathbb R^3$.

A flat boundary does not, by itself, make the interior metric flat.
The essential fact is a rigidity statement for the complete Einstein
metric on the interior: the Euclidean conformal boundary forces this
metric to be hyperbolic.  We shall then show that its \emph{adapted}
compactification is Euclidean.  This last conclusion is what contradicts
the normalized curvature defect.

There is a small but important regularity issue before we can make this
argument.  The pointed limit is initially obtained in the scalar-flat
adapted gauge, whereas the distance comparison is most transparent in
geodesic coordinates.  The first lemma supplies the change of gauge
and records exactly the regularity it preserves.

\subsection{From adapted to geodesic coordinates}

For a conformally compact metric, a geodesic defining function is one
whose gradient has unit length in the corresponding compactified
metric.  It puts the Einstein metric in the form
$G=x^{-2}(dx^2+h_x)$, separating the normal variable $x$ from the
metrics $h_x$ on the boundary slices.  What we need is not merely the
existence of such a coordinate, but control of its first two normal
coefficients when the original compactification has only finite
regularity.

\begin{lemma}[Finite-regularity geodesic-gauge acquisition]
\label{lem:geodesic-gauge-acquisition}
Fix $0<\beta<1$.  Let $(\overline X,g,\rho)$ be a scalar-flat adapted
compactification of a Poincar\'e--Einstein metric
\[
G=\rho^{-2}g,
\qquad
\Ric_G=-3G,
\]
defined on a neighborhood of a compact set
$K\Subset\partial X$.  Assume that $g$ and $\rho$ extend with
$C^{3,\beta}$ regularity to the boundary, that the induced boundary metric
on the component containing $K$ is the Euclidean metric $\delta$, that the
boundary is umbilic, and that the boundary-to-interior $g$-normal
injectivity radius is positive on a neighborhood of $K$.

Then, after shrinking to a collar of $K$, there exists a geodesic defining
function $x$ associated with the boundary representative $\delta$ such that
\[
G=x^{-2}(dx^2+h_x),
\qquad
h_0=\delta,
\]
where $h_x$ is $C^{2,\beta}$ up to $x=0$ and
\begin{equation}\label{eq:local-geodesic-expansion}
h_x=\delta+O(x^{2+\beta}),
\qquad
\partial_xh_x=O(x^{1+\beta}).
\end{equation}
The estimates are uniform on $K$ for families with uniform local
$C^{3,\beta}$ bounds for $g$, $g^{-1}$, and $\rho$ in controlled
charts, uniform ellipticity, and a uniform positive lower bound for
the boundary normal injectivity radius.
\end{lemma}

\begin{proof}
The proof has two parts.  We first obtain enough derivatives to solve
the geodesic defining-function equation without losing the regularity
required in the statement.  We then read off the first two coefficients
of the metric from the Einstein equation.

\medskip\noindent\emph{Regularity and the change of defining function.}
Choose relatively compact boundary neighborhoods
\[
K\Subset U'\Subset U
\]
in the given Euclidean boundary component.  On $U$ the metric
$g=\rho^2G$ is $C^{3,\beta}$, so $G$ is in particular $C^2$
conformally compact.  The boundary representative is the smooth flat
metric $\delta$.

We first use boundary harmonic coordinates.  Their coordinate
functions are $C^{4,\beta}$; passing to these coordinates preserves
the initial $C^{3,\beta}$ regularity of $g$.  The mixed Schouten
boundary system then improves this regularity one step at a time.
Indeed, suppose on a smaller patch that
\[
 g\in C^{k+2,\beta},\quad A\in C^{k,\beta},\quad
 H\in C^{k+1,\beta},\qquad k\ge1.
\]
The tangential Dirichlet data for the Schouten tensor are
$C^{k+1,\beta}$, and its normal data are $C^{k,\beta}$.  The interior
equation has the corresponding coefficient regularity.  The mixed
elliptic estimate consequently gives $A\in C^{k+1,\beta}$.  The
Codazzi equation then gives $H\in C^{k+2,\beta}$, and the triangular
harmonic-coordinate equations for the metric give
$g\in C^{k+3,\beta}$.

Repeating these three steps on successively smaller compact patches
produces $g\in C^{6,\beta}$.  The estimates depend only on the stated
local norm, ellipticity, patch separations, and collar control.  Since
$q=\delta$, no higher boundary-curvature jets enter.  The same
bootstrap applies to a finite-regularity limit by weak boundary
elliptic regularity; the boundary identities already have continuous
traces at the starting regularity.

It remains to improve the defining function.  To avoid dividing by a
function that vanishes at the boundary, use the adapted identities
\[
 P=-\tfrac12\Delta_g\rho\in C^{1,\beta},\quad
 dP=\Ric(\nabla\rho,\cdot),\quad
 \nabla^2\rho=-\tfrac12\rho\Ric-\tfrac12Pg.
\]
They successively control derivatives of $P$ and $\rho$;
in coordinates, the Hessian contains the usual Christoffel-symbol term
$\Gamma*d\rho$.  Iteration gives
$\rho\in C^{6,\beta}$ and $P\in C^{5,\beta}$.

Write $x=e^\omega\rho$.  The requirement that $x$ be geodesic is
equivalent to the following first-order equation for $\omega$:
\[
 F(z,d\omega):=2\langle d\rho,d\omega\rangle_g+
                  \rho|d\omega|_g^2-P=0,
 \qquad\omega|_{\partial X}=0.
\]
The equation is noncharacteristic at the boundary because
$\partial_{\rm in}\rho=1$.  Its Hamiltonian has
$C^{5,\beta}$ coefficients in the position variables, so the
characteristic vector field is $C^{4,\beta}$.  The characteristic
construction, followed by the inverse-function theorem, gives
$\omega\in C^{4,\beta}$ on a smaller collar whose width is uniform
under the assumed bounds.

We therefore have $x=e^\omega\rho$ and
$\bar g=e^{2\omega}g$ of class $C^{4,\beta}$, with
$|dx|_{\bar g}=1$.  The gradient field
$V=\nabla_{\bar g}x$ is $C^{3,\beta}$, as is its flow map.  Pulling
back the metric by that flow uses one derivative of the map.  This
explains the two-derivative reserve: the resulting normal-flow
splitting has precisely the regularity
\[
 \bar g=dx^2+h_x,\qquad h_x\in C^{2,\beta}.
\]
This accounts for the loss of derivatives in passing from adapted to
geodesic coordinates.  A finite cover of $K$ gives uniform estimates
for the families in the lemma.

\medskip\noindent\emph{The first coefficients of $h_x$.}
These coefficients are not free.  The Einstein equation fixes them
in terms of the boundary metric.  For
\[
G=x^{-2}\bar g
\]
in dimension four, the conformal Ricci formula, together with
$\Ric_G=-3G$ and $|dx|_{\bar g}=1$, gives
\begin{equation}\label{eq:geodesic-Einstein}
x\,\Ric_{\bar g}
+
2\nabla_{\bar g}^2x
+
(\Delta_{\bar g}x)\bar g
=
0.
\end{equation}
In the geodesic splitting $\bar g=dx^2+h_x$,
\[
\nabla_{\bar g}^2x|_{T\{x=\mathrm{const}\}}
=
\frac12h_x',
\qquad
\Delta_{\bar g}x
=
\frac12\tr_{h_x}h_x'.
\]
Hence the tangential part of \eqref{eq:geodesic-Einstein} is
\begin{equation}\label{eq:tangential-Einstein-geodesic}
x(\Ric_{\bar g})_{ij}
+
(h_x')_{ij}
+
\frac12(\tr_{h_x}h_x')(h_x)_{ij}
=
0.
\end{equation}

At $x=0$ the Ricci term disappears.  Taking the $h_0$-trace of
the remaining equation gives
\[
h_0'=0.
\]
The next coefficient is obtained by differentiating once in the
normal direction.  More precisely, subtract the equation at $x=0$,
divide by $x$, and let $x\downarrow0$.  This is legitimate because
$h_x\in C^{2,\beta}$, and gives
\begin{equation}\label{eq:second-boundary-coefficient-1}
(\Ric_{\bar g})_{ij}|_{x=0}
+
(h_0'')_{ij}
+
\frac12(\tr_{h_0}h_0'')(h_0)_{ij}
=
0.
\end{equation}

Let
\[
L_{ij}:=\frac12(h_x')_{ij}
\]
be the second fundamental form of the level sets of $x$.  Since $L=0$
at $x=0$, the tangential Gauss formula gives
\[
(\Ric_{\bar g})_{ij}|_{x=0}
=
(\Ric_{h_0})_{ij}
-\frac12(h_0'')_{ij}.
\]
Substitution into \eqref{eq:second-boundary-coefficient-1} gives
\[
\Ric_{h_0}
+
\frac12h_0''
+
\frac12(\tr_{h_0}h_0'')h_0
=
0.
\]
Taking the trace in boundary dimension three,
\[
\tr_{h_0}h_0''
=
-\frac12R_{h_0},
\]
and therefore
\[
h_0''
=
-2\Ric_{h_0}
+\frac12R_{h_0}h_0
=
-2\mathsf P_{h_0}.
\]
Thus the second coefficient is minus twice the Schouten tensor of
the boundary metric.  Since $h_0=\delta$ is flat,
\[
h_0'=0,
\qquad
h_0''=0.
\]

The desired decay is now a simple consequence of H\"older
regularity.  Since $h_x\in C^{2,\beta}$,
\[
|h_x''-h_0''|\le Cx^\beta.
\]
Thus
\[
|h_x'|
\le
C\int_0^x s^\beta\,ds
\le Cx^{1+\beta},
\]
and a second integration gives
\[
|h_x-\delta|
\le
C\int_0^x s^{1+\beta}\,ds
\le Cx^{2+\beta}.
\]
Therefore
\[
h_x=\delta+O(x^{2+\beta}),
\qquad
\partial_xh_x=O(x^{1+\beta}),
\]
with constants uniform under the stated family bounds.
\end{proof}

\begin{remark}\label{rem:gauge-acquisition-scope}
Notice that the pointed convergence does not immediately yield Lemma~\ref{lem:geodesic-gauge-acquisition}.  In the normalized blow-up below, boundary-normal injectivity is supplied first by the injectivity estimates, and the finite $C^{3,\beta}$ bounded-geometry control is supplied separately by the normalized local regularity argument.  The lemma only converts the finite-regularity into the geodesic gauge needed for the boundary rigidity theorem.  
\end{remark}

\subsection{The rigidity mechanism}

We can now state the central rigidity theorem.  Its proof is based on
a comparison between two ways of measuring the same boundary quantity:
one through volume growth in the Einstein metric, and the other through
the sharp Sobolev inequality on Euclidean $\mathbb R^3$.

The quantity arises already in hyperbolic space.  If $x$ is a
geodesic defining function, put $r=-\log x$; let $t$ be the distance
from a fixed interior point.  Toward the conformal boundary one has
$t=r+O(1)$.  The expression
\[
F=e^{(r-t)/2}=x^{-1/2}e^{-t/2}
\]
removes the leading divergence of $t$.  We shall show that its
boundary trace satisfies a Yamabe-type differential inequality on
$\mathbb R^3$.  Volume comparison bounds its $L^6$ mass from above;
the sharp Sobolev inequality bounds the same mass from below.  The two
constants coincide, forcing equality in Bishop--Gromov and hence
hyperbolic rigidity.

The assumptions are local on compact
boundary patches; we must therefore keep all distance
comparisons local until a globally defined boundary function is produced.

\begin{theorem}[Flat-boundary PE rigidity]\label{thm:boundary-rigidity}
Let $(X^4,G)$ be a smooth complete connected Riemannian manifold with $\Ric_G=-3G$.  Suppose $G$ has a conformal boundary component identified with $(\R^3,\delta)$ and, on every relatively compact boundary coordinate patch, a geodesic compactification
\[
G=x^{-2}(dx^2+h_x),
\qquad
h_0=\delta,
\]
which is locally $C^{2,\beta}$ for some $\beta>0$ and satisfies
\[
h_x=\delta+O(x^{2+\beta}),
\qquad
h_x'=O(x^{1+\beta}).
\]
Then $G$ is isometric to hyperbolic four-space.

If, in addition, $g=\rho^2G$ is a scalar-flat adapted compactification with boundary representative $\delta$, $\rho>0$ in $X$, and $|d\rho|_g\le1$, then $g$ is the flat metric on the standard closed half-space.
\end{theorem}

\begin{proof}
We first construct the boundary trace of the renormalized distance,
then derive an inequality for it.  The final step is a comparison of
sharp constants.  The construction is local near each compact part of
the boundary, whereas the comparison concerns the entire $\mathbb R^3$.

\medskip\noindent\emph{1. Controlling minimizing geodesics near the boundary.}

Fix one interior base point \(o\) for the rest of the proof.
Given \(K\Subset U'\Subset U\Subset\mathbb R^3\), where $K$ is compact and $U',U$ are bounded coordinate domains,
choose a thin collar
\[
(0,x_0)\times U
\]
whose closure excludes $o$. Shrink $x_0>0$ so that the tangential
metrics in
\[
G=x^{-2}(dx^2+h_x)
\]
are uniformly comparable to \(\delta\). Write
\(|Z|_\delta\le c^{-1/2}|Z|_{h_x}\) for a fixed comparison
constant. Set \(t=d_G(o,\cdot)\) and \(r=-\log x\).

The key geometric observation is that a geodesic which reaches a very
thin collar cannot travel far in the tangential directions.  This
follows by differentiating $r$ along the geodesic.

On the collar, \(|dr|_G=1\), and on vectors perpendicular to \(\nabla r\), \begin{equation}
\nabla_G^2r=\operatorname{Id}-\frac{x}{2}h_x^{-1}h_x'
=\operatorname{Id}+O(x^{2+\beta}).
\label{eq:br-hess}
\end{equation} Consequently, on every connected geodesic segment contained in this thin collar, \[
\phi=r'(\gamma),\qquad
\phi'=a(s)(1-\phi^2),\qquad
\tfrac12\le a(s)\le\tfrac32.
\] The tangential speed satisfies \[
|y'|_\delta\le c^{-1/2}x\sqrt{1-\phi^2}.
\] For a segment contained in \(x\le x_1\), monotonicity of \(\phi\), with the limiting interpretation on radial segments, gives \begin{equation}
\int |y'|_\delta\,ds
\le \frac{x_1}{(1/2)\sqrt c}
\int_{-1}^{1}\frac{d\phi}{\sqrt{1-\phi^2}}
=\frac{2\pi x_1}{\sqrt c}.
\label{eq:br-displacement}
\end{equation} Choose \(x_1>0\) so small that the right side is less than the
Euclidean distance from \(K\) to \(\partial U'\).  

For a minimizing geodesic from \(o\) to a point in \((0,x_1)\times K\), consider its last connected segment in \((0,x_1)\times U'\). Its entrance cannot be lateral: that would require the tangential displacement from \(\partial U'\) to \(K\), contradicting \eqref{eq:br-displacement}. Nor can entrance be through \(x=0\), which lies at infinite \(G\)-distance. It must enter through \(x=x_1\), at a point of the compact slice \(\{x_1\}\times\overline{U'}\). Write its entry time \(s_0\). At entry the radial velocity is nonnegative:
\[
\phi=\frac{d}{ds}r(\gamma(s))\ge0.
\]
Indeed, the geodesic enters the thin collar through $x=x_1$ in the
direction of decreasing $x$.  Along its final segment, \begin{equation}
\phi'=(1-\phi^2)(1+e(s)),\qquad |e(s)|\le Cx(s)^{2+\beta}.
\label{eq:br-phi}
\end{equation} Thus \(\phi\) increases, becomes bounded below by a positive constant within bounded time, and cannot exit back through \(x=x_1\). The selected terminal segment already stays inside the coordinate domain by construction.

Since \(x'=-x\phi\), direct differentiation gives \[
\frac d{ds}\log\frac{1-\phi^2}{x^2}=-2\phi e(s).
\] The integral of its absolute value is bounded: \[
\int_{s_0}^s|\phi e|\,ds
\le C\int_{x(s)}^{x_1}x^{1+\beta}\,dx.
\] It follows that \begin{equation}
1-\phi^2\le Cx^2.
\label{eq:br-phi-x}
\end{equation} The constant depends on the fixed compact local collar. Moreover, after bounded time \(x(s)\) decreases exponentially. The tangential speed \[
|dy/ds|_\delta\le Cx\sqrt{1-\phi^2}\le Cx^2
\] is integrable on the terminal part; the initial bounded-time part also has bounded displacement. The entry point belongs to the compact slice chosen above.

We next establish the asymptotic relation $t=r+O(1)$.  The entry points lie in a
compact slice, so their distances from $o$ are uniformly bounded.
Along the terminal segment, $\phi(s_0)\ge0$ and
$\phi'=a(s)(1-\phi^2)$ with $a(s)\ge\frac12$. Hence
\[
0\le\int_{s_0}^s(1-\phi)\,d\sigma
=\int_{\phi(s_0)}^{\phi(s)}\frac{d\phi}{a(1+\phi)}
\le 2\log2.
\]
Since $d(t-r)/ds=1-\phi$, the compactness of the entry slice
gives uniform two-sided bounds for $t-r$ on compact boundary cylinders.
This establishes the distance asymptotics without appealing to the
subsequently derived gradient estimate.

At points where the distance function is differentiable, its
gradient is the terminal unit tangent of a minimizing geodesic.  Hence \[
|d(t-r)|_G^2=2(1-\phi)\le Cx^2,
\] so \begin{equation}
|d(t-r)|_{dx^2+h_x}\le C.
\label{eq:br-gradient}
\end{equation}This estimate is understood almost everywhere; cut points need not
be excluded.  Since distance is locally Lipschitz, the uniform weak
gradient bound extends $t-r$ as a locally Lipschitz function to
$x=0$.

Define \[
F=x^{-1/2}e^{-t/2}=e^{(r-t)/2}.
\] There is a positive locally Lipschitz boundary trace \(F_0\), and on each compact boundary cylinder \begin{equation}
0<c_K\le F\le C_K,\qquad |dF|_{dx^2+h_x}\le C_K.
\label{eq:br-F-bounds}
\end{equation} 
\medskip\noindent\emph{2. The inequality satisfied by the renormalized distance.}

The boundary trace exists, but we still need an equation that
constrains it.  The starting point is Laplacian comparison for
the distance $t$.  It remains valid in the sense of distributions,
which is essential because $t$ is not smooth on its cut locus.

Completeness and \(\operatorname{Ric}(G)=-3G\) give distributional Laplacian comparison \[
\Delta_Gt\le3\coth t
\] away from \(o\), while \(|dt|_G=1\) almost everywhere. Applying the weak chain rule to \(v=e^{-t/2}\) reverses the comparison sign because the exponential is decreasing: \begin{equation}
(\Delta_G+\tfrac54)v
\ge-\frac{3v^5}{1-v^4}.
\label{eq:br-lapl-v}
\end{equation} The singular part supported on the cut locus has the favorable
sign.  Thus there is no need to describe the cut locus explicitly.

To pass from the interior inequality to the boundary, conjugate by
the leading factor $x^{1/2}$.  Write \(J=\sqrt{\det h_x}\), \(a=\partial_x\log J\). Since \(v=x^{1/2}F\), direct computation gives \[
(\Delta_G+\tfrac54)(x^{1/2}F)
=x^{5/2}\left[
F_{xx}+(a-x^{-1})F_x+\Delta_{h_x}F+\frac{a}{2x}F
\right].
\] Thus, in distributions, \begin{equation}
F_{xx}+(a-x^{-1})F_x+\Delta_{h_x}F+\frac{a}{2x}F
\ge-\frac{3F^5}{1-x^2F^4}.
\label{eq:br-F-pde}
\end{equation} The second-order operator is in divergence form with respect to
$J\,dx\,dy/x$.  This particular measure explains why logarithmic
averaging in $x$ is natural.

One further estimate is needed to remove the normal derivative in
that averaging argument.  It comes directly from the eikonal equation:
 \begin{equation}
F_x=-\frac{x}{F}\left(F_x^2+|d_yF|_{h_x}^2\right)
\quad\text{a.e.}
\label{eq:br-eikonal}
\end{equation} Consequently \(F_x/x\) is locally bounded, by \eqref{eq:br-F-bounds}. This is the key normal estimate.

\medskip\noindent\emph{3. Passing the inequality to the boundary.}

We cannot simply put $x=0$ in \eqref{eq:br-F-pde}: its normal part
contains the singular coefficient $x^{-1}$.  Instead we average with
respect to $dx/x$.  The normal contribution has bounded mass, whereas
the logarithmic length of the averaging interval tends to infinity.

Fix nonnegative \(\eta\in C_c^\infty(\mathbb R^3)\), and a sufficiently small fixed \(\delta<x_1\) in a compact collar neighborhood of its support. Geodesic defining functions associated with the same boundary representative agree on overlaps, so the local trace \(F_0\) is globally well-defined. Choose smooth \(0\le\chi_\epsilon\le1\), supported in \((\epsilon/2,\delta)\), equal to one on \((\epsilon,\delta/2)\), increasing and then decreasing, with \[
\int|\chi_\epsilon'|\,dx\le2.
\] Put \(L_\epsilon=\int\chi_\epsilon(x)\,dx/x\to\infty\). Test \eqref{eq:br-F-pde}, in its divergence form with measure \(J\,dx\,dy/x\), against \(\chi_\epsilon(x)\eta(y)\). The normal term is \[
-\int\frac JxF_x\eta\chi_\epsilon'\,dx\,dy.
\] Its absolute value is \(O_K(1)\), independently of \(\epsilon\), because \(F_x/x\) is bounded and \(\int|\chi_\epsilon'|\le2\). Dividing by \(L_\epsilon\), it tends to zero.

After integration by parts in \(y\), the tangential term is \[
\frac1{L_\epsilon}\int\frac{\chi_\epsilon}{x}
F\,(\Delta_{h_x}\eta)\,J\,dx\,dy
\longrightarrow \int F_0\Delta_\delta\eta\,dy.
\] This is the elementary logarithmic version of Ces\`aro
convergence: on the fixed compact support, all tangential coefficients
converge uniformly as $x\downarrow0$.

Because \(a/x=O(x^\beta)\), the zeroth-order term has logarithmic average zero. The right side converges to \(-3\int F_0^5\eta\). Therefore \begin{equation}
\Delta_\delta F_0\ge-3F_0^5
\quad\text{in distributions on }\mathbb R^3.
\label{eq:br-boundary-pde}
\end{equation} 

Set \(\psi=\sqrt2F_0\). Then \begin{equation}
-8\Delta_\delta\psi\le6\psi^5.
\label{eq:br-psi-pde}
\end{equation}

\medskip\noindent\emph{4. The upper bound from hyperbolic volume comparison.}

The main analytic step is finished: $\psi$ is a positive function on
$\mathbb R^3$ satisfying \eqref{eq:br-psi-pde}.  We now obtain two
sharp inequalities for its total sixth power.  First, metric balls
in $G$ cannot grow faster than balls in hyperbolic space, and the
boundary trace controls how much of such a ball lies in a thin collar.

On a fixed compact boundary set \(K\), \(F(x,y)\to F_0(y)\) uniformly and positively. Since \[
t=-\log x-2\log F,
\] for any fixed \(\eta>0\), a sufficiently thin compact collar has the property that, for large \(T\), \[
\{y\in K,\ e^{-T+\eta}F_0(y)^{-2}\le x\le\delta\}
\subset B_G(o,T).
\] The volume density is \(x^{-4}J\,dx\,dy\). Its leading term gives \[
\liminf_{T\to\infty}e^{-3T}\operatorname{Vol}_G B_G(o,T)
\ge\frac{e^{-3\eta}}3\int_KF_0^6\,dy.
\] The hyperbolic ball asymptotic is \[
\operatorname{Vol}_{\mathbb H^4}B(T)
\sim\frac{\omega_3}{24}e^{3T},\qquad \omega_3=\operatorname{Vol}(S^3)=2\pi^2.
\] Bishop--Gromov gives an upper volume ratio of one. Letting \(\eta\downarrow0\), then exhausting \(\mathbb R^3\) by \(K\), proves \begin{equation}
\int_{\mathbb R^3}F_0^6\le\frac{\omega_3}{8},
\qquad
0<A:=\int_{\mathbb R^3}\psi^6\le\omega_3.
\label{eq:br-mass}
\end{equation} More precisely, if \(\Theta\) is the asymptotic hyperbolic volume ratio, the same calculation gives \begin{equation}
\Theta\ge A/\omega_3.
\label{eq:br-theta}
\end{equation} We have used volumes of distance balls, rather than areas of
possibly singular distance spheres to avoid a regularity assumption on the distance function.

\medskip\noindent\emph{5. The reverse inequality from the sharp Sobolev constant.}

The boundary differential inequality gives the complementary estimate.
Since $\psi$ is not yet known to have finite Dirichlet energy, we
first test it with a compactly supported cutoff and only then let
the cutoff exhaust $\mathbb R^3$.

Let \(\zeta_R\) be a Euclidean cutoff, equal to one on \(B_R\), supported in \(B_{2R}\), with \(|d\zeta_R|\le C/R\). Test \eqref{eq:br-psi-pde} against \(\zeta_R^2\psi\), justified by local Lipschitz regularity and nonnegative smooth approximation. Completion of the square gives \begin{equation}
8\int|d(\zeta_R\psi)|^2
\le6\int\zeta_R^2\psi^6+8\int\psi^2|d\zeta_R|^2.
\label{eq:br-sobolev}
\end{equation} The cutoff error satisfies
\begin{align*}
\int\psi^2|d\zeta_R|^2
&\le\frac C{R^2}\operatorname{Vol}(B_{2R}\setminus B_R)^{2/3}
\left(\int_{B_{2R}\setminus B_R}\psi^6\right)^{1/3}\\
&\le C\left(\int_{B_{2R}\setminus B_R}\psi^6\right)^{1/3}\longrightarrow0.
\end{align*}
Thus \(\psi\) has finite homogeneous \(H^1\) energy and \begin{equation}
8\int|d\psi|^2\le6A.
\label{eq:br-equality}
\end{equation} Applying the sharp Euclidean three-dimensional Sobolev inequality to \(\zeta_R\psi\), and passing to the limit, gives \[
6\omega_3^{2/3}A^{1/3}\le8\int|d\psi|^2\le6A.
\] Since \(A>0\), the preceding chain implies \(A\ge\omega_3\).
Together with \eqref{eq:br-mass}, this gives
\[
A=\omega_3.
\]
Equation \eqref{eq:br-theta} now gives \(\Theta\ge1\), while
Bishop--Gromov gives \(\Theta\le1\).  All centered volume ratios
must therefore be equal to one.  Rigidity in Bishop--Gromov implies
that $(X,G)$ is hyperbolic four-space, proving the first assertion.

\medskip\noindent\emph{6. Flatness in the adapted gauge.}

Hyperbolicity is not quite enough for our blow-up argument, which
concerns the compactified metric $g$.  We must also identify the
boundary coordinates and show that the adapted conformal factor
is the standard one.

Once \(G\) is identified with hyperbolic four-space, the interior isometry induces, by taking endpoints of the normal geodesic rays in the local compactification collars, a globally defined local conformal map
\[
\operatorname{dev}:\mathbb R^3\longrightarrow S^3_\infty.
\]
The endpoint map is conformal because the induced metric on the
conformal boundary can be recovered from the leading term of the
hyperbolic metric.  Here is a direct verification.  Use the hyperboloid model of
$\mathbb H^4$ in Minkowski space with bilinear form
$\langle\ ,\ \rangle_{4,1}$. The normal curves, parametrized by
$r=-\log x$, are unit geodesics. Their images have the form
\[
 Z(r,y)=\tfrac12\bigl(e^rU(y)+e^{-r}V(y)\bigr),
\]
where $U,V$ are future null vectors. Since
$G_{ab}=e^{2r}(\delta_{ab}+o(1))$, differentiation in $y$ gives
\[
 \tfrac14\langle\partial_aU,\partial_bU\rangle_{4,1}
 =\lim_{r\to\infty}e^{-2r}G_{ab}=\delta_{ab}.
\]
Writing $U=U^0(1,\vartheta)$, with $U^0>0$ and
$\vartheta\in S^3$, the null relation gives
\[
 \vartheta^*g_{\rm rd}=\frac{4}{(U^0)^2}\delta.
\]
Thus the endpoint map is a local conformal immersion.
Compatibility on overlaps follows because the interior isometry is single-valued and geodesic defining functions for the same boundary representative agree where both are defined.  By the Liouville theorem in dimension three, each local conformal chart is the restriction of a M\"obius transformation of \(S^3\). On an overlap, two such transformations agree on an open set and hence coincide globally. The connectedness of \(\mathbb R^3\) allows continuation across a chain of overlapping charts to obtain one M\"obius transformation restricting to \(\operatorname{dev}\). Hence \(\operatorname{dev}\) is injective and its image is the complement of one point.  Composing the hyperbolic isometry with the inverse Möbius isometry therefore identifies the given boundary coordinates with the canonical Euclidean boundary coordinates \(y\) of \[
G=w^{-2}(dw^2+dy^2),\qquad w>0.
\] The scalar-flat compactification takes the form \[
g=u^2(dw^2+dy^2),\qquad u=\rho/w>0,\qquad u|_{w=0}=1.
\] The boundary value follows from the prescribed boundary representative \(\delta\), or from uniqueness of the geodesic defining function for that representative. Scalar flatness gives \(\Delta_{\mathbb R^4}u=0\).

The adapted gradient inequality becomes \[
|d(wu)|_\delta^2\le u^2,
\] and hence \[
w^2|du|_\delta^2+2wu\,u_w\le0.
\]
Thus $u_w\le0$.  Since $u(y,0)=1$ and $u>0$, it follows that
$0<u\le1$.  The function $u-1$ is therefore bounded and harmonic on
the half-space, with zero boundary values.  Its odd reflection is a
bounded harmonic function on $\mathbb R^4$.  By the harmonic Liouville
theorem it is constant, and by oddness the constant is zero.  Hence
$u=1$ and $g$ is Euclidean.  This proves the second assertion.

\end{proof}

\subsection{Excluding the boundary blow-up}

The flat-boundary theorem applies to the normalized limit, but only
after we show that the limit defining function remains positive in
the interior and that the resulting Einstein metric is complete.
These points are not automatic consequences of convergence of the
compactified metrics, and we verify them in the following proof.

\begin{proposition}[Exclusion of the boundary alternative]
\label{prop:exclude-boundary}
The alternative \eqref{eq:boundary-alternative} of
Proposition~\ref{prop:normalized-extraction} is impossible.
\end{proposition}

\begin{proof}
\medskip\noindent\emph{The limiting boundary.}
Assume \eqref{eq:boundary-alternative}.  After passing to a subsequence,
choose
\[
y_i\in\partial B^4
\]
such that
\[
d_{\bar g_i}(p_i,y_i)
=
d_{\bar g_i}(p_i,\partial B^4)
\le C.
\]
By \eqref{eq:scaled-boundary-geometry},
\[
\inj(S^3,\bar q_i,y_i)\longrightarrow\infty
\]
and, for every derivative order needed below,
\[
\widehat\nabla^{\,j}\widehat\Rm_{\bar q_i}
\longrightarrow0
\]
uniformly on bounded pointed balls.  Hence the pointed boundary manifolds
\[
(S^3,\bar q_i,y_i)
\]
converge in bounded-geometry coordinates to a complete flat
three-manifold of infinite injectivity radius.  Such a manifold is
isometric to
\[
(\R^3,\delta).
\]
Consequently the pointed limit
\[
(X_\infty,g_\infty,p_\infty)
\]
contains a boundary component intrinsically isometric to
$(\R^3,\delta)$.

\medskip\noindent\emph{Convergence of the defining functions.}
We must now recover the Einstein metric from the compactified limit.
The adapted equations pass unchanged under the constant rescaling;
explicitly, they read
\begin{equation}\label{eq:bar-adapted-system}
\begin{aligned}
\bar\rho_i|_{\partial B^4}&=0,
&
\partial_{\rm in}\bar\rho_i&=1,
&
|d\bar\rho_i|_{\bar g_i}&\le1,
\\
\nabla^2_{\bar g_i}\bar\rho_i
&=
-\frac12\bar\rho_i\Ric_{\bar g_i}
-\frac12\bar P_i\bar g_i,
&
d\bar P_i
&=
\Ric_{\bar g_i}(\nabla\bar\rho_i,\cdot).
\end{aligned}
\end{equation}
Moreover
\[
\bar P_i|_{\partial B^4}
=
\frac23\bar H_i.
\]
The scalar Gauss equation and normalized curvature bound give
uniform bounds for the rescaled boundary mean curvatures.  This is
the initial control for $\bar P_i$ because its boundary value is
$2\bar H_i/3$.

Fix a compact set
\[
L\Subset X_\infty
\]
at finite $g_\infty$-distance from the distinguished boundary component.
For sufficiently large $i$, identify $L$ with its image in $B^4$ by the
pointed convergence maps.  Since
\[
|d\bar P_i|_{\bar g_i}
\le
|\Ric_{\bar g_i}|\,
|d\bar\rho_i|_{\bar g_i}
\le C,
\]
and since $\bar P_i$ is uniformly bounded on the boundary, integration
along curves of uniformly bounded length from the boundary gives
\begin{equation}\label{eq:P-C0-bound-boundary}
\|\bar P_i\|_{C^0(L)}\le C_L.
\end{equation}
Similarly, $|d\bar\rho_i|\le1$ and $\bar\rho_i=0$ on the boundary imply
\begin{equation}\label{eq:rho-C0-bound-boundary}
\|\bar\rho_i\|_{C^0(L)}\le C_L.
\end{equation}

In the interior and boundary harmonic charts of
Proposition~\ref{prop:normalized-local-regularity},
\[
\|\bar g_i\|_{C^{3,\beta}}\le C_L,
\]
so
\[
\|\Ric_{\bar g_i}\|_{C^{1,\beta}}\le C_L.
\]
The coordinate form of the Hessian equation in
\eqref{eq:bar-adapted-system} is
\[
\partial_a\partial_b\bar\rho_i
=
\Gamma^c_{ab}(\bar g_i)\partial_c\bar\rho_i
-\frac12\bar\rho_i(\Ric_{\bar g_i})_{ab}
-\frac12\bar P_i(\bar g_i)_{ab}.
\]
The right-hand side is uniformly bounded, so $\bar\rho_i$ has
local $C^{1,1}$ bounds.  The point of using the Hessian equation,
rather than only the scalar Laplace equation, is that it controls
every second derivative, including the mixed derivatives at the
boundary.  Combining it with the Lipschitz bounds gives
\[
\|\bar\rho_i\|_{C^{2,\gamma}(L)}\le C_{L,\gamma}
\qquad(0<\gamma<\beta).
\]
Then
\[
d\bar P_i
=
\Ric_{\bar g_i}(\nabla\bar\rho_i,\cdot)
\]
gives
\[
\|\bar P_i\|_{C^{1,\gamma}(L)}\le C_{L,\gamma},
\]
and substitution once more in the Hessian equation gives
\[
\|\bar\rho_i\|_{C^{3,\gamma}(L)}\le C_{L,\gamma}.
\]
The same coordinate calculation applies on boundary harmonic charts; no
additional boundary regularity theorem is needed because all second
derivatives of $\bar\rho_i$ are prescribed by the Hessian equation and
the boundary values are fixed.

Thus, after a diagonal subsequence,
\begin{equation}\label{eq:boundary-rho-limit}
\bar P_i\longrightarrow P_\infty
\quad\text{locally in }C^{1,\gamma},
\qquad
\bar\rho_i\longrightarrow\rho_\infty
\quad\text{locally in }C^{3,\gamma},
\end{equation}
for every fixed
\[
0<\gamma<\beta,
\]
with
\[
\rho_\infty|_{\partial X_\infty}=0,
\qquad
\partial_{\rm in}\rho_\infty=1,
\qquad
|d\rho_\infty|_{g_\infty}\le1.
\]

\medskip\noindent\emph{Positivity of the limiting defining function.}
We claim that
\[
\rho_\infty>0
\qquad\text{on }\operatorname{int}X_\infty.
\]
This prevents the conformal factor from creating an interior
singularity in the limiting Einstein metric.  Indeed, the identity
\[
|d\bar\rho_i|_{\bar g_i}^2
=
1-\bar\rho_i\bar P_i
\]
passes to the limit.  If $\rho_\infty(z)=0$ at an interior point, then
$\rho_\infty\ge0$ and $z$ is an interior minimum, so
\[
d\rho_\infty(z)=0.
\]
On the other hand, the limiting identity gives
\[
|d\rho_\infty(z)|_{g_\infty}^2=1,
\]
a contradiction.

Hence, on every
\[
L\Subset\operatorname{int}X_\infty,
\qquad
\inf_L\rho_\infty>0.
\]
Since
\[
G_i
=
\rho_i^{-2}g_i
=
\bar\rho_i^{-2}\bar g_i,
\]
we obtain
\begin{equation}\label{eq:boundary-PE-limit}
G_i
\longrightarrow
G_\infty:=\rho_\infty^{-2}g_\infty
\quad\text{in }C^{3,\gamma}(L).
\end{equation}
Passing to the Ricci tensors gives
\[
\Ric_{G_\infty}=-3G_\infty
\]
on the interior.  Harmonic-coordinate elliptic regularity for
the Einstein equation then implies that $G_\infty$ is smooth on
$\operatorname{int}X_\infty$.

\medskip\noindent\emph{Completeness of the limiting Einstein metric.}
It remains to exclude a finite-length escape toward the conformal
boundary.  For this the gradient bound is especially well suited:
from
\[
G_\infty=\rho_\infty^{-2}g_\infty
\]
and $|d\rho_\infty|_{g_\infty}\le1$,
\begin{equation}\label{eq:logrho-Lipschitz-G}
|d\log\rho_\infty|_{G_\infty}
=
|d\rho_\infty|_{g_\infty}
\le1.
\end{equation}
Let $\sigma:[0,T)\to\operatorname{int}X_\infty$ have finite
$G_\infty$-length.  Equation \eqref{eq:logrho-Lipschitz-G} implies
\[
|\log\rho_\infty(\sigma(t))
-\log\rho_\infty(\sigma(0))|
\le L_{G_\infty}(\sigma|_{[0,t]}),
\]
so $\rho_\infty$ stays bounded above and bounded away from zero along
$\sigma$.  Hence
\[
L_{g_\infty}(\sigma)
=
\int_\sigma \rho_\infty\,ds_{G_\infty}
<\infty.
\]
Completeness of $(X_\infty,g_\infty)$ as a manifold with boundary
therefore gives a limit point in its metric completion.  It cannot lie
on $\partial X_\infty$, because there $\rho_\infty\to0$, while the
preceding logarithmic estimate keeps $\rho_\infty$ bounded away from
zero.  Thus the limit lies in the interior.  Hence $G_\infty$ is
complete.

\medskip\noindent\emph{Application of flat-boundary rigidity.}
Fix
\[
0<\gamma<\beta.
\]
The normalized boundary-normal injectivity bound and the
$C^{3,\gamma}$ bounded geometry pass to every compact boundary patch.
Lemma~\ref{lem:geodesic-gauge-acquisition}, applied with exponent
$\gamma$, yields a geodesic defining function $x$ satisfying
\begin{equation}\label{eq:boundary-geodesic-expansion-main}
G_\infty
=
x^{-2}(dx^2+h_x),
\qquad
h_x
=
\delta+O(x^{2+\gamma}),
\qquad
\partial_xh_x
=
O(x^{1+\gamma}).
\end{equation}
The remaining adapted-compactification properties also pass to the
limit.  Indeed,
\[
R_{\bar g_i}=0
\]
and the pointed $C^{3,\gamma}$ convergence imply
\[
R_{g_\infty}=0.
\]
Umbilicity is a first-order boundary condition on the metric and
therefore passes to the boundary limit.  Moreover
\[
\rho_\infty|_{\partial X_\infty}=0,
\qquad
\partial_{\rm in}\rho_\infty=1,
\qquad
|d\rho_\infty|_{g_\infty}\le1,
\]
by \eqref{eq:boundary-rho-limit} and the convergence already established.
Thus
\[
g_\infty=\rho_\infty^2G_\infty
\]
is a scalar-flat adapted compactification with Euclidean boundary
representative on the distinguished boundary component.
We have verified the local geodesic expansion, completeness of the
Einstein metric, and the adapted boundary conditions required by
Theorem~\ref{thm:boundary-rigidity}.  It follows that
\[
G_\infty\cong\mathbb H^4,
\qquad
g_\infty\cong(\R^4_+,g_{\rm eucl}),
\]
contradicting the normalized nonzero defect
\eqref{eq:nonzero-normalized-defect}.  Thus the boundary alternative is
impossible.
\end{proof}

\section{Interior blow-up and Ricci-flat ALE limits}\label{sec:interior}

We have excluded curvature concentration near the conformal boundary.
The remaining possibility is that the normalized base points recede
without bound from $\partial B^4$.  In that case the pointed limit is
complete and has no boundary.  One might expect it to be Ricci-flat,
but this does not follow merely by taking a limit of the scalar-flat
compactifications.  We must use the equation for their defining
functions.

The key observation is that the rescaled defining function cannot
remain bounded at a fixed point of a complete interior limit.  Once
its values tend to infinity, the Hessian equation forces the Ricci
tensor of the limit to vanish.  The Sobolev inequality then gives the
volume growth needed to recognize the limit as an asymptotically
locally Euclidean (ALE) space.  In the next section, the topology of
an ALE core will exclude that possibility unless the limit is flat.

\subsection{Curvature energy and the ALE condition}

Two facts special to four dimensions turn the Ricci-flat limit into
a space with controlled geometry at infinity.  First, Cheeger--Naber
\cite[Theorem~1.13]{CheegerNaber2015} prove a scale-invariant
$L^2$ curvature estimate for noncollapsed four-manifolds with bounded
Ricci curvature.  Consequently, if a complete Ricci-flat
four-manifold satisfies
\[
\Vol B(p,R)\ge v_0R^4
\qquad(R>0),
\]
then, after rescaling by $R^{-2}$ and applying their unit-ball estimate,
\[
\int_{B(p,R)}|\Rm|^2\,dV\le C(v_0)
\qquad(R>0).
\]
The essential point is that in dimension four the integral of
$|\Rm|^2$ is unchanged by this rescaling.  Since the same constant
works for every radius, monotone convergence gives finite total
curvature energy.

A complete Ricci-flat four-manifold with Euclidean volume growth
and finite total $L^2$ curvature admits ALE coordinates at infinity,
by the theorem of Bando--Kasue--Nakajima
\cite{BandoKasueNakajima1989}.  Thus, outside a compact set, the space
approaches $(\mathbb R^4\setminus B_R)/\Gamma$ for a finite group
acting freely on $S^3$.

These two results explain the form of the next theorem.  The analytic
part produces Ricci-flatness and ALE geometry; the last assertion
uses the topological obstruction proved in
Section~\ref{sec:ale-topology}.

\begin{theorem}[Interior defining-function divergence]\label{thm:interior-rho}
Let $(X_i,g_i,p_i)$ be smooth scalar-flat four-manifolds, possibly with boundary, and let $\rho_i>0$ in the interior.  Assume:
\begin{enumerate}
\item $G_i=\rho_i^{-2}g_i$ satisfies $\Ric_{G_i}=-3G_i$;
\item $|d\rho_i|_{g_i}\le1$;
\item $|\Ric_{g_i}|\le C$ globally;
\item $d_{g_i}(p_i,\partial X_i)\to\infty$;
\item $(X_i,g_i,p_i)$ converges in $C^{2,\gamma}_{\mathrm{loc}}$ on an exhaustion, for some $\gamma>0$, to a complete boundaryless pointed manifold $(E,g_\infty,p_\infty)$.  In the application one has $C^{3,\beta}$ convergence.
\end{enumerate}
Then
\[
\rho_i(p_i)\to\infty,
\qquad
\Ric_{g_\infty}=0.
\]
If, in addition, the limit inherits a Sobolev inequality
\[
\left(\int_E|f|^4\right)^{1/2}\le C_S\int_E|df|^2,
\qquad f\in C_c^1(E),
\]
then $E$ has two-sided Euclidean volume growth, finite total $L^2$ curvature, and is ALE.  If every compact subset of $E$ embeds smoothly in $B^4$, then $E$ is flat.
\end{theorem}

\begin{remark}[Logical order in the interior blow-up]\label{rem:interior-noncircular}
The two conclusions $\rho_i(p_i)\to\infty$ and $\Ric_{g_\infty}=0$ are proved in that order.  The divergence of the defining function uses only the scalar-flat conformal-Einstein equations, the gradient bound, the uniform normalized Ricci bound, completeness of the pointed limit, and the Omori--Yau principle.  Ricci-flatness is then obtained by dividing the Hessian equation by the diverging base-point value.  Thus we are able to establish Ricci-flatness using the defining-function itself. 
\end{remark}

\begin{proof}
The order of the proof is important.  We first show that the defining
function diverges before Ricci-flatness is established; only afterward do
we deduce that the limit is Ricci-flat.  The Sobolev and ALE conclusions
come last.

\medskip\noindent\emph{1. The adapted equations.}

Take \(\Delta=\operatorname{div}\nabla\). Put \[
P=\frac{1-|d\rho|^2}{\rho}\geq0.
\] The conformal Ricci formula is \[
\operatorname{Ric}_{G}
=\operatorname{Ric}_{g}+\frac2\rho\nabla^2\rho+
\left(\frac{\Delta\rho}{\rho}-\frac{3|d\rho|^2}{\rho^2}\right)g
=-\frac3{\rho^2}g.
\] Tracing with \(R_g=0\), then substituting back, gives \begin{equation}
\nabla^2\rho=-\frac{\rho}{2}\operatorname{Ric}_g-\frac P2g,
\qquad
\Delta\rho=-2P=-\frac{2(1-|d\rho|^2)}{\rho}.
\label{eq:int-hess}
\end{equation} Differentiating \(P\) and using \eqref{eq:int-hess} gives \begin{equation}
\nabla P=\operatorname{Ric}_g(\nabla\rho,\cdot).
\label{eq:int-P}
\end{equation} In particular $|\nabla P_i|\le C$.  These equations are
preserved by $(g,\rho,P)\mapsto(K^2g,K\rho,K^{-1}P)$.

\medskip\noindent\emph{2. A maximum-principle observation.}

The following elementary consequence of the Omori--Yau principle
explains why bounded defining functions are impossible in a complete
interior limit.

\noindent\textbf{Auxiliary lemma.} A complete boundaryless manifold with Ricci curvature bounded below admits no positive \(C^2\) function satisfying \begin{equation}
|d\rho|\leq1,\qquad
\Delta\rho=-2(1-|d\rho|^2)/\rho.
\label{eq:int-no-positive}
\end{equation}

\emph{Proof of the auxiliary lemma.} Apply the Omori--Yau minimum
principle to $\rho$, which is bounded below by zero. It gives \(z_j\) with \[
\rho(z_j)\to\inf\rho,\quad |d\rho|(z_j)\to0,\quad
\Delta\rho(z_j)\geq-j^{-1}.
\] For a fixed \(p\), eventually \(\rho(z_j)\leq\rho(p)+1\) and \(|d\rho|^2(z_j)\leq1/2\). Equation \eqref{eq:int-no-positive} implies \[
\Delta\rho(z_j)\leq-\frac1{\rho(z_j)}
\leq-\frac1{\rho(p)+1},
\] a contradiction. This also handles \(\inf\rho=0\). On a compact manifold use an actual minimum. \(\square\)

Only boundedness below of $\rho$ is required: one applies the
bounded-above form of the principle to $-\rho$.  The lower Ricci bound
is sufficient; see \cite{Borbely2012}.

\medskip\noindent\emph{3. Bounded basepoint values would give a forbidden limit.}

Suppose a subsequence has \(\rho_i(p_i)\leq A\). The gradient bound alone controls $\rho_i$ on bounded balls, but it
does not immediately control $P_i$.  We obtain the missing bound from
the Hessian equation.  Every unit-speed geodesic through $p_i$,
restricted to $[-1,1]$, lies in the interior for large $i$.  Along it, \[
0<\rho_i(\gamma(t))\leq A+1,\qquad
P_i(\gamma(t))\geq P_i(p_i)-C.
\] For \(a_i(t)=\rho_i(\gamma(t))\), \eqref{eq:int-hess} gives \[
a_i''(t)\leq\frac{C(A+2)-P_i(p_i)}2.
\] Integrating in both directions and adding cancels \(a_i'(0)\), so \[
0<a_i(1)+a_i(-1)
\leq2A+\frac{C(A+2)-P_i(p_i)}2.
\] Hence \begin{equation}
0\leq P_i(p_i)<4A+C(A+2).
\label{eq:int-P-base}
\end{equation}

On every fixed-radius pointed ball, the gradient estimates bound both \(\rho_i\) and \(P_i\). Equation \eqref{eq:int-hess} then bounds \(\nabla^2\rho_i\). Pull back to the fixed limit manifold and take a diagonal subsequence: \[
\rho_i\to\rho_\infty\text{ in }C^{1,\beta}_{\rm loc},\quad
P_i\to P_\infty\text{ locally uniformly},\quad 0<\beta<1.
\] If the nonnegative \(\rho_\infty\) vanished at an interior point, its gradient would vanish there. But \[
|d\rho_i|^2=1-\rho_iP_i
\] passes to the limit and would give \(|d\rho_\infty|^2=1\) there. Therefore \(\rho_\infty>0\) everywhere.

To improve the regularity of the limiting function, observe that the scalar
equation
\[
\Delta_{g_i}\rho_i
=
-\frac{2(1-|d\rho_i|_{g_i}^2)}{\rho_i}
\]
passes to the limit in distributions.  Since $\rho_\infty>0$,
$\rho_\infty\in C^{1,\beta}_{\rm loc}$, and
$g_\infty\in C^{2,\gamma}_{\rm loc}$, its right-hand side is locally
H\"older continuous after decreasing the exponent if necessary.
Interior Schauder estimates applied on compact subsets therefore give,
after passage to a further subsequence,
\[
\rho_i\longrightarrow\rho_\infty
\quad\text{locally in }C^{2,\beta'}
\]
for some $\beta'>0$.

We also need enough regularity to apply the maximum principle
in the classical form.  In fact the limit is smooth.  On every
compact subset
$K\Subset E$, positivity gives
\[
\inf_K\rho_\infty>0.
\]
Hence
\[
G_i=\rho_i^{-2}g_i
\longrightarrow
G_\infty:=\rho_\infty^{-2}g_\infty
\]
locally in $C^{2,\beta''}$ for some $\beta''>0$.
Passing to the Ricci tensors in
\[
\Ric_{G_i}=-3G_i
\]
gives
\[
\Ric_{G_\infty}=-3G_\infty
\]
classically.  Harmonic-coordinate elliptic regularity for the Einstein
equation therefore makes $G_\infty$ smooth. 

Moreover $R_{g_i}=0$ and the $C^2$ convergence give
\[
R_{g_\infty}=0.
\]
Since
\[
g_\infty=\rho_\infty^2G_\infty,
\]
the conformal scalar-curvature formula in dimension four yields
\[
-6\Delta_{G_\infty}\rho_\infty
+
R_{G_\infty}\rho_\infty
=
0.
\]
Because $R_{G_\infty}=-12$, this becomes
\[
\Delta_{G_\infty}\rho_\infty=-2\rho_\infty.
\]
Elliptic bootstrapping now gives
\[
\rho_\infty\in C^\infty(E),
\qquad
g_\infty=\rho_\infty^2G_\infty\in C^\infty(E).
\]

The limit Ricci tensor remains bounded below by $-C$.  We may therefore
apply the Omori--Yau minimum principle to the smooth complete manifold
$(E,g_\infty)$ and the positive function $\rho_\infty$.  The preceding
``Complete positive finite defining functions are impossible'' lemma
gives a contradiction.  Thus
\[
\rho_i(p_i)\longrightarrow\infty.
\]

Notice the logical point: this divergence uses the boundedness of
the normalized Ricci tensors, not their convergence to zero.  We may
therefore establish Ricci-flatness without circular reasoning.

\medskip\noindent\emph{4. Ricci-flatness of the pointed limit.}

Let \(a_i=\rho_i(p_i)\to\infty\) and \(u_i=\rho_i/a_i\). On every fixed pointed ball $B(p_i,R)$ the gradient bound gives $|\rho_i-a_i|\le R$; hence $\rho_i\to\infty$ locally uniformly, $u_i\to1$, and $du_i\to0$ locally uniformly. Divide the Hessian equation by \(a_i\): \begin{equation}
\nabla^2u_i=-\frac{u_i}{2}\operatorname{Ric}_{g_i}
-\frac{P_i}{2a_i}g_i.
\label{eq:int-u-hess}
\end{equation} Since \[
0\leq P_i/a_i\leq(a_i\rho_i)^{-1}\to0,
\] the right side tends to \(-\operatorname{Ric}_{g_\infty}/2\). The left side tends to zero in distributions because \(u_i\to1\), \(du_i\to0\), and the metrics converge in $C^{2,\gamma}_{\mathrm{loc}}$. Therefore \(\operatorname{Ric}_{g_\infty}=0\).  In harmonic coordinates the equation
\[
\operatorname{Ric}(g_\infty)=0
\]
is a uniformly elliptic quasilinear system with $C^{0,\gamma}$ coefficients.  Interior Schauder bootstrapping therefore upgrades $g_\infty$ from $C^{2,\gamma}$ to smoothness on $E$.  Together with the preceding step, this proves the first two conclusions of Theorem~\ref{thm:interior-rho}.

\medskip\noindent\emph{5. Volume growth, curvature energy, and ALE geometry.}

The final part explains why the Ricci-flat limit belongs to the class
of spaces treated by the embedded-core obstruction.  Assume that the limit satisfies
\begin{equation}\label{eq:int-sobolev}
\left(\int_E|f|^4\,dV_{g_\infty}\right)^{1/2}
\le
C_S\int_E|df|^2\,dV_{g_\infty},
\qquad
f\in C_c^1(E).
\end{equation}
Since $\Ric_{g_\infty}=0$, Bishop--Gromov gives
\[
\Vol B(p,2R)\le16\,\Vol B(p,R).
\]
Let $\eta$ be a cutoff which is one on $B(p,R)$, supported in
$B(p,2R)$, and satisfies $|d\eta|\le R^{-1}$.  Applying
\eqref{eq:int-sobolev} to $\eta$ yields
\[
\Vol B(p,R)^{1/2}
\le
C_SR^{-2}\Vol B(p,2R)
\le
16C_SR^{-2}\Vol B(p,R).
\]
Therefore
\begin{equation}\label{eq:int-growth}
\frac{R^4}{256C_S^2}
\le
\Vol B(p,R)
\le
\omega_4R^4
\qquad(R>0).
\end{equation}

Set
\[
v_0:=\frac{1}{256C_S^2}>0.
\]
For each $R>0$, define
\[
g_R:=R^{-2}g_\infty.
\]
Then
\[
\Ric_{g_R}=0,
\qquad
\Vol_{g_R}B_{g_R}(p,1)
=
R^{-4}\Vol_{g_\infty}B_{g_\infty}(p,R)
\ge v_0.
\]
Cheeger--Naber \cite[Theorem~1.13]{CheegerNaber2015} therefore gives
\[
\int_{B_{g_R}(p,1)}
|\Rm_{g_R}|^2\,dV_{g_R}
\le C(v_0).
\]
Here the statement of Cheeger--Naber is in terms of the average
of $|\Rm|^2$.  Bishop's upper bound
$\Vol_{g_R}B(p,1)\le\omega_4$ gives the displayed integral estimate.
If its volume is written with a strict inequality,
use the fixed lower constant $v_0/2$.
Because the $L^2$-curvature integral is invariant under constant
rescaling in dimension four,
\begin{equation}\label{eq:int-CN-global}
\int_{B_{g_\infty}(p,R)}
|\Rm_{g_\infty}|^2\,dV_{g_\infty}
\le C(v_0)
\qquad(R>0).
\end{equation}
Letting $R\to\infty$ and using monotone convergence gives
\begin{equation}\label{eq:int-finite-energy}
\int_E|\Rm_{g_\infty}|^2\,dV_{g_\infty}<\infty.
\end{equation}

The ALE description will have only one end.  Indeed, the complete
Ricci-flat manifold $E$ has exactly one end: if it had
at least two ends, the minimizing-geodesic construction would give
a line and the Cheeger--Gromoll theorem would imply $E\cong\R\times N$.
For this product to have more than one end, $N$ must be compact.
That gives at most linear volume growth, contrary to
$\Vol B(p,R)\ge v_0R^4$.

Hence, $E$ is complete, one-ended,
Ricci-flat, with Euclidean volume growth and finite critical
curvature.
Theorem~(1.5) and Remark~(1.8)(1) of
\cite{BandoKasueNakajima1989} give ALE coordinates, of order four
in dimension four, with derivative decay. In particular, outside
a compact set $E$ is diffeomorphic to
\[
(\R^4\setminus B_R)/\Gamma
\]
for a finite subgroup $\Gamma<SO(4)$ acting freely on $S^3$.

\medskip\noindent\emph{6. The embedding obstruction.}
Suppose every compact subset of $E$ embeds smoothly in $B^4$ and $E$
is nonflat.  We may then choose a sufficiently large compact ALE core
$D\subset E$.
However, by hypothesis, $D$ embeds smoothly in $B^4$.
Theorem~\ref{thm:embedded-ALE}, as we will establish below, excludes such an embedding for a complete
nonflat Ricci-flat ALE four-manifold.  Hence $E$ must be flat.  This
proves the final assertion of the theorem.

\end{proof}

\section{The obstruction to an embedded ALE core}\label{sec:ale-topology}

Why should an ALE bubble be excluded simply because it forms inside
the four-ball?  The decisive information is not the curvature near
the bubble's center, but the topology of a large compact region
$K$ in the bubble.  Its boundary is a spherical space form
$Y=S^3/\Gamma_\infty$.  If $K$ embeds in $B^4$, the hypersurface
$Y$ embeds in $S^4$, and the topology of its two complementary
regions severely restricts the possibilities.

The obstruction brings together two facts.  Mayer--Vietoris forces
the Euler characteristic and signature of $K$ to take their smallest
possible values.  The ALE Chern--Gauss--Bonnet and signature formulas
then turn those values into a stringent inequality for the signature
eta invariant of $Y$.  The possible embedded spherical links have
eta invariants too large to satisfy it.

\begin{theorem}[Embedded-core ALE obstruction]\label{thm:embedded-ALE}
Let $(E,g)$ be a connected, complete, oriented, smooth, nonflat Ricci-flat four-manifold with one ALE end $S^3/\Gamma_\infty$.  Suppose the decay is sufficient for the ALE Chern--Gauss--Bonnet and signature formulas, as is the case for Ricci-flat ALE metrics.  If a sufficiently large compact ALE truncation $K\subset E$ smoothly embeds in $B^4$, then a contradiction follows.  Hence no such nonflat $E$ exists.
\end{theorem}

The theorem uses the actual free orthogonal representation of
$\Gamma_\infty$, not just its abstract group.  This matters when
computing eta invariants.  The proof therefore has three parts:
topology of the core, the curvature inequality, and the representations
of the possible links.

\begin{proof}
\medskip\noindent\emph{1. Topology forced by the embedding.}
Let $Y=\partial K=S^3/\Gamma_\infty$ and compose the embedding $K\hookrightarrow B^4$ with $B^4\hookrightarrow S^4$.  Because the embedding of $K$ is smooth, the induced embedding of its
boundary $Y$ is smooth and hence locally flat.  Thus the locally-flat
embedding restrictions for spherical three-manifolds apply to this
actual boundary embedding. The connected hypersurface $Y$ separates $S^4$ into the given
embedded region $K$ and a connected compact complementary region $L$.
Thus
\[
S^4=K\cup_YL.
\]
Because $Y$ is a spherical space form, $H_1(Y;\mathbb Z)$ is finite and $H_2(Y;\mathbb Z)=0$.  Integral Mayer--Vietoris gives
\[
H_2(K;\mathbb Z)=H_2(L;\mathbb Z)=0.
\]
At degree one, the same sequence shows that
$H_1(K;\mathbb Z)\oplus H_1(L;\mathbb Z)$ is a quotient of the finite group
$H_1(Y;\mathbb Z)$; hence
\[
H_1(K;\mathbb Q)=H_1(L;\mathbb Q)=0.
\]
The connecting map $H_4(S^4;\mathbb Z)\to H_3(Y;\mathbb Z)$ sends the fundamental class to the oriented separating hypersurface and is an isomorphism, so $H_3(K;\mathbb Z)=H_3(L;\mathbb Z)=0$.  Since both pieces have nonempty boundary, $H_4(K;\mathbb Q)=H_4(L;\mathbb Q)=0$.  It follows that $\chi(K)=1$.  Moreover $H_2(K;\mathbb R)=0$, so the intersection form of $K$ is trivial and $\tau(K)=0$.  A sufficiently large ALE end is diffeomorphic to $Y\times[R,\infty)$. Since
\[
E=K\cup_Y\bigl(Y\times[0,\infty)\bigr),
\]
and the product end strongly deformation retracts onto
$Y=Y\times\{0\}\subset K$, the inclusion
\[
K\hookrightarrow E
\]
is a homotopy equivalence.  Consequently the real homology of $E$ is the homology of $K$:
\[
H_*(E;\R)\cong H_*(K;\R),
\qquad
\chi(E)=\chi(K)=1.
\]
In particular $H_2(E;\R)=0$, so the intersection form, and hence the
signature in the ALE formula, vanishes:
\[
\tau(E)=0.
\]  Hence
\begin{equation}\label{eq:chi-tau-ALE}
\chi(E)=\chi(K)=1,
\qquad
\tau(E)=\tau(K)=0.
\end{equation}

\medskip\noindent\emph{2. The curvature inequality for the link.}

The Ricci-flat condition leaves only the two Weyl curvature terms in
the characteristic formulas.  We specify the norm convention because
it affects the numerical coefficients.  In the following formulas,
$|W^\pm|$ denotes the
Hilbert--Schmidt norm of the Weyl endomorphism of $\Lambda^\pm$.
Its squared full four-index tensor norm is four times this squared
operator norm. With the full tensor norm the corresponding coefficients
are $1/(32\pi^2)$ and $1/(48\pi^2)$; the resulting eta inequality is
identical.
Write
\[
I_+=\int_E|W^+|^2\,dV,
\qquad
I_-=\int_E|W^-|^2\,dV.
\]
For a Ricci-flat ALE four-manifold, the Chern--Gauss--Bonnet and signature formulas are
\begin{align}
\chi(E)&=\frac{I_++I_-}{8\pi^2}+\frac1{|\Gamma_\infty|},\label{eq:ALE-GB}\\
\tau(E)&=\frac{I_+-I_-}{12\pi^2}-\eta_{\rm sig}(Y),\label{eq:ALE-signature}
\end{align}
with the usual orientation convention; see \cite{Nakajima1990,Zhang2019}.  From \eqref{eq:chi-tau-ALE} and $|I_+-I_-|\le I_++I_-$,
\begin{equation}\label{eq:eta-bound}
|\eta_{\rm sig}(Y)|
\le\frac23\left(1-\frac1{|\Gamma_\infty|}\right)<\frac23
\end{equation}
whenever $\Gamma_\infty$ is nontrivial.  If $\Gamma_\infty$ is trivial, \eqref{eq:ALE-GB} forces
$I_++I_-=0$ and the metric is flat, contrary to our hypothesis.

\medskip\noindent\emph{3. The possible links and their representations.}

The inequality \eqref{eq:eta-bound} is meaningful only if we know
which spherical space forms can occur as the boundary of an embedded
region of $S^4$.  Since the actual spherical link $Y$ embeds smoothly, and hence locally
flatly, in $S^4$, the finite-fundamental-group case of
\cite[Theorem~5.2(1), p.~52]{Hillman2024} (see also \cite{CrispHillman1998}) gives
\[
\Gamma_\infty\cong 1,\quad Q_8,\quad\text{or }I^*=SL(2,5).
\]
The trivial case is excluded.  For the other two groups, the eta
invariant depends on the given free action on $S^3$, not merely on an
abstract identification of $\Gamma_\infty$.  We therefore establish
that the relevant orthogonal representations lie, up to orientation,
in one of the two $SU(2)$ factors of $SO(4)$.

\smallskip\noindent\emph{The binary dihedral case.}
For the generalized binary dihedral group with $m\ge2$,
\[
\operatorname{Dic}_m=
\langle a,b: a^{2m}=1,\ b^2=a^m,\ bab^{-1}=a^{-1}\rangle,
\]
a faithful free orthogonal action on $\R^4$ forces the two real rotation angles of $a$ to agree up to sign.  Indeed $a^m$ has order two, and freeness implies $a^m=-I$.  If the two angles were distinct up to sign, the two rotation planes of $a$ would be intrinsically determined and preserved by $b$; on each plane $b$ conjugates a nontrivial rotation to its inverse and is therefore orientation-reversing, hence has square $I$.  This would give $b^2=I$, contradicting $b^2=a^m=-I$.  Hence there is an orthogonal complex structure $J$ with
\[
a=\cos\theta\,I+\sin\theta\,J,
\qquad bJ=-Jb,
\qquad J^2=b^2=-I.
\]
The operators $J$ and $b$ generate a quaternionic structure on $\R^4$; after
identifying $\R^4\cong\mathbb H$, both $a$ and $b$ act by left multiplication
by unit quaternions.  Thus, after possibly reversing orientation, the
representation lies in one $SU(2)$ factor.  In particular this applies to $Q_8=\operatorname{Dic}_2$.

\smallskip\noindent\emph{The binary icosahedral case.}
For $I^*$, let
\[
\rho_\pm:I^*\longrightarrow SO(\Lambda^2_\pm)\cong SO(3)
\]
be the two homomorphisms induced by the end representation.  Let $z$ be the
central involution of $I^*$.  Freeness implies $z=-I$ on $\R^4$, so $z$ acts
trivially on $\Lambda^2_\pm$.  Thus both $\rho_\pm$ factor through
\[
I^*/\{1,z\}\cong A_5.
\]
Because $A_5$ is simple, each induced homomorphism is either trivial or
injective, in the latter case with image isomorphic to $A_5$.  Suppose both
are nontrivial.  Choose $u\in I^*$ of order four with $u^2=z$.  Its image in
$A_5$ has order two, so both $\rho_+(u)$ and $\rho_-(u)$ have order two.
Represent $u$ by $(\ell,r)\in SU(2)\times SU(2)$ under
$\operatorname{Spin}(4)\to SO(4)$.  Since the induced actions are $\operatorname{Ad}_\ell$ and
$\operatorname{Ad}_r$, both $\ell$ and $r$ have trace zero.  Hence they are conjugate unit
quaternions: $\ell=q r q^{-1}$ for some $q\in S^3\subset\mathbb H$.  But then
$u$ fixes $q$ under the action $x\mapsto\ell x r^{-1}$, contradicting
freeness.  Therefore one of $\rho_\pm$ is trivial.  The kernel of either
projection $SO(4)\to SO(\Lambda^2_\pm)$ is a single $SU(2)$ factor, so the
actual end representation lies in one $SU(2)$ factor, up to orientation.

\medskip\noindent\emph{4. The eta invariants.}
We can now use the same equivariant signature formula for both
remaining links.  For a finite subgroup of one $SU(2)$ factor,
if an element has complex eigenvalues $e^{\pm i\theta_\gamma}$,
the spherical equivariant signature formula \cite{LockViaclovsky2016} gives, up to orientation,
\begin{equation}\label{eq:eta-sum}
|\eta_{\rm sig}(S^3/\Gamma_\infty)|
=
\frac1{|\Gamma_\infty|}\sum_{\gamma\ne1}
\cot^2\!\left(\frac{\theta_\gamma}{2}\right).
\end{equation}
For $Q_8=\operatorname{Dic}_2$, the cyclic subgroup contributes
\[
\sum_{k=1}^{3}\cot^2\frac{\pi k}{4}=2,
\]
and the remaining four elements each contribute $1$.  Dividing by $8$ gives
\begin{equation}\label{eq:eta-Q8}
|\eta_{\rm sig}(S^3/Q_8)|=\frac34.
\end{equation}
This contradicts \eqref{eq:eta-bound}, whose right-hand side is $7/12$.

For $I^*$, quotient by the center to the rotational icosahedral group $A_5$.  It has $15$ rotations through $\pi$, $20$ through $2\pi/3$, and $12$ each through $2\pi/5$ and $4\pi/5$.  The two $SU(2)$ lifts of a rotation through $\alpha$ contribute
\[
\cot^2(\alpha/4)+\tan^2(\alpha/4)
=4\csc^2(\alpha/2)-2.
\]
The central nonidentity element contributes zero.  The total contributions are respectively $30$, $200/3$, and $144$, so
\begin{equation}\label{eq:eta-Istar}
|\eta_{\rm sig}(S^3/I^*)|
=\frac{30+200/3+144}{120}
=\frac{361}{180}.
\end{equation}
The right-hand side of \eqref{eq:eta-bound} is $119/180$ in this
case, giving the required contradiction.  We have exhausted the
possible spherical links, and the theorem follows.
\end{proof}

\subsection{Excluding the interior blow-up}

We apply the embedded-core obstruction to the pointed limit already
constructed in Section~\ref{sec:interior}.  Only one point remains
to be checked: the global Sobolev inequality must pass to that limit.
This passage uses the fact that every fixed compact subset is
represented in the interior of the original manifold for large $i$.

\begin{proposition}[Exclusion of the interior alternative]\label{prop:exclude-interior}
The alternative \eqref{eq:interior-alternative} of Proposition~\ref{prop:normalized-extraction} is impossible.
\end{proposition}

\begin{proof}
Assume \eqref{eq:interior-alternative}.  The pointed limit
$(X_\infty,g_\infty,p_\infty)$ is complete and boundaryless.
Theorem~\ref{thm:interior-rho} first gives
\[
\bar\rho_i(p_i)\longrightarrow\infty,
\qquad
\Ric_{g_\infty}=0.
\]

It remains to verify the Sobolev statement in the second part of
that theorem.  Since the relevant exponent is the critical exponent
in dimension four, the Sobolev constant is unchanged by the
constant blow-up rescaling.  By \eqref{eq:global-interior-sobolev}, every compactly
supported interior function for the rescaled metrics satisfies
\[
\left(\int|f|^4\,dV_{\bar g_i}\right)^{1/2}
\le
\frac{6}{c_1(Y_0)}
\int|df|_{\bar g_i}^2\,dV_{\bar g_i}.
\]
This inequality is invariant under constant rescaling in dimension four.

Let $f\in C_c^1(X_\infty)$.  Since the boundary escapes to infinity,
the support of $f$ is identified, for all sufficiently large $i$, with
a compact interior subset of $(B^4,\bar g_i)$ by the pointed convergence
maps.  Transport $f$ to that subset and extend it by zero after inserting
a cutoff which is identically one on its support.  Passing to the limit
under the pointed $C^{3,\beta}$ convergence gives
\[
\left(\int_{X_\infty}|f|^4\,dV_{g_\infty}\right)^{1/2}
\le
\frac{6}{c_1(Y_0)}
\int_{X_\infty}|df|_{g_\infty}^2\,dV_{g_\infty}.
\]
Thus the Sobolev requirement of
Theorem~\ref{thm:interior-rho} holds.

Proposition~\ref{prop:normalized-extraction} also gives smooth embeddings
of every compact subset of $X_\infty$ into $B^4$.  The final assertion of
Theorem~\ref{thm:interior-rho} therefore implies that $g_\infty$ is flat.
This contradicts \eqref{eq:nonzero-normalized-defect}.  Hence the interior
alternative is impossible.
\end{proof}

\section{Completion of the compactness theorem}\label{sec:compactness-proof}

The two blow-up alternatives have both been excluded.  The remaining
steps take place on the original four-ball, without rescaling.
First we recover a global bound on volume and Weyl curvature from the
pointwise curvature estimate.  Then we use the mixed elliptic equations
from the preliminaries to improve local control to the prescribed
boundary regularity and assemble the local charts into a global,
boundary-marked limit.

It is useful to separate these tasks: the energy estimate is short,
whereas the global compactness requires attention to coordinates,
boundary markings, and defining functions.

\begin{lemma}[Recovery of volume and Weyl energy]\label{lem:energy-recovery}
Assume
\begin{equation}\label{eq:pointwise-curvature-bound}
\sup_i\|\Rm_{g_i}\|_{L^\infty(B^4)}\le K<\infty.
\end{equation}
Set
\[
A_0:=\sup_i\Area(S^3,q_i)<\infty.
\]
Then there exists $H_0<\infty$, depending only on $K$ and the compact boundary family, such that $|H_i|\le H_0$, and
\begin{equation}\label{eq:volume-bound2}
\Vol(B^4,g_i)\le
V_0:=\left(\frac{2A_0H_0}{c_0}\right)^2.
\end{equation}
Consequently
\begin{equation}\label{eq:energy-bound-conclusion}
\int_{B^4}|W_{G_i}|^2\,dV_{G_i}
=\int_{B^4}|W_{g_i}|^2\,dV_{g_i}
\le K^2V_0.
\end{equation}
\end{lemma}

\begin{proof}
The scalar Gauss equation \eqref{eq:gauss-boundary} bounds the
boundary mean curvatures by $H_0$, using the curvature
and the compact family of boundary metrics.  To turn that boundary
bound into an interior volume bound, test the first Escobar quotient
\eqref{eq:escobar} with the constant function $u\equiv1$.
Since $R_{g_i}=0$, \eqref{eq:Y1-Y2-lower} and \eqref{eq:c0-def}
give
\begin{equation}\label{eq:volume-bound}
c_0\le Y_1(B^4,S^3,[g_i])
\le\frac{2\int_{S^3}H_i\,dV_{q_i}}{\Vol(B^4,g_i)^{1/2}}
\le\frac{2\int_{S^3}|H_i|\,dV_{q_i}}{\Vol(B^4,g_i)^{1/2}}
\le\frac{2A_0H_0}{\Vol(B^4,g_i)^{1/2}},
\end{equation}
Rearranging gives \eqref{eq:volume-bound2}.  Finally, in dimension
four the density $|W|^2dV$ is conformally invariant.  Because
$|W_{g_i}|\le|\Rm_{g_i}|$, the pointwise curvature bound together
with the volume estimate gives \eqref{eq:energy-bound-conclusion}.
\end{proof}

\begin{remark}[The recovered global energy bound]\label{rem:first-global-CG}
Lemma~\ref{lem:energy-recovery} is the first point where a global
Weyl-energy bound is available.  It is useful to identify precisely
what this changes.  Before the lemma, the curvature estimates were
proved by normalized injectivity, finite boundary regularity and the
two blow-up exclusions.  After the lemma, the hypotheses of the
compactness theorem \cite[Theorem~I]{ChangGe2025} are accessible to us:
the boundary family is compact and PSC, its Yamabe constants have a
uniform positive lower bound, the total Weyl energy is bounded, and
$H_2(B^4;\mathbb Z)=H^2(B^4;\mathbb Z)=0$.
We thus invoke that theorem for final finite-order
compactness. 
Below, we record also the direct boundary-harmonic/Schouten closure 
with the coordinate, boundary-marking and defining-function steps used 
by the later weighted continuation argument.
\end{remark}

\begin{proof}[Proof of Theorem~\ref{thm:compactness}]
\medskip\noindent\emph{1. Uniform curvature and the first global bounds.}
Suppose, contrary to the first assertion, that
\[
\sup_i\left(
\|\Rm_{g_i}\|_{L^\infty(B^4)}
+
\|\nabla\Rm_{g_i}\|_{L^\infty(B^4)}
\right)=\infty,
\]
then Proposition~\ref{prop:normalized-extraction} produces, after passage to a
subsequence, a normalized pointed limit satisfying either
\eqref{eq:boundary-alternative} or \eqref{eq:interior-alternative}.
Propositions~\ref{prop:exclude-boundary} and
\ref{prop:exclude-interior} exclude both alternatives.  Therefore
\begin{equation}\label{eq:global-curvature-first-derivative}
\sup_i\left(\|\Rm_{g_i}\|_{L^\infty}+\|\nabla\Rm_{g_i}\|_{L^\infty}\right)<\infty.
\end{equation}
Lemma~\ref{lem:energy-recovery} now supplies the missing global estimate
\[
\sup_i
\int_{B^4}|W_{G_i}|^2\,dV_{G_i}
<\infty.
\]
At this stage the energy-dependent compactness theorem of Chang--Ge
\cite[Theorem~I]{ChangGe2025} applies, with no missing topological
condition on the four-ball.  For any preassigned $k\ge6$ and
$\alpha'<\alpha$ it gives the asserted convergence of adapted
compactifications modulo boundary-fixing diffeomorphisms. We give 
the direct finite-regularity closure as well, since it makes the 
boundary marking and the defining function used in the next section explicit.

\medskip\noindent\emph{2. A finite global atlas.}
We now use the pointwise bounds and the finite regularity estimates
to obtain compactness of the metrics themselves.  Once $|\Rm|+|\nabla\Rm|$ is uniformly bounded, use one fixed
rescaling and the proved injectivity estimate to obtain the controlled
atlas in Proposition~\ref{prop:normalized-local-regularity}. Let its radius be $r>0$.
Every metric ball of radius $r/4$, including a boundary ball, has volume
at least $v>0$. The volume bound in Lemma~\ref{lem:energy-recovery} bounds the number of
disjoint such balls. Selecting separated points along any minimizing
metric-space geodesic therefore gives a uniform diameter bound.
Consequently the whole compact manifold is covered by a uniformly
bounded number of the controlled charts. The chart-and-overlap proof
gives global $C^{3,\beta}$ Cheeger--Gromov compactness, with the
eventual smooth diffeomorphism type of the ball.

\medskip\noindent\emph{3. Keeping the boundary marking fixed.}

Geometric compactness initially gives convergence after changing
coordinates on the whole compact manifold.  Our continuation argument
must keep the identification of the boundary with the fixed sphere.
We therefore explain why the changes of coordinates can be corrected
to fix the boundary pointwise.

To see this explicitly,
suppose $\Phi_i:X_\infty\to B^4$ are the initial global convergence
diffeomorphisms and $f_i=\Phi_i|_{\partial X_\infty}$. The original
boundary metrics $q_i$ converge on the fixed $S^3$. The pulled-back
boundary metrics converge on $\partial X_\infty$. The maps $f_i$
are uniformly bilipschitz. In local coordinates the transformation law
for the Levi--Civita connections of $f_i^*q_i$ and $q_i$ bounds the
second derivatives of $f_i$ by first derivatives and the two Christoffel
symbols. Differentiating that identity supplies a uniform
$C^{4,\beta}$ bound at the current metric regularity. After a
subsequence $f_i\to f_\infty$ as diffeomorphisms, with one more
derivative than the boundary metrics. Fix a sufficiently large index $i_0$.
The smooth boundary maps $b_i=f_i^{-1}\circ f_{i_0}$ are all close
to the identity and converge to $b_\infty=f_\infty^{-1}\circ f_{i_0}$.
Extend them to the fixed collar by a cutoff isotopy, obtaining smooth
diffeomorphisms $\eta_i$ converging to $\eta_\infty$ in a topology
with one more derivative than the metric convergence. Then
$\Phi_i\circ\eta_i$ all restrict to the same smooth map $f_{i_0}$.
The maps
$\Phi_i\circ\eta_i\circ\Phi_{i_0}^{-1}:B^4\to B^4$
are smooth and fix the boundary pointwise; their pulled-back metrics
converge to a fixed pullback of $\eta_\infty^*g_\infty$.
The construction uses only maps close to the identity in a fixed
collar.

\medskip\noindent\emph{4. Recovering the prescribed boundary regularity.}

The normalized regularity proposition gives only the first three
metric derivatives.  We now use the mixed Schouten and harmonic
boundary systems to reach the finite order specified in the theorem.
Assume
$q_i$ is uniformly $C^{k,\alpha}$, $k\ge6$. At an intermediate
stage $g_i\in C^{r,\beta}$, $3\le r\le k-1$, one has
$A_i\in C^{r-2,\beta}$ and $H_i\in C^{r-1,\beta}$.
The Schouten Dirichlet data have regularity $C^{r-1,\beta}$,
because $\widehat A(q_i)\in C^{k-2,\alpha}$; the Neumann
data have regularity $C^{r-2,\beta}$, because
$dR_{q_i}\in C^{k-3,\alpha}$. The mixed estimate gains one
derivative for $A_i$, Codazzi then gains one for $H_i$, and the
triangular metric equation gains one for $g_i$. Iteration gives
$g_i\in C^{k,\beta}$ uniformly in a finite atlas.

The same iteration also recovers any H\"older exponent strictly
below the one prescribed at the boundary.  More precisely, there is no
obstruction to recovering $\alpha'<\alpha$.
In a final mixed estimate for $A_i\in C^{k-2,\alpha'}$, the
interior right side needs only $C^{k-4,\alpha'}$; the already
available $C^{k,\beta}$ metric gives spare derivatives for all those
coefficients and products. In the Dirichlet data,
$\widehat A(q_i)$ is $C^{k-2,\alpha}$ and $H_i^2$ is
$C^{k-1,\beta}\subset C^{k-2,\alpha'}$. In the Neumann
data, $dR_{q_i}$ is $C^{k-3,\alpha}$ and $H_i\,dH_i$ is
$C^{k-2,\beta}\subset C^{k-3,\alpha'}$. Thus the mixed
estimate yields $A_i\in C^{k-2,\alpha'}$, Codazzi yields
$H_i\in C^{k-1,\alpha'}$, and the metric boundary equations
yield $g_i\in C^{k,\alpha'}$. To obtain convergence rather than
just a bound, do this with an exponent strictly between the desired
$\alpha'$ and $\alpha$, then use compact embedding. The boundary
marking correction improves simultaneously by the connection identity.

\medskip\noindent\emph{5. Recovering the defining function.}

Convergence of the compactified metrics is not quite the end:
we must show that the limiting compactification is the canonical
adapted one and that it gives a complete Einstein metric on the
interior.  We use the defining-function identities again.

Suppose $g_i\to g_\infty$ in $C^{k,\nu}$ on the fixed ball,
after boundary-fixing diffeomorphisms. Uniform
ellipticity gives a uniform diameter $D$. The adapted gradient bound
and boundary condition give
\[
 0\le\rho_i\le d_{g_i}(\cdot,\partial B^4)\le D.
\]
Moreover $P_i|_{\partial}=2H_i/3$ and
$|dP_i|\le|\Ric_i|$ give a uniform global bound for $P_i$.
The Hessian equation supplies $\rho_i\in C^{1,1}$ in the fixed
controlled charts. Repeatedly using
\[
 dP_i=\Ric_i(\nabla\rho_i,\cdot),\qquad
 \partial_a\partial_b\rho_i
 =\Gamma^c_{ab}(g_i)\partial_c\rho_i
   -\tfrac12\rho_i(\Ric_i)_{ab}-\tfrac12P_i(g_i)_{ab}
\]
gives uniform $C^{k,\nu}$ bounds for $\rho_i$ and
$C^{k-1,\nu}$ bounds for $P_i$. Thus, for every
$\nu'<\nu$, they converge after a subsequence in
$C^{k,\nu'}$ and $C^{k-1,\nu'}$, respectively.
The limit has boundary value zero and inward derivative one.
The identity $|d\rho_\infty|^2=1-\rho_\infty P_\infty$
prevents an interior zero. Hence
$G_\infty=\rho_\infty^{-2}g_\infty$ is a PE metric on the
interior and is complete, since the defining function vanishes
simply at the entire boundary of a compact manifold. 
The resulting scalar-flat compactification is canonical for
$q_\infty$.  Indeed, if
$\widetilde g=\widetilde\rho^{\,2}G_\infty$ is another scalar-flat
compactification with the same boundary metric, choose a fixed smooth
defining function $r$. Both $\rho_\infty/r$ and
$\widetilde\rho/r$ extend positively and continuously to the boundary:
this follows by integrating their $r$-derivatives and needs only
$C^1$ defining functions. The constructed quotient
$\rho_\infty/r$ is in fact $C^{k-1,\nu'}$. Hence
$f=\widetilde\rho/\rho_\infty$ is positive and continuous on the
compact closure, $C^2$ in the interior, and $f|_{\partial B^4}=1$
because $\widetilde g|_{T\partial B^4}=g_\infty|_{T\partial B^4}
=q_\infty$. Since $\widetilde g=f^2g_\infty$, the four-dimensional
scalar-curvature transformation law gives
\[
 -6\Delta_{g_\infty}f
 =R_{\widetilde g}f^3-R_{g_\infty}f=0.
\]
The maximum and minimum principles on the compact connected manifold
with boundary yield $f\equiv1$. Therefore the constructed limit agrees
with the canonical adapted compactification.

This proves all parts of the compactness theorem: a
boundary-marked limit, the stated finite H\"older regularity,
and a complete smooth Einstein metric in the interior.
\end{proof}

\section{Finite-dimensional equivariant Poincar\'e--Einstein continuation}
\label{sec:equivariant-moduli}

We now come to the second difficulty in the proof: the existence of a
filling must be continued even when the linearized Einstein equation
has a nontrivial kernel.  At such a metric the implicit-function
theorem does not, by itself, continue the solution along an arbitrary
path of boundary metrics.  Compactness prevents solutions from
escaping; it does not prevent two branches from meeting or a branch
from turning back.

The guiding idea is to enlarge the prescribed boundary path by a
\emph{finite} number of auxiliary parameters.  These parameters are
chosen to remove the obstruction at each filling.  The resulting
solution sets are then ordinary finite-dimensional manifolds, and a
mod-two boundary count can be carried out.  Three points require
care: local solutions must be genuine Einstein metrics, different
local descriptions must agree smoothly, and the compactness theorem
must apply in the topology in which those descriptions live.

We treat these questions in that order.  The first subsection
constructs local harmonic-map gauges and finite-dimensional solution
families.  The second identifies the invariant obstruction using a
boundary pairing.  The third supplies a topology of \emph{marked}
fillings---metrics whose prescribed boundary identification has not
been forgotten---and transfers geometric compactness to it.  The
last subsection assembles the construction and applies parity.

\subsection{A differentiable harmonic-map gauge for finite parameter families}
\label{sec:finite-harmonic-gauge}

Composing a metric of fixed H\"older regularity with a moving diffeomorphism can
lose derivatives. We solve the Einstein equation in complete weighted
Banach spaces, and obtain as many additional \emph{spatial}
derivatives as desired along the finite-dimensional families of
solutions.  Those derivatives will later make changes of gauge
smooth in the finite parameters.

Throughout this subsection each displayed weighted space is equipped
with its actual complete norm.  

We work first near one fixed smooth filling.  The boundary geometry
is encoded by a defining function $x$, while a power of $x$ measures
how fast a perturbation vanishes at infinity.  Thus the weight will
control both the elliptic mapping properties and the ordinary
regularity of the compactified metric.

Fix a smooth conformally compact Poincar\'e--Einstein metric $G_*$ on
$X=B^4$, with $\operatorname{Ric}_{G_*}=-3G_*$, and a smooth boundary
representative $q_*$.  In a fixed collar choose its geodesic defining
function $x$ and identify
\[
 x^2G_*=dx^2+q_*-x^2P_{q_*}+O(x^3),
 \qquad P_q=\operatorname{Ric}_q-\tfrac14R_qq.
\]
Extend $x$ positively to the interior.  For a tensor bundle $E$ define
\[
 \mathcal H^{k,\alpha}_\delta(E)
 =\{u:x^{-\delta}u\in C^{k,\alpha}(X,G_*;E)\},\qquad
 \|u\|_{k,\alpha,\delta}
 =\|x^{-\delta}u\|_{C^{k,\alpha}(X,G_*)}.
 \tag{HG.1}
\]
Here the unweighted norm is the bounded, uniform geometric H\"older
norm, equivalently the supremum of ordinary norms in uniformly sized
rescaled hyperbolic charts.  We always take $k\ge4$ and $0<\alpha<1$.
These are Banach spaces.  The notation $\mathcal H^\infty_\delta$
means membership at every finite order; it does not mean ordinary
smoothness on the compactification.
Fix throughout the metric and gauge weight
\[
 d=\frac52.
 \tag{HG.1a}
\]
The inequalities $2<d<3$ serve two different purposes.  The upper
bound places the metric equation in the Fredholm range determined by
its indicial roots.  The lower bound ensures that the compactified
remainder has two continuous derivatives and no contribution to the
first two boundary jets.  Both facts will be used; the numerical value
$d=5/2$ is a convenient choice in this interval.

The reference metric must be chosen with some care.  If we varied
only the leading boundary term $q$, the error in the Einstein equation
would appear too early in the expansion to lie at the chosen weight.
We therefore prescribe also the second normal coefficient, determined
by the Schouten tensor of $q$.  This is the \emph{prepared boundary
jet}.

Let $a\mapsto q(a)$ be a finite-dimensional smooth family of smooth
boundary metrics near $q_*$.  Smoothness here includes smooth dependence
in every finite boundary $C^r$ norm.  With a fixed cutoff equal to one
near $x=0$, set
\[
 H(a)=G_*+\chi x^{-2}\pi^*\left[
 q(a)-q_*-x^2\bigl(P_{q(a)}-P_{q_*}\bigr)\right].
 \tag{HG.2}
\]
Thus $x^2H(a)=dx^2+q(a)-x^2P_{q(a)}+O(x^3)$, with the universal
geodesic second jet as well as zero radial first jet.
For small $a$ these are positive metrics, uniformly equivalent to
$G_*$, and their Ricci curvatures are uniformly negative.  Each finite
spatial derivative of $H(a)$ depends smoothly on $a$ in the uniform
geometric norm.  The collar curvature formula gives
\[
 \operatorname{Ric}_{H(a)}+3H(a)
 \in\mathcal H^\infty_3(S^2T^*X)
 \subset\mathcal H^\infty_d(S^2T^*X),
 \tag{HG.3}
\]
smoothly in $a$.  Here is the exact second-jet check.  Put
$\bar H=dx^2+k(x)$, $k=q+x^2b+O(x^3)$, and
$A=\tfrac12k'$, $\mathcal H=\operatorname{tr}_kA$.
The conformal Ricci identity is
\[
 \operatorname{Ric}_{x^{-2}\bar H}+3x^{-2}\bar H
 =\operatorname{Ric}_{\bar H}
       +x^{-1}\bigl(2\nabla^2x+(\Delta x)\bar H\bigr).
 \tag{HG.3a}
\]
At $x=0$, its tangential ordinary components have constant term
$\operatorname{Ric}_q+b+(\operatorname{tr}_qb)q$;
its $00$ constant term is $-\operatorname{tr}_qb+
\operatorname{tr}_qb=0$, and its mixed components vanish to
positive order.  With $b=-P_q$, the tangential term is zero because
$\operatorname{tr}P_q=R_q/4$.
Thus the ordinary residual is $O(x)$ and its complete-metric norm
is $O(x^3)$.  Smoothness supplies the same order after every
geometric derivative.  

\subsubsection{The Einstein equation in harmonic identity gauge}

The Ricci tensor is invariant under diffeomorphisms, and its
linearization is not elliptic before one fixes this freedom.  We use
an equation which is elliptic and has a particularly simple
linearization at the reference Einstein metric.  The added term
measures the failure of the identity map to be harmonic.  We will prove
the vanishing for every solution
of the gauged equation.

For two metrics $g,H$ put
\[
 T(g,H)^c=g^{ab}\bigl(\Gamma(H)^c_{ab}-\Gamma(g)^c_{ab}\bigr),
 \qquad \omega(g,H)=g(T(g,H),\cdot).
 \tag{HG.4}
\]
Thus $T(g,H)$ is the tension of the identity map from $(X,g)$ to
$(X,H)$.  Define
\[
 \Phi(a,h)=\operatorname{Ric}_g+3g+\delta_g^*\omega(g,H(a)),
 \qquad g=H(a)+h.
 \tag{HG.5}
\]
This is the Einstein--DeTurck equation belonging to harmonic identity
gauge.  

The first lemma ensures that this gauge introduces no spurious
solutions, using the negative Ricci curvature to force the
gauge one-form to vanish.

\begin{lemma}[Mapping properties and removal of the DeTurck term]
\label{lem:hg-equation}
For $h$ sufficiently small in $\mathcal H^{k,\alpha}_d$,
\[
 \Phi:A\times\mathcal H^{k,\alpha}_d(S^2T^*X)
       \longrightarrow\mathcal H^{k-2,\alpha}_d(S^2T^*X)
 \tag{HG.6}
\]
is smooth.  Its zeros are precisely Einstein metrics $g=H(a)+h$
satisfying $T(g,H(a))=0$.  At $(a,h)=(0,0)$, where $H(0)=G_*$,
\[
 D_h\Phi= L_{G_*}
 =\tfrac12\nabla^*\nabla-\mathring R.
 \tag{HG.7}
\]
\end{lemma}

\begin{proof}
\emph{Smoothness in the weighted spaces.}
In each uniform chart the expression is a smooth algebraic expression
in $g,g^{-1},\partial g,\partial^2g,H,\partial H,\partial^2H$.
All differences in (HG.4) contain $h$ or one derivative of $h$;
hence $\omega\in\mathcal H^{k-1,\alpha}_d$.  Taylor expansion about
$H(a)$, together with (HG.3), puts (HG.5) in the stated weighted
range.  Banach multiplication and matrix inversion prove smoothness;
no coefficient of the moving metric is evaluated at a moving point.

The same local calculation gives a quantitative estimate which will be
needed at degenerate fillings.  In a hyperbolically rescaled chart,
the second tensor derivative of the gauged operator is a sum of terms
of the forms
\[
 u\,D_0^2v,\qquad v\,D_0^2u,\qquad
 (D_0u)(D_0v),\qquad u\,D_0v,\qquad uv,
\]
with uniformly controlled contractions.  Differentiating at most
$k-2$ times uses no more than $k$ derivatives of either tensor.
Because the product of two weight-$d$ tensors has weight $2d>d$,
\[
 \|D_h^2\Phi(a,h)[u,v]\|_{k-2,\alpha,d}
 \le C\|u\|_{k,\alpha,d}\|v\|_{k,\alpha,d}.
\]
Integrating twice along the segment from $0$ to $h$ gives
\begin{equation}\label{eq:uniform-gauge-quadratic}
\|\Phi(a,h)-\Phi(a,0)-D_h\Phi(a,0)h\|_{k-2,\alpha,d}
\le C\|h\|_{k,\alpha,d}^{2},
\end{equation}
and comparing the derivatives along the segment from $(0,0)$ to
$(a,h)$ gives
\begin{equation}\label{eq:uniform-gauge-derivative}
\|D_h\Phi(a,h)-L_{G_*}\|_{\mathcal L(\mathcal H_d^{k,\alpha},
                                            \mathcal H_d^{k-2,\alpha})}
\le C\bigl(|a|+\|h\|_{k,\alpha,d}\bigr).
\end{equation}
The constants are uniform on a smaller finite-parameter neighborhood.
These are estimates at the stated order.

\smallskip\noindent\emph{Removing the gauge term.}
Use $\beta_g=\delta_g+\tfrac12d\operatorname{tr}_g$, with
$\delta$ the negative divergence.  The contracted Bianchi identity
and the one-form Weitzenb\"ock identity give, exactly,
\[
 \beta_g\Phi
 =\beta_g\delta_g^*\omega
 =\tfrac12(\nabla^*\nabla-\operatorname{Ric}_g)\omega.
 \tag{HG.8}
\]
The metric $g$ remains uniformly Ricci-negative in the chosen
neighborhood.  At a zero of $\Phi$, integrate (HG.8) against $\omega$.
Both $\omega$ and $\nabla\omega$ are $L^2$: the squared decay
$x^{2d}=x^5$ is integrable against $x^{-4}dx\,dV_q$.  Cutoffs with bounded
geometric derivatives justify integration by parts, and give
\[
 0=\int_X\bigl(|\nabla\omega|^2
       -\operatorname{Ric}_g(\omega^\sharp,\omega^\sharp)\bigr)dV_g.
\]
Consequently $\omega=0$ and $\operatorname{Ric}_g=-3g$.
The converse is immediate.  Finally, at $g=H=G_*$,
$D_g\omega(h)=\beta_{G_*}h$.  The derivative of Ricci plus $3g$
is $L_{G_*}h-\delta^*\beta h$; the added term cancels the second
summand and proves (HG.7).
\end{proof}

\subsubsection{A gauge change which avoids derivative loss}

The preceding equation describes metrics already in harmonic identity
gauge.  To compare two nearby metrics, however, we must also put a
metric into that gauge by a boundary-fixing change of coordinates.
This is where the choice of harmonic \emph{map} matters.  In its
equation the possibly rough domain metric is evaluated at a fixed
point; only the smooth reference metric is evaluated at a moving
point.  Hence differentiation with respect to the map does not
require differentiating a rough tensor by composition.

For a small vector field $V$ let $F_V(z)=\exp_z^{G_*}V(z)$.
Transport the tension of $F_V:(X,g)\to(X,H(a))$ back to $T_zX$
along this fixed-reference short geodesic.  Denote the resulting
section by $\mathcal T(a,h,V)$.  Its coordinate formula before
transport is
\[
 \tau^c(g,H,F)=g^{ij}(z)\left[
 \partial_i\partial_jF^c
 -\Gamma(g)^\ell_{ij}(z)\partial_\ell F^c
 +\Gamma(H)^c_{ab}(F(z))\partial_iF^a\partial_jF^b\right].
 \tag{HG.9}
\]
The position of the arguments in this formula is essential: the
rough metric $g$ is evaluated at $z$, and only the smooth reference
$H(a)$ is evaluated at $F(z)$.

We state the local gauge theorem in the precise weighted spaces in
which it will be used.  The final assertion about the boundary jet is
what permits later comparisons without changing the prescribed
boundary representative.

\begin{lemma}[Nonlinear harmonic-map gauge]
\label{lem:hg-nonlinear-gauge}
Fix $\delta\in(2,d]$.  On a neighborhood of $(0,0,0)$,
\[
 \mathcal T:
 A\times\mathcal H^{k,\alpha}_d(S^2T^*X)
  \times\mathcal H^{k+1,\alpha}_\delta(TX)
 \longrightarrow\mathcal H^{k-1,\alpha}_\delta(TX)
 \tag{HG.10}
\]
is smooth.  Its vector derivative at $(0,0,0)$ is
\[
 D_V\mathcal T=-(\nabla^*\nabla-\operatorname{Ric}_{G_*})
             =-(\nabla^*\nabla+3).
 \tag{HG.11}
\]
This is an isomorphism.  There is therefore a unique small
$V(a,h)\in\mathcal H^{k+1,\alpha}_d$ with
$\tau(g,H(a),F_{V(a,h)})=0$, depending smoothly on $(a,h)$.
It is also the unique sufficiently small root at every auxiliary
weight $\delta\in(2,d)$.

The maps $F_{V(a,h)}$ are diffeomorphisms, extend as $C^{3,\theta}$
maps to $\bar X$ for every $0<\theta<d-2=1/2$, and have the
identity three-jet at the boundary.
The pushed-forward metric $(F_{V(a,h)})_*g$ is in harmonic identity
gauge relative to $H(a)$.
\end{lemma}

\begin{proof}
\emph{Differentiability of the map equation.}
Formula (HG.9) is a second-order differential expression in $V$,
with $g$ and its first derivatives occurring at the fixed domain
point.  The smooth target has bounded derivatives of every order in
uniform charts.  Nemytskii differentiation, the chain rule, and
Banach multiplication therefore prove (HG.10).  At $V=0$, the
residual is $T(g,H(a))\in\mathcal H^{k-1,\alpha}_d$; every other
term has at least one factor $V$ or a derivative of $V$, and lies
in weight $\delta$.  This also proves the weighted claim for every
parameter derivative.

\smallskip\noindent\emph{Invertibility of the vector equation.}
The identity-map Jacobi formula gives (HG.11).  For completeness,
the limiting hyperbolic indicial polynomials, in a frame
$e_0=\partial_t$, $e_a=e^{-t}\partial_{y_a}$, are
\[
 s(3-s)+6\quad\hbox{on the normal component},\qquad
 s(3-s)+4\quad\hbox{on the tangential components}.
 \tag{HG.12}
\]
Indeed $\nabla_{e_0}e_a=0$, $\nabla_{e_a}e_0=e_a$, and
$\nabla_{e_a}e_b=-\delta_{ab}e_0$; applying
$\nabla^*\nabla+3$ to $e^{-st}e_0$ and $e^{-st}e_a$ gives
these expressions.  The common interval between roots is $(-1,4)$.
Lee's weighted Fredholm theorem, in the full geometric H\"older
spaces defined above, identifies the kernel and cokernel throughout
this interval with their $L^2$ counterparts; see
\cite[Theorem~C and Proposition~E]{Lee2006}.  In particular no
identification of the full uniform H\"older space with a closure of
smooth bounded sections is required.  The coercive identity
\[
 \langle(\nabla^*\nabla+3)V,V\rangle_{L^2}
 =\|\nabla V\|_{L^2}^2+3\|V\|_{L^2}^2
\]
annihilates both kernel and adjoint obstruction.  Hence the operator
is an isomorphism at $d$ and at every auxiliary
$2<\delta<d$.

One may also see directly why these weights are favorable.  For the
scalar comparison operator $-\Delta_{G_*}+3$,
\[
 (-\Delta_{G_*}+3)x^\delta
 =\bigl(3+\delta(3-\delta)+O(x)\bigr)x^\delta.
\]
Kato's inequality and this positive barrier control the weighted
$C^0$ norm of a vector solution; uniformly rescaled Schauder
estimates control its derivatives.  This is consistent with the
Fredholm inverse just established.

The Banach implicit-function theorem now gives the unique small
solution and its smooth parameter dependence.  More precisely, a
Taylor expansion of the harmonic-map tension at $V=0$ gives
\[
 \mathcal T(a,h,V)
 =T(H(a)+h,H(a))-J_{G_*}V+\mathcal R(a,h,V),
\]
where $J_{G_*}=\nabla_{G_*}^*\nabla_{G_*}+3$ and
\[
 \|\mathcal R(a,h,V)\|_{k-1,\alpha,d}
 \le C\Bigl((|a|+\|h\|_{k,\alpha,d})\|V\|_{k+1,\alpha,d}
                 +\|V\|_{k+1,\alpha,d}^{2}\Bigr).
\]
Absorbing the small terms after applying $J_{G_*}^{-1}$ yields
\begin{equation}\label{eq:uniform-small-harmonic-gauge}
 \|V(a,h)\|_{k+1,\alpha,d}
 \le C\|T(H(a)+h,H(a))\|_{k-1,\alpha,d}
 \le C'\|h\|_{k,\alpha,d}.
\end{equation}
Since the weight-$d$ root lies in each weaker weighted space,
uniqueness in the smaller auxiliary-weight neighborhood identifies
the two constructions.  

\smallskip\noindent\emph{Behavior at the boundary.}
In a rescaled chart at height $x_0$, $V$ and its chart derivatives
are $O(x_0^d)$.  The fixed exponential map therefore gives ordinary
coordinate displacement $F_V-z=O(x^{d+1})$, with
\[
 |D^j(F_V-\operatorname{Id})|\le Cx^{d+1-j},\qquad 0\le j\le4.
 \tag{HG.13}
\]
The bounds for $j=3,4$ give a $C^{3,\theta}$ extension for
$\theta<d-2$, with identity three-jet.  For small $V$ it is a local diffeomorphism,
has bounded $G_*$ displacement, and is proper.  Its degree is one,
so it is a global diffeomorphism of $B^4$.
Finally tension is natural under precomposition by an isometry.
The map $F_V:(X,g)\to(X,(F_V)_*g)$ is an isometry, whence
$\tau(g,H(a),F_V)=0$ is equivalent to
$\tau((F_V)_*g,H(a),\operatorname{Id})=0$.
\end{proof}

\subsubsection{Ordinary boundary regularity and finite-dimensional families}

The weighted estimates are phrased in the complete Einstein metric,
where a boundary neighborhood looks large after rescaling.  For the
later marked topology we must also know what they imply in ordinary
coordinates on the closed ball.  Because $d>2$, the compactified
perturbation has a vanishing second boundary jet and is classically
$C^2$.  This permits the use of boundary regularity for Einstein
metrics with smooth conformal infinity.

\begin{lemma}[The ordinary regularity actually supplied by weight $d$]
\label{lem:hg-admissibility}
If $h\in\mathcal H^{3,\alpha}_d(S^2T^*X)$, then $\bar h=x^2h$
has, in the fixed smooth collar coordinates,
\[
 |D^j\bar h|\le Cx^{d-j},\qquad 0\le j\le3.
 \tag{HG.14}
\]
In particular $x^2(H(a)+h)$ belongs to $C^{2,\theta}$ for every
$0<\theta<d-2=1/2$, and has the same two-jet as
$dx^2+q(a)-x^2P_{q(a)}$.  Every Einstein zero of (HG.5)
admits an ordinarily smooth compactification after a
boundary-identity collar change of coordinates when $q(a)$ is smooth.
\end{lemma}

\begin{proof}
In the frame $x\partial_x,x\partial_y$, the components of $h$
are exactly the ordinary components of $\bar h$.  Weighted
geometric differentiation controls their derivatives by $x\partial$.
Expressing an ordinary derivative of order $j\le3$ in these
operators proves (HG.14).  The second derivatives tend to zero at
boundary like $x^{d-2}$.  To check their H\"older seminorm,
compare two points with coordinate separation $\ell$.  If $\ell$
is comparable to the smaller boundary height, the pointwise bound
gives $C\ell^{d-2}$.  Otherwise join the points in a region of
comparable height and integrate the third-derivative bound to get
$C\ell x^{d-3}\le C\ell^{d-2}$.  This proves the claimed
$C^{2,\theta}$ control and its continuity in the weighted norm.

Interior ellipticity of (HG.5) first makes every zero smooth in the
given interior coordinates. Apply Chru\'sciel--Delay--Lee--Skinner, \emph{Boundary regularity
of conformally compact Einstein metrics}, Theorem A,
arXiv:math/0401386v2, printed p.~3.  We need only an
Einstein metric, a classical $C^2$ conformal compactification, and a
smooth boundary representative.  In dimension four it produces a
smooth geodesic compactification after a boundary-identity collar
diffeomorphism.  The theorem initially states only $C^{1,\lambda}$
regularity of that diffeomorphism, for every $\lambda<1$.
First extend this collar change to the whole ball.  Both boundary
metrics are $q(a)$, so its first derivative is the identity at the
boundary.  It is $C^1$ close to the identity on a sufficiently thin
collar.  A cutoff of its displacement, in a fixed collar chart or
embedding, gives a diffeomorphism equal to it near the boundary and
to the identity farther inside.  The cutoff derivative times the
displacement tends to zero with the collar width, because the
displacement is $O(x^{1+\lambda})$.  The resulting pullback is
smoothly compactified globally; the extension is smooth in the
interior and fixes every boundary point.

Now the extended map improves to $C^{3,\theta}$.  Construct the
unique negative scalar curvature $-1$ compactification with boundary
metric $q(a)$ on both whole compact balls.  The global
maximum-principle uniqueness in the compactness-bridge section makes
the extended map an isometry of these compact metrics.  The original
one is $C^{2,\theta}$ and the new one is smooth.  A $C^1$ compact
metric isometry is weakly harmonic in target harmonic coordinates,
and elliptic regularity with the $C^{2,\theta}$ coefficients
gives $C^{3,\theta}$ regularity; equivalently one then uses the
connection transformation law.  This supplies an admissible
equivalence without changing the boundary marking or the interior
smooth structure.
\end{proof}

We can now apply the implicit-function theorem.  Its kernel is
finite-dimensional, and additional boundary parameters will later
span its cokernel.  The resulting zero set is finite-dimensional;
the essential supplementary assertion is that its metrics possess
arbitrarily many spatial derivatives.  This is precisely what the
smooth overlap construction will require.

\begin{lemma}[Finite reductions and derivative reserves]
\label{lem:hg-finite-reductions}
Suppose the derivatives of finitely many smooth boundary parameters
span the cokernel of $L_{G_*}$ in (HG.7).  Then the zeros of (HG.5)
near $(0,0)$ form a finite-dimensional smooth manifold of dimension
$\dim A$.  In every local finite-dimensional chart its metric
representatives depend smoothly on the chart coordinates with values
in $\mathcal H^{j,\alpha}_d$ for each finite $j$.

The same conclusion holds after restricting to the tensors invariant
under a fixed finite isometric action, provided the boundary
parameters span the invariant cokernel.
\end{lemma}

\begin{proof}
We first construct the local zero set by an implicit-function theorem.
We then explain why its elements acquire arbitrarily high weighted
spatial regularity, and why this extra regularity depends smoothly
on the finite parameters.

\smallskip\noindent\emph{The finite-dimensional zero set.}
Choose coordinates $a=(a_1,\ldots,a_J)$ on the finite boundary
parameter space $A$, where $J=\dim A$, and let $S_j$ be the prepared
Einstein derivative in the $a_j$ direction.  By hypothesis these
columns span the obstruction to $L_*=L_{G_*}$, so the augmented operator
\[
 A_0(t,u)=L_*u+\sum_{j=1}^{J}t_jS_j,
 \qquad (t,u)\in\mathbb R^J\oplus\mathcal H^{k,\alpha}_d,
\]
is surjective.  It is Fredholm, with finite-dimensional complemented
kernel $V$.  Choose a bounded projection $\Pi$ onto $V$.  Then
\[
 \mathscr A_0=(A_0,\Pi):
 \mathbb R^J\oplus\mathcal H^{k,\alpha}_d
 \longrightarrow\mathcal H^{k-2,\alpha}_d\oplus V
\]
is bijective and has bounded inverse by the open-mapping theorem.
The derivative estimate \eqref{eq:uniform-gauge-derivative}, together
with the smooth boundary columns, implies
\[
 \|\mathscr A_{a,h}-\mathscr A_0\|
 \le C\bigl(|a|+\|h\|_{k,\alpha,d}\bigr),
 \qquad \mathscr A_{a,h}=(D\Phi(a,h),\Pi).
\]
After shrinking the parameter and tensor neighborhood, a Neumann
series therefore gives the uniform estimate
\begin{equation}\label{eq:uniform-augmented-inverse}
\|(t,u)\|_{\mathbb R^J\oplus\mathcal H^{k,\alpha}_d}
\le2\|\mathscr A_0^{-1}\|
\|\mathscr A_{a,h}(t,u)\|_{\mathcal H^{k-2,\alpha}_d\oplus V}.
\end{equation}
The implicit-function theorem, using $\Pi$ as the fixed finite
kernel coordinate, now gives a smooth zero-set chart.  Since $L_*$
has index zero, the dimension of the zero set is $\dim A$.
The projection and the Jacobi kernel have not been discarded: only
the finite-dimensional obstruction has been supplied with transverse
boundary directions.

\smallskip\noindent\emph{Spatial elliptic regularity.}
There are two necessary regularity steps.  First, every zero in this
low-order neighborhood has all weighted spatial derivatives.  In a
rescaled ball, the harmonic-gauge Einstein equation is a uniformly
elliptic quasilinear system with smooth reference coefficients.
Its nonhomogeneous reference error and its derivatives are $O(x_0^3)$.
The usual differentiated Schauder estimates, applied to the remainder
and divided by $x_0^d$, bootstrap the remainder at each successive
order with uniform constants.  The leading coefficients are uniformly
elliptic and their lower-order norms are controlled by the preceding
step.  Covering the end by these rescaled balls and handling the
compact interior proves $h\in\mathcal H^\infty_d$.

\smallskip\noindent\emph{Smooth parameter dependence at higher orders.}
Second, fix any one point of the finite-dimensional zero set and any
one higher order $j$.  Its augmented derivative remains an
isomorphism at the original order.  Uniformly degenerate elliptic
regularity upgrades its inverse to order $j$: the finite-dimensional
boundary columns are spatially smooth, and the differential equation
then upgrades the tensor component of a low-order inverse solution.
The higher-order implicit-function theorem at this point, with
the same projection $\Pi$, and uniqueness in the original low-order
chart prove smooth parameter dependence into
$\mathcal H^{j,\alpha}_d$ on a neighborhood of that point.  In
particular, for every preassigned order $r\ge k$ and every finite
parameter derivative order $m$, each compact convex finite-dimensional
subchart $Q$ satisfies
\begin{equation}\label{eq:uniform-finite-family-regularity}
\sup_{z\in Q}\|D_z^m h(z)\|_{r,\alpha,d}\le C_{r,m,Q},
\qquad
\|h(z)-h(z')\|_{r,\alpha,d}\le C_{r,1,Q}|z-z'|.
\end{equation}
The second inequality follows from the first by the mean-value
theorem.  The orders and constants are fixed one at a time, as
required by the later finite-dimensional overlap construction.

\smallskip\noindent\emph{The invariant problem.}
For an invariant construction choose $x$, the collar, and the cutoff
invariant.  All displayed operators are then equivariant under the
fixed action.  Restriction to invariant Banach subspaces commutes
with the elliptic inverses and the implicit-function theorem.  The
same argument applies.
\end{proof}

A finite-dimensional solution manifold would not yet describe the
geometric problem unless nearby genuine fillings actually entered its
gauge.  The next lemma checks this in the \emph{same} weighted norm
used for the implicit-function theorem.  

\begin{lemma}[Quantitative local capture in the actual weighted norm]
\label{lem:hg-local-capture}
Let $a_i\to0$ in a finite smooth boundary family, and let
$g_i=H(a_i)+h_i$ be Einstein metrics with
$h_i\to0$ in $\mathcal H^{k,\alpha}_d$.  The harmonic-map gauges
of Lemma~\ref{lem:hg-nonlinear-gauge} give representatives
\[
 g'_i=(F_i)_*g_i=H(a_i)+h'_i,
 \qquad h'_i\longrightarrow0\quad\hbox{in }
 \mathcal H^{k,\alpha}_d.
 \tag{HG.21}
\]
Moreover $h'_i\to0$ at every preassigned finite higher weighted
order.  When the finite boundary family is transverse, these
representatives lie in its finite-dimensional zero-set chart and
converge in the finite-dimensional chart topology.
\end{lemma}

\begin{proof}
The harmonic-map implicit-function theorem gives displacement
$V_i\to0$ in $\mathcal H^{k+1,\alpha}_d$.
A moving pullback of a general fixed-order rough tensor is not
declared continuous into the same order.  Instead use the exact
decomposition
\[
 h'_i=(F_i)_*h_i+
       \bigl((F_i)_*H(a_i)-H(a_i)\bigr).
 \tag{HG.22}
\]
The first term satisfies a uniform norm bound
\[
 \|(F_i)_*h_i\|_{k,\alpha,d}
 \le C\|h_i\|_{k,\alpha,d}.
\]
This is boundedness of each pullback, with uniform bounds for its
Jacobian and inverse in the rescaled charts.  The second term has the bound
\[
 \|(F_i)_*H(a_i)-H(a_i)\|_{k,\alpha,d}
 \le C\|V_i\|_{k+1,\alpha,d}.
\]
Here the tensor being moved is the spatially smooth reference in a
finite smooth family; its extra spatial derivatives make this
estimate a direct Taylor and chain-rule calculation.  The leading
boundary tensor changes only in weight $d$ because the maps have
identity three-jet and their displacement has weight $d$.  These bounds prove (HG.21).

Each $g'_i$ is an Einstein zero of (HG.5), and we have already
proved its convergence at the fixed low weighted order.  To obtain
higher-order convergence we use the finite-dimensional zero set.

Choose finitely many additional smooth boundary directions spanning
the cokernel at the central solution, as in
Lemma~\ref{lem:hg-finite-reductions}, and give them auxiliary
parameters.  This does not change the original solutions: each
$g'_i$ occurs with those auxiliary parameters equal to zero.  The
augmented equation has a transverse local zero chart, whose
coordinates $z_i$ tend to those of the central metric because of
the low-order convergence.  On a compact convex subchart $Q$, the
higher-order parameter estimate
\eqref{eq:uniform-finite-family-regularity} reads
\begin{equation}
 \|h(z)-h(z')\|_{r,\alpha,d}\le C_{r,Q}|z-z'|,
 \qquad r\ge k,\quad z,z'\in Q.
 \tag{HG.23}
\end{equation}
Thus the remainders converge at every preassigned higher weighted
order.  The inverse of the local
low-order chart is continuous, so convergence in that norm also gives
convergence in the finite-dimensional chart topology.

\end{proof}

\subsubsection{Decay of comparison isometries}

Two gauge representatives of the same marked filling are related by
an isometry fixing the conformal boundary pointwise.  To compare their
finite-dimensional charts, this isometry must decay at the same rate
as the metric remainders.  Ordinary boundary regularity alone gives
a weaker statement; the harmonic-map equation will improve it to the
chosen weight.  The next two lemmas record the qualitative and
quantitative forms of that improvement.

Here we start with an interior-smooth marked isometry extending
$C^{3,\theta}$ to the compactification for some $0<\theta<1/2$.
The canonical negative-scalar compactification and its connection
transformation law provide this regularity and its compactness in the
separate topology argument.  The following lemmas give the exact
weighted regularity and smallness required for differentiable overlaps.

The boundary three-jet is forced by the matching prepared metric
jets, and the subsequent decay improvement is an elliptic consequence.

\begin{lemma}[Vector inverses for prepared conormal metrics]
\label{lem:conormal-vector-inverse}
Let $G$ be an interior-smooth Poincar\'e--Einstein metric on the
interior of a compact four-manifold, with $\operatorname{Ric}_G=-3G$.
Assume that $x^2G$ is a classical $C^{2,\theta}$ compactification
for some $0<\theta<1$, that $G$ is asymptotically hyperbolic, and that
its coefficients and their inverses have uniformly bounded geometric
derivatives of every finite order in the rescaled charts.
For every $-1<\delta<4$, every integer $r\ge2$, and $0<\alpha<1$,
the operator
\[
 Q_G=\nabla_G^*\nabla_G+3:
 \mathcal H^{r,\alpha}_\delta(TX)
 \longrightarrow\mathcal H^{r-2,\alpha}_\delta(TX)
\]
is an isomorphism. Its inverse satisfies
\[
 \|V\|_{r,\alpha,\delta}
 \le C_{G,r,\alpha,\delta}
       \|Q_GV\|_{r-2,\alpha,\delta}.
\]
For a compact finite-dimensional family continuous in all the stated
geometric coefficient norms, with uniform compactification and
ellipticity bounds, these constants may be chosen uniformly, at each
fixed $r,\alpha,\delta$.
\end{lemma}

\begin{proof}
Choose $0<\alpha_0<\min\{\theta,\alpha\}$. First work only at
the low order allowed by the classical compactification. Lee's
Theorem~C and Proposition~E \cite{Lee2006}, applied to one-forms and then transferred
to vectors by the metric, give a Fredholm map of index zero
\[
 Q_G:\mathcal H^{2,\alpha_0}_\delta
       \longrightarrow\mathcal H^{0,\alpha_0}_\delta
       \qquad(-1<\delta<4).
\]
Here $l=2$, $\beta=\theta$ in the theorem, so
$2<2+\alpha_0\le2+\theta$ is within its asserted regularity range.
The indicial radius is
$\sqrt{3^2/4+1+3}=5/2$, with central norm weight $3/2$.
The exterior coercivity follows even globally from
\[
 \langle Q_GV,V\rangle_{L^2}
      =\|\nabla V\|_2^2+3\|V\|_2^2
      \quad(V\in C_c^\infty(TX)).
\]
The Fredholm kernel equals the $L^2$ kernel. Cutoff integration of
this identity shows that the latter is zero. Index zero therefore
makes the displayed low-order map an isomorphism.

Given $f\in\mathcal H^{r-2,\alpha}_\delta$, the embedding into
$\mathcal H^{0,\alpha_0}_\delta$ is bounded. Let $V$ be the unique
low-order solution of $Q_GV=f$. On nested uniform charts $B_1\Subset
B_2$ centered at height $x_0$, the ordinary interior Schauder estimate
for this vector system gives
\[
 \|V\|_{C^{r,\alpha}(B_1)}
 \le C_r\bigl(
       \|f\|_{C^{r-2,\alpha}(B_2)}
       +\|V\|_{C^0(B_2)}\bigr).
\]
At $r=2$ this is also the usual regularity implication for the already
constructed $C^{2,\alpha_0}$ solution; the target exponent is
$\alpha$, not $\alpha_0$. At higher $r$ differentiate the equation
and induct. The coefficient assumptions give the required uniform
$C^{r-2,\alpha}$ bounds at every fixed order. For $Q_G$ on vectors
these use at most $r$ geometric derivatives of the metric, which are
available by assumption.

In each such chart $x/x_0$ and its inverse have bounded derivatives.
Dividing by $x_0^\delta$, taking the supremum over a uniform cover,
and including finitely many interior charts yield
\[
 \|V\|_{r,\alpha,\delta}
 \le C_r\bigl(
       \|f\|_{r-2,\alpha,\delta}
       +\|V\|_{0,\delta}\bigr).
\]
The low-order inverse bounds the final term by
$C\|f\|_{0,\alpha_0,\delta}$ and hence by
$C'\|f\|_{r-2,\alpha,\delta}$. This proves existence and the
asserted bound at the original requested exponent $\alpha$.
Uniqueness follows by the embedding into the low-order space.

For a compact parameter family use a fixed defining function and
background weighted atlas. The low-order operators depend
continuously in bounded-operator norm. Their inverses are locally
uniform by the Neumann-series identity, and a finite cover of the
parameter set makes the low-order bound uniform. The fixed-order
Schauder constants are uniform by the assumed coefficient and
ellipticity bounds. This proves the final statement.
\end{proof}

\begin{lemma}[Decay improvement for marked isometries]
\label{lem:hg-isometry-decay}
Let
\[
 g_i=H_i(q)+h_i,\qquad h_i\in\mathcal H^\infty_d,
 \quad \operatorname{Ric}_{g_i}=-3g_i,
 \quad \tau(g_i,H_i(q),\operatorname{Id})=0,
\]
where both references have compactified leading form
$dx^2+q-x^2P_q+O(x^3)$ in the same fixed collar.
If $F:(X,g_1)\to(X,g_2)$ is a boundary-identity isometry extending
$C^{3,\theta}$ for some $0<\theta<d-2$, then it has identity
three-jet.  Near the boundary its displacement vector belongs to
$\mathcal H^\infty_d(TX)$.
Consequently fixed such isometries and their inverses act boundedly
on each weighted tensor scale and preserve the weight-$d$ affine
metric models.  The same conclusion holds for any two all-orders
conormal PE representatives with this prepared boundary jet.
\end{lemma}

\begin{proof}
The argument has two parts: the shared boundary metric determines
the first three jets of the isometry, and the harmonic-map equation
improves its decay from the resulting weak weight to $d$.

\smallskip\noindent\emph{The boundary jets.}
Write $\bar g_i=x^2g_i$.  The isometry equation is
\[
 F^*\bar g_2=(x\circ F/x)^2\bar g_1.
 \tag{HG.15}
\]
Tangential restriction at $x=0$ makes the positive normal derivative
of $F$ equal to one, since both boundary metrics are $q$.
The mixed equation then kills its normal tangential component.
Thus $dF=\operatorname{Id}$ at the boundary.  Write the remaining
second normal jet as $F^0=x+ax^2+o(x^2)$,
$F^a=y^a+b^ax^2+o(x^2)$.  The $00$ and $0a$ components of
(HG.15) at order $x$ give $4a=2a$ and $2b_a=0$.
Tangential and mixed second derivatives are fixed by the boundary
value and first jet; hence $j^2F=j^2\operatorname{Id}$.

Write the remaining third normal jet as
\[
 F^0=x+A(y)x^3+o(x^3),\qquad
 F^a=y^a+B^a(y)x^3+o(x^3).
\]
Both compact metrics have $g_{00}=1+O(x^d)$ and
$g_{0a}=O(x^d)$, and their tangential second jet is $-P_q$.
The $00$ coefficient of (HG.15) at order $x^2$ is
$6A=2A$, so $A=0$.  Its $0a$ coefficient gives $3B_a=0$.
All third derivatives containing a tangential derivative are already
fixed by the identity second jet on the boundary.  Therefore
$j^3F=j^3\operatorname{Id}$, and Taylor's theorem gives
\[
 D^j(F-\operatorname{Id})=O(x^{3+\theta-j}),\qquad 0\le j\le3.
 \tag{HG.16}
\]

\smallskip\noindent\emph{Improvement to the prepared weight.}
Near infinity define $V=\exp_{g_1}^{-1}F$ and choose
$2<\delta<2+\theta<d$.  In rescaled charts (HG.16) puts
$V$ in weight $\delta$, initially through two derivatives with
any fixed H\"older exponent $\alpha<1$: the bound on the third
ordinary derivative controls the rescaled second-derivative
H\"older seminorm.  Since the identity on $(X,g_2)$ is harmonic
into $(X,H_2(q))$, naturality makes
$F:(X,g_1)\to(X,H_2(q))$ harmonic.  Interior Schauder bootstrap
in the rescaled charts gives $V\in\mathcal H^\infty_\delta$.

Expand that equation at the identity using $g_1$ exponential
coordinates and parallel transport.  For target $g_1$ the
linearization is $-(\nabla^*\nabla-\operatorname{Ric}_{g_1})$.
The actual target differs from $g_1$ by a weight-$d$ tensor, so
\[
 (\nabla^*\nabla+3)V
 =T(g_1,H_2(q))+E_d(V)+Q(V).
 \tag{HG.17}
\]
Every coefficient of the at-most-second-order operator $E_d$ has
weight $d$, and $Q$ is at least quadratic in the two-jet of $V$.
This is direct Taylor expansion of (HG.9); all geometric derivatives
of the coefficients are bounded.  The terms on the right have
weights $d$, $d+\delta$, and at least $2\delta>d$, respectively.
Thus the right side lies in $\mathcal H^\infty_d$.

Cut $V$ off toward the compact interior.  Its image under
$\nabla^*\nabla+3$ lies in $\mathcal H^\infty_d$; cutoff errors
have compact support.  Lemma~\ref{lem:conormal-vector-inverse}
makes the vector operator an isomorphism at weights $\delta$ and $d$,
even when the prepared compactification is initially only
$C^{2,\theta}$. Its two inverse solutions agree by uniqueness
in weight $\delta$. Hence $V\in\mathcal H^\infty_d$.
The same argument applies to $F^{-1}$.  For arbitrary prepared
conormal PE representatives, use the actual fixed target $g_2$
instead of $H_2(q)$.  An isometry $F:(X,g_1)\to(X,g_2)$ is
harmonic, the target has bounded derivatives of every order in
rescaled charts, and $g_2-g_1$ has weight $d$.  The identical
expansion and two-weight inverse argument therefore apply.  This
last argument requires no moving-rough-coefficient differentiability:
its target is fixed and conormal smooth, and it is used only for
spatial regularity of the given map.  Uniform-chart chain rules
prove bounded fixed pullback on each weighted scale.  The prepared
second metric jet and the improved displacement preserve the affine
metric model.  
\end{proof}

\begin{lemma}[Quantitative upgrade for nearby comparison maps]
\label{lem:hg-quantitative-isometries}
Suppose $g$ and $H$ range in finite smooth weighted families as above,
and $F:(X,g)\to(X,H)$ is harmonic, has identity three-jet, and
converges to the identity in ordinary $C^{3,\theta}$ for some
$0<\theta<d-2$.  Suppose $T(g,H)\to0$ in
$\mathcal H^{k-1,\alpha}_d$.  Then its displacement converges to
zero in $\mathcal H^{k+1,\alpha}_d$.
\end{lemma}

\begin{proof}
The two weights have different roles.  The lower one is obtained
from ordinary compactified regularity; the upper one is the weight
at which the harmonic-map implicit-function theorem was solved.
Uniqueness will identify the resulting maps.

Fix $2<\delta<2+\theta$.  For a comparison map sufficiently close
to the identity, put $V=\exp_g^{-1}F$.  The identity three-jet and
Taylor's theorem give, in uniform rescaled charts,
\[
 \|V\|_{2,\alpha,\delta}
 \le C\|F-\operatorname{Id}\|_{C^{3,\theta}(\bar X)}.
 \tag{HG.18}
\]
Indeed ordinary displacement is $O(x^{3+\theta})$; converting to a
bounded complete-metric vector divides it by one power of $x$.
The third ordinary derivative controls the rescaled H\"older
seminorm of its second derivative.  The strict inequality
$\delta<2+\theta$ gives convergence to zero at weight $\delta$.

The map $F:(X,g)\to(X,H)$ is harmonic.  Its equation has the rough
domain coefficients evaluated at the fixed point, and the target
coefficients evaluated at $F(z)$.  The latter range in a smooth
finite-dimensional family.  Differentiated rescaled Schauder
estimates, beginning with (HG.18), give
\[
 \|V\|_{k+1,\alpha,\delta}
 \le C_k\Bigl(\|F-\operatorname{Id}\|_{C^{3,\theta}(\bar X)}
       +\|T(g,H)\|_{k-1,\alpha,\delta}\Bigr)
 \longrightarrow0.
 \tag{HG.19}
\]
There is no constant term: the equation for $V$ differs from that
of the identity only by terms containing $V$, its derivatives, or
the identity-map tension.  The constants are uniform on a smaller
finite smooth parameter neighborhood.

The harmonic-map theorem also supplies a unique small solution
$V_d$ at weight $d$.  Since $d>\delta$, this is a solution at the
weaker weight.  Estimate (HG.19) places the given comparison map in
the uniqueness neighborhood at that weight. Hence the two harmonic
maps coincide. Express both displacement vectors using the same
exponential parametrization; the change from the fixed-reference
exponential to the $g$-exponential is uniformly smooth on these finite
families and preserves the weighted norm. The quantitative gauge
estimate \eqref{eq:uniform-small-harmonic-gauge} therefore gives
\[
 \|V\|_{k+1,\alpha,d}
 \le C\|T(g,H)\|_{k-1,\alpha,d}.
 \tag{HG.20}
\]
By hypothesis the right-hand side tends to zero, completing the
claim.  Smooth parameter
dependence on finite solution families follows from their
higher-order derivative reserves.
\end{proof}

\subsubsection{Smooth changes of gauge on finite-dimensional solution sets}

We now have enough regularity to compare two local constructions.
It is important that the smoothness claimed here is smoothness
\emph{between finite-dimensional solution charts}.  We make no
statement about the full diffeomorphism action on a rough Banach
space.  We use a fixed comparison isometry at the central
point and solve just one additional harmonic-map equation for the
nearby family.

\begin{proposition}[Differentiable gauge changes on finite zero sets]
\label{prop:hg-overlap}
Let $Z_i,Z_j$ be two finite-dimensional zero-set charts from
Lemma~\ref{lem:hg-finite-reductions}, for the same finite smooth
boundary family $q(a)$.  At a pair of marked-isometric points,
the induced change of representatives is smooth between the finite
charts.  It preserves $a$, and all triple-overlap identities hold
exactly.  If the marked quotient has the local properness property
that nearby equivalent representatives have comparison isometries
close to the given comparison isometry, these maps give actual
compatible local charts on that quotient.
\end{proposition}

\begin{proof}
A fixed comparison isometry identifies the central metrics.
We then solve the harmonic-map equation for the \emph{finite moving
family}, using its spatial derivative reserves.  The uniqueness of
marked isometries makes the resulting transition canonical.

Let $F_0$ be the marked isometry at the chosen overlap point.
By Lemma~\ref{lem:hg-isometry-decay}, fixed pushforward by $F_0$
preserves all weighted scales.  Push the $i$-family forward by
$F_0$ once; its central metric is now the selected metric of the
$j$-family.  Solve (HG.9) from this moving domain family into
$H_j(a)$, for a map near the identity.  Its derivative in the map
variable is invertible: throughout the sufficiently small $j$-chart
it is a small bounded-operator perturbation of the invertible
operator (HG.11) at the central reference Einstein metric.
The nonlinear harmonic-map implicit-function theorem therefore
provides the unique small map $F(z)$.

By Lemma~\ref{lem:hg-finite-reductions}, the domain family has every
finite weighted spatial derivative reserve.  For any prescribed
integer $r$, choose $r+2$ additional spatial derivatives.  Ordinary
chain rules in the uniform charts, including
\[
 \partial_b F^{-1}
 =-(dF)^{-1}(\partial_bF)\circ F^{-1},
\]
show that inversion and pushforward are $C^r$ into the chosen lower
weighted tensor norm on these families.  One unused spatial
derivative makes the H\"older remainder estimates continuous; the
extra derivative in the stated reserve covers the metric tensor
Jacobian factors.  The bounds are uniform in the rescaled charts
and preserve weight $d$, since $F-\operatorname{Id}$ has that
geometric weight.  This is a finite-regularity chain-rule argument
with a finite derivative reserve, not differentiability of
$\operatorname{Diff}^{k+1,\alpha}\times\operatorname{Met}^{k,\alpha}$
into $\operatorname{Met}^{k,\alpha}$.

The pushed-forward metric is Einstein and lies in harmonic identity
gauge relative to $H_j(a)$, hence belongs to $Z_j$ after shrinking
the overlap.  Composing with its finite-dimensional chart gives a
$C^r$ transition for every $r$, hence a smooth transition.
Local properness ensures that any other representative of the same
nearby marked class uses a nearby comparison isometry.  That
isometry is harmonic into $H_j(a)$ and is therefore the same root
by the implicit-function uniqueness: more explicitly, after centering
at $F_0$, local properness first gives ordinary $C^{3,\theta}$
closeness, and Lemma~\ref{lem:hg-quantitative-isometries} upgrades it
to smallness in the weight-$d$ norm in which that uniqueness holds.
This proves that the
transition is the transition of quotient charts, not only a map
between selected families.

The inverse transition is obtained by reversing the roles.  On a
triple overlap, both the direct comparison map and the composition
are PE isometries with the same boundary identity.  Their quotient
is a boundary-trivial PE isometry, hence the identity by the
boundary-identity rigidity lemma.  Thus inverses and cocycles hold
exactly, independently of choices of gauge construction.
\end{proof}

\subsubsection{Keeping the prescribed group action}

So far we have worked with one Einstein metric, without using the
finite group.  At an equivariant filling the isometries give an
action on the tensor spaces and commute with the equations.  Choosing
invariant references allows us to restrict the construction to
invariant tensors.  The additional point is that a nearby \emph{fixed
marked class} has a genuinely invariant representative in this
particular gauge.

At a fixed marked Einstein class let $I_\gamma$ be its coherent
finite isometric action.  Choose $x$, the collar, and the references
$H(a)$ invariant.  Formula (HG.9) is natural, so uniqueness gives
\[
 F(I_\gamma z)=I_\gamma F(z)
\]
whenever the domain metric is invariant.  The finite-dimensional
reduction in invariant tensors therefore parametrizes invariant
solutions, including degenerate ones.

There is also a direct local fixed-locus check:  if a nearby
marked class is fixed by the prescribed boundary action, its
isometric extensions are close, by marked-isometry properness, to
the central $I_\gamma$.  Compose each extension with
$I_\gamma^{-1}$.  Both the transformed metric and the original one
are in the same harmonic reference gauge, and the comparison map is
near the identity.  Harmonic-map uniqueness forces that comparison
map to be the identity.  Thus the harmonic-gauge representative
itself is $I_\gamma$-invariant.  Conversely an invariant
representative certainly defines a fixed marked class.

On overlaps the comparison isometry conjugates the two prescribed
isometric extensions: both sides are isometries with the same
boundary action, so boundary-identity rigidity identifies them.
Consequently the invariant finite reductions have exactly the same
smooth transition and cocycle properties as the ambient reductions.

\subsubsection{What the local construction provides}

The distinction between a local differential construction and its
global topology is worth keeping clear.  The arguments above provide
smooth finite-dimensional solution families and smooth transitions.
They do not yet prove that every relevant filling belongs to one of
these neighborhoods or that the solution space is Hausdorff and
compact.  We establish those properties separately below.

In particular, the analytic results above resolve the moving-pullback
problem for
finite smooth parameter families.  What remains is
the separately proved geometric marked-isometry properness and the
identification of compactness with the weighted local topology.
Once those are supplied, the finite reductions form a genuine
Hausdorff finite-dimensional continuation manifold near the compact
unperturbed zero set.  

\subsection{The invariant obstruction and the hyperbolic seed}
\label{sec:finite-obstruction-seed}

The finite-dimensional parameters are useful only if they can remove
\emph{every} obstruction to solving the Einstein equation.  The tensor
operator is formally self-adjoint, so its obstruction space can be
identified with the space of square-integrable Jacobi fields.
The point which matters for our application is that this
identification respects the finite group action.  We shall then
use a boundary Green formula to show that variations of the
prescribed boundary metric detect every invariant Jacobi field.

Throughout this subsection $G$ is a smoothly conformally compact Einstein
metric on $B^4$, $\operatorname{Ric}_G=-3G$, and $q$ is a smooth
representative of its conformal infinity.  If a finite group acts
isometrically, the geodesic collar, the positive interior extension of
its defining function, and all cutoffs are chosen invariant.  This is
possible by uniqueness of the geodesic defining function for the
invariant metric $q$ and by averaging away from that collar.

We first dispose of a basic ambiguity.  The boundary marking must
identify an isometry uniquely: otherwise two different extensions of
one boundary transformation could give incompatible group laws.
For Einstein fillings this ambiguity does not occur.

\begin{lemma}[Boundary-identity rigidity]\label{lem:boundary-identity}
A PE isometry extending as a $C^1$ diffeomorphism of the compactification,
equal to the identity on its boundary, is the identity.
This applies both to smooth compactified
representatives and to the prepared weighted representatives used in
the continuation construction.
\end{lemma}

\begin{proof}
First take a smoothly compactified representative.  The compactified
isometry identity and the identity tangential derivative force the
normal scaling to be one.  In particular the conformal ratio used
in the following compactification has continuous boundary value one.
The canonical
negative-scalar compactification in
Lemma~\ref{lem:negative-scalar-compactification} makes a marked PE
isometry an isometry of a compact smooth metric, and hence smooth up
to the boundary by boundary harmonic coordinates and the connection
equation.  Equivalently one may use the usual boundary regularity of
PE isometries.  Let $x$ be the geodesic defining function for $q$.
Then $x\circ F$ is the geodesic defining function with the same
boundary representative: the boundary marking is the identity and
the compactified tangential metric forces its normal scaling to be
one.  Uniqueness gives $x\circ F=x$ near the boundary.  Therefore
$F$ preserves $x^2G$ and its inward normal field $\nabla x$.
Since it fixes every starting point of the normal flow, it is the
identity on a collar.  An isometry of a connected manifold which is
the identity on an open set is the identity everywhere, by uniqueness
of geodesics continued along paths.

For a prepared weighted representative use the admissible smooth
representative supplied by the classical boundary-regularity and
global-collar-extension argument in the compactness bridge.
Conjugation preserves the boundary marking.  The same argument then
applies and conjugates back to the original representative.
\end{proof}

The following lemma gives the actual equivariant index calculation.
An ambient index equal to zero would not be sufficient by itself:
passing to fixed vectors can change the dimensions of the kernel
and cokernel unless their \emph{representations} are identified.
Here self-adjointness provides precisely that identification.

\begin{lemma}[The weighted kernel and obstruction representation]
\label{lem:finite-invariant-linear-index}
Let $d=5/2$, $k\ge4$, and $0<\alpha<1$.  The operator
\[
 L_G=\tfrac12\nabla^*\nabla-\mathring R:
 \mathcal H^{k,\alpha}_d(S^2T^*X)
 \longrightarrow\mathcal H^{k-2,\alpha}_d(S^2T^*X)
\]
is Fredholm.  Its kernel is the finite-dimensional $L^2$ Jacobi space
$K_G$, and its range is precisely the annihilator of $K_G$ under the
$L^2(G)$ pairing.  Every $\kappa\in K_G$ is trace-free and
divergence-free.  For a finite isometric group $I\Gamma$, the invariant
restriction has
\[
 \ker L_G^{I\Gamma}=K_G^{I\Gamma},\qquad
 \operatorname{coker}L_G^{I\Gamma}\cong K_G^{I\Gamma}.
 \tag{OB.1}
\]
In particular its Fredholm index is zero as an equality of the actual
kernel and obstruction representations.
\end{lemma}

\begin{proof}
Lee's weighted Fredholm theorem applies directly to the full uniform
geometric H\"older spaces in this lemma
\cite[Theorem~C and Proposition~D]{Lee2006}.  Since
$2L_G=\Delta_L+6$, the real-hyperbolic tensor indicial radius is
$3/2$ around central weight $3/2$, giving the interval $(0,3)$.
The weight $d=5/2$ lies strictly inside it.  The theorem identifies
the kernel with the $L^2$ Jacobi space and, using self-adjointness,
the range with its $L^2$ annihilator.  Thus both kernel and
obstruction retain their full representations under the isometry
group.

This choice of Fredholm reference also resolves a point about
function spaces.  The full bounded uniform H\"older spaces defined
in (HG.1) need not equal the closure of smooth bounded sections in
the same H\"older norm. Lee's theorem is invoked in its own full geometric spaces.

Taking the trace of $L_G\kappa=0$ gives
$(\nabla^*\nabla+6)\operatorname{tr}\kappa=0$.
The contracted Bianchi identity for its linearization gives
$\beta_G L_G\kappa=\tfrac12(\nabla^*\nabla+3)\beta_G\kappa$.
Weighted elliptic regularity supplies derivatives at the same decaying
weight, so these fields and their first derivatives are in $L^2$.
The corresponding coercive integrations imply
$\operatorname{tr}\kappa=0$ and $\beta_G\kappa=0$, hence
$\delta_G\kappa=0$.  Integrations may first be made with compact
cutoffs of bounded geometric derivative; $d>3/2$ makes the error
terms tend to zero.

Every isometry commutes with $L_G$, preserves $K_G$, and preserves the
$L^2$ pairing.  Consequently the obstruction identification is
equivariant, rather than just an equality of dimensions.  More
explicitly, let $\mathsf A=|\Gamma|^{-1}\sum I_\gamma^*$ be
averaging.  It is a bounded projection on both weighted spaces.  If
$f$ is invariant and lies in the full range, $L_Gu=f$ implies
$L_G\mathsf Au=f$.  For invariant $f$, pairing with an arbitrary
$\kappa\in K_G$ is the same as pairing with $\mathsf A\kappa$.
Thus the invariant range is exactly the annihilator of
$K_G^{I\Gamma}$, which proves both statements in (OB.1).
\end{proof}

At the beginning of the boundary path we have the hyperbolic
metric.  To start parity from a single regular solution, we need to
know that it has no decaying Jacobi fields.  On hyperbolic space this
follows from a simple spectral estimate.

\begin{lemma}[Coercivity at the hyperbolic metric]
\label{lem:finite-hyperbolic-linear-seed}
For the hyperbolic metric $G_{\rm hyp}$ on $B^4$, $K_{G_{\rm hyp}}=0$.
The fixed-boundary weighted metric operator, and every finite-group
invariant restriction of it, are therefore isomorphisms.
\end{lemma}

\begin{proof}
The scalar spectral inequality on $\mathbb H^4$ is
\[
 \int|df|^2\ge\frac94\int f^2,
 \qquad f\in C_c^\infty(\mathbb H^4).
\]
Explicitly, choose a Busemann function $b$ with
$|db|=1$ and $\Delta b=3$.  Expand the integral of
$|df+\tfrac32f\,db|^2$ and integrate its cross term to obtain
$0\le\int|df|^2-\tfrac94\int f^2$.
Apply this inequality and Kato's inequality to compact cutoffs of
$|\kappa|$, passing to the limit in $L^2$.
For sectional curvature $-1$ the curvature action on a trace-free
symmetric tensor is $\mathring R\kappa=\kappa$.  Hence
\[
 2\langle L\kappa,\kappa\rangle
 =\|\nabla\kappa\|_2^2-2\|\kappa\|_2^2
 \ge\frac14\|\kappa\|_2^2.
\]
A Jacobi field in the kernel is trace-free by the preceding lemma,
so it vanishes.  Fredholm range identification then gives the stated
isomorphisms.  The same estimate proves invertibility of the
hyperbolic normal operator used in the boundary parametrix below.
\end{proof}

\subsubsection{The boundary value of a decaying Jacobi field}
\label{subsec:strong-kernel-boundary-trace}

The obstruction is an interior tensor, whereas our free parameters
are boundary metrics.  A Green formula can relate them only if the
Jacobi tensor has a sufficiently definite leading coefficient at
infinity.  A weak asymptotic statement is not enough: the boundary
term involves both that coefficient and its normal derivative.
The next lemma supplies strong asymptotics in every tangential
$C^m$ norm.  We isolate the indicial root responsible for
the leading term and then solve a radial equation for the coefficient.

\begin{lemma}[Strong leading trace of an $L^2$ Jacobi tensor]
\label{lem:strong-jacobi-boundary-trace}
Let $G=x^{-2}(dx^2+q_x)$ be a smooth conformally compact PE metric
on a four-manifold, with smooth geodesic compactification and smooth
compact boundary.  Put
\[
 L_G=\tfrac12\nabla^*\nabla-\mathring R.
\]
If $\kappa\in L^2$ and $L_G\kappa=0$, then there is a smooth
tangential trace-free tensor $\kappa_{(3)}$ on the boundary such that,
in a smooth zero-frame (orthonormal after a smooth bundle change),
\begin{equation}\label{eq:strong-jacobi-trace}
 u=x^3\kappa_{(3)}+O_b(x^{3+\epsilon}),\qquad
 D u=3x^3\kappa_{(3)}+O_b(x^{3+\epsilon}),
 \qquad D=x\partial_x,
\end{equation}
for every $0<\epsilon<1$.  Here $u$ denotes the components of
$\kappa$ in the covector frame $dx/x,dy^a/x$ (equivalently the
ordinary compactified tensor $x^2\kappa$), and $O_b$ controls all
compositions of $x\partial_x$ and ordinary tangential derivatives.
The normal components are $O_b(x^{3+\epsilon})$.  In particular,
$x^{-3}u_T\to\kappa_{(3)}$ in every boundary $C^m$ norm, and the
corresponding radial derivative has the asserted strong limit.
\end{lemma}

\begin{proof}
There are three steps.  A parametrix first gives decay with all
tangential derivatives.  The indicial matrix identifies which
components can occur at leading order.  A radial integration then
produces an actual smooth boundary coefficient, with the normal
derivative needed for Green's formula.

\paragraph{Conormal regularity from a parametrix.}
Biquard--Mazzeo \cite[Definition~8 and proof of Theorem~13, pp.~1437--1438 and 1441--1442]{BiquardMazzeo2011} construct a generalized inverse whose
kernel projection is a very residual operator, conormal on the
ordinary product compactification.  Applied to a smooth
asymptotically real-hyperbolic metric, their conclusion is that a
nullspace element is in $\mathcal A^\tau$, where $\tau$ is the
smallest positive indicial root.  Their definition of
$\mathcal A^\tau$ controls \emph{ordinary tangential derivatives};
it is stronger than the rescaled-derivative reserve used in the
nonlinear gauge section.

The theorem's full ellipticity requirement must be checked before its
kernel conclusion can be used.  The principal symbol of $2L_G$ is
the positive scalar Laplace symbol on symmetric two-tensors.  Its
normal operator is the corresponding operator on hyperbolic four-space;
it is invertible on $L^2$ by the following elementary coercivity
calculation.  For a compactly supported tensor $h$ on
hyperbolic space,
\[
 \langle\mathring Rh,h\rangle=|h|^2-(\operatorname{tr}h)^2
 \le |h|^2.
\]
A Busemann function with $|\nabla t|=1$, $\Delta t=3$ gives
\[
 \int|\nabla f|^2\ge\frac94\int f^2
\]
by expanding $\int|\nabla f+(3/2)f\nabla t|^2$.  Kato's
inequality therefore yields
\[
 \int\langle L_{\mathbb H^4}h,h\rangle
 \ge\left(\frac98-1\right)\int|h|^2
 =\frac18\int|h|^2.
\]
The Friedrichs realization is consequently invertible on $L^2$.
This verifies the normal-operator invertibility required by the theorem.

We compute below that $\tau=3$.  Since the kernel projection is
very residual, applying it to $\kappa$ gives
\begin{equation}\label{eq:jacobi-b-conormal}
 u\in\mathcal A^3:
 \quad D^j\partial_y^\alpha u=O(x^3)
 \quad\hbox{for every }j,\alpha.
\end{equation}
Equivalently this follows from a left parametrix
$\mathcal G L=I-Q$: for $L\kappa=0$, $\kappa=Q\kappa$.
We use a bounded kernel projection.

\paragraph{The indicial roots.}
The asymptotic equation separates according to the normal, mixed,
and tangential components of a symmetric tensor.  The tangential
trace-free component will be the only one capable of contributing
at the leading order $x^3$.
Use the hyperbolic orthonormal frame
$e_0=\partial_t$, $e_a=e^{-t}\partial_{y^a}$, with $x=e^{-t}$.
At the indicial level
\[
 \nabla_{e_0}e_a=0,\qquad
 \nabla_{e_a}e_0=e_a,\qquad
 \nabla_{e_a}e_b=-\delta_{ab}e_0.
\]
Write a constant symmetric tensor as
\[
 U=\begin{pmatrix}a&v^T\\v&A\end{pmatrix}.
\]
Let $J_a$ be the skew matrix with $J_ae_0=e_a$ and
$J_ae_a=-e_0$.  The connection on symmetric tensors acts by
$[J_a,U]$, and direct matrix multiplication gives
\[
 -\sum_a[J_a,[J_a,U]]
 =\begin{pmatrix}
     6a-2\operatorname{tr}A&6v^T\\
     6v&2A-2aI
   \end{pmatrix}.
\]
Using $\mathring RU=U-(\operatorname{tr}U)I$ in constant
curvature $-1$, the indicial action of $2L$ on $x^sU$ is
\begin{align*}
 a&\longmapsto [s(3-s)+6]a,\\
 v&\longmapsto [s(3-s)+4]v,\\
 A&\longmapsto s(3-s)A+2(\operatorname{tr}A)I.
\end{align*}
Thus the tangential trace-free block has roots $0,3$, the mixed
block has roots $-1,4$, and both scalar blocks have roots
$(3\pm\sqrt{33})/2$.  The root closest to $3/2$ on the positive
side is $3$, and only the tangential trace-free block has that root.
It is simple in its scalar indicial polynomial, whose derivative
at $3$ is $-3$.  In particular $\tau=3$ in
\eqref{eq:jacobi-b-conormal}.

\paragraph{Recovering a strong boundary trace.}
The indicial calculation gives the possible power of $x$, but not
by itself convergence of its coefficient.  Because the parametrix
controls ordinary tangential derivatives, the remaining argument is
an ordinary differential equation in the radial variable, with
values in each boundary $C^m$ space.
For the smooth PE metric the full operator, in this zero-frame,
has the form
\[
 2L=I(2L,D)+x\mathcal R,
\]
where $\mathcal R$ is a second-order operator generated by
$D$ and ordinary tangential derivatives, with smooth bounded
coefficients up to $x=0$.  Terms with two tangential derivatives
may in fact contain $x^2$; only the displayed first-order vanishing
is needed.  The frame may be chosen parallel in the radial
direction; all differences from its hyperbolic indicial connection
vanish at least to order $x$.  Equation
\eqref{eq:jacobi-b-conormal} therefore gives
\[
 I(2L,D)u=f,\qquad f\in\mathcal A^4.
\]
All the following ordinary differential equations hold in every
boundary $C^m$ norm because the tangential derivatives are already
bounded by \eqref{eq:jacobi-b-conormal}.

In the tangential trace-free block put $u_T=x^3v_T$.  Then
\[
 D(D+3)v_T=F_T,\qquad F_T\in\mathcal A^1.
\]
Let $w=Dv_T$.  Since $w$ is bounded, integrating
$D w+3w=F_T$ from zero gives
\[
 w(x)=x^{-3}\int_0^x t^2F_T(t)\,dt=O_b(x).
\]
Consequently $\partial_xv_T=w/x$ is bounded in every boundary
$C^m$ norm, and $v_T$ has a limit $\kappa_{(3)}\in C^\infty$ with
\[
 v_T-\kappa_{(3)}=O_b(x),\qquad Dv_T=O_b(x).
\]

For the mixed block write $u_M=x^3v_M$.  Its equation is
\[
 (D-1)(D+4)v_M=F_M,\qquad F_M\in\mathcal A^1.
\]
Putting $w=(D+4)v_M$ and integrating the first equation from a
fixed $x_0>0$ gives
\[
 w(x)=x\left(\frac{w(x_0)}{x_0}
       +\int_{x_0}^x\frac{F_M(t)}{t^2}\,dt\right)
       =O_b(x(1+|\log x|)).
\]
Boundedness of $v_M$ eliminates the $x^{-4}$ solution of its
second equation, so
\[
 v_M(x)=x^{-4}\int_0^x t^3w(t)\,dt
       =O_b(x(1+|\log x|)).
\]
Thus $v_M\in\mathcal A^\epsilon$ for every $\epsilon<1$.
This estimate deliberately allows a possible resonant $x\log x$
term.
For either scalar block write $u_S=x^3v_S$, set
$a_+=(\sqrt{33}-3)/2>1$ and
$b_+=(\sqrt{33}+3)/2>0$, and obtain
\[
 (D-a_+)(D+b_+)v_S=F_S,\qquad F_S\in\mathcal A^1.
\]
The same integrations give first
$(D+b_+)v_S=O_b(x)$ (the extra homogeneous term is $O(x^{a_+})$)
and then $v_S=O_b(x)$; boundedness excludes $x^{-b_+}$.
All normal and tangential-trace components therefore vanish after
division by $x^3$.

Combining the blocks proves \eqref{eq:strong-jacobi-trace}.
In a radial parallel orthonormal frame the outward normal to
$\{x\ge\varepsilon\}$ is $N=-x\partial_x$, so it also gives
\[
 \nabla_N\kappa=-3\kappa+O_b(x^{3+\epsilon}),
 \qquad \kappa(N,\cdot)=O_b(x^{3+\epsilon}).
\]
These are strong estimates sufficient for the boundary Green
pairing. 
\end{proof}

We can now explain the significance of the trace.  Varying the
boundary metric produces an inhomogeneous term in the linearized
Einstein equation.  Integrating that term against a Jacobi field
leaves, at infinity, an ordinary pairing between the boundary
variation and the leading Jacobi coefficient.  If every such pairing
vanished, boundary unique continuation would force the Jacobi field
to vanish.  Thus the boundary variations span the obstruction space.

\begin{proposition}[Smooth invariant boundary directions span every obstruction]
\label{prop:finite-invariant-obstruction}
Suppose $G$ carries the coherent isometric action $I\Gamma$ extending
the prescribed boundary action.  Let $e(\eta)$ be the variation of
the prepared reference metric under a smooth boundary variation
$\eta$:
\[
 e(\eta)=\chi\{x^{-2}\pi^*\eta-\pi^*(DP_q\eta)\}.
\]
Let $S_G\eta=D_a\Phi_G(\eta,0)$ be the boundary derivative of the
harmonic-identity Einstein equation.  Then $S_G\eta$ belongs to
$\mathcal H^\infty_d$, and
\begin{equation}\label{eq:finite-obstruction-green}
 \int_X\langle S_G\eta,\kappa\rangle_G\,dV_G
 =-\frac32\int_{S^3}\langle\eta,\kappa_{(3)}\rangle_q\,dV_q,
 \qquad \kappa\in K_G.
\end{equation}
Here $\kappa_{(3)}$ is the strong smooth boundary coefficient from
Lemma~\ref{lem:strong-jacobi-boundary-trace}.
Consequently
\[
 \mathcal O_G:C^\infty(S^2T^*S^3)^\Gamma\longrightarrow K_G^{I\Gamma},
 \qquad \mathcal O_G\eta=\mathsf P_{K_G}S_G\eta,
\]
is onto.  Finitely many smooth invariant tensors suffice to span its
range.  Adjoining their coefficients to the local equation makes its
full derivative surjective, with finite-dimensional kernel.
\end{proposition}

\begin{proof}
\emph{The Green identity.}
When $h=0$, the domain and target reference metric coincide, so
$\Phi(a,0)=\operatorname{Ric}_{H(a)}+3H(a)$.  The second-jet
calculation (HG.3a) shows that its boundary derivative has weight
three, and hence weight $d$.  At the central Einstein metric its
linearization is
\[
 S_G\eta=L_Ge(\eta)-\delta_G^*\beta_Ge(\eta).
 \tag{OB.2}
\]
For $X_\epsilon=\{x\ge\epsilon\}$, the outward unit normal near
its boundary is $N=-x\partial_x$.  Since $L_G\kappa=0$ and
$\delta_G\kappa=0$, two integrations by parts give
\[
 \begin{split}
 \int_{X_\epsilon}\langle S_G\eta,\kappa\rangle
 ={}&\frac12\int_{x=\epsilon}
    \{\langle e,\nabla_N\kappa\rangle
                 -\langle\nabla_Ne,\kappa\rangle\}\,dA_G\\
 &-\int_{x=\epsilon}
       (\beta_Ge)(\kappa(N,\cdot)^\sharp)\,dA_G.
 \end{split}
 \tag{OB.3}
\]
The coefficient $1/2$ is the coefficient of $\nabla^*\nabla$ in
$L_G$.  Near the boundary, $e$ has leading coefficient $\eta$ in
the rescaled covector frame, $\nabla_Ne=O(x^2)$, and
$\beta_Ge=O(1)$.  The strong boundary expansion gives
\[
 \kappa=x^3\kappa_{(3)}+O_0(x^{3+\varepsilon}),\quad
 \nabla_N\kappa=-3x^3\kappa_{(3)}+O_0(x^{3+\varepsilon}),\quad
 \kappa(N,\cdot)=O_0(x^{3+\varepsilon})
\]
for any fixed $0<\varepsilon<1$, with the first two formulas
understood in the rescaled frame and the leading coefficient purely
tangential.  Also $dA_G=x^{-3}(1+O(x^2))\,dV_q$.
Thus the last term of (OB.3) and the term containing $\nabla_Ne$
tend to zero, while the remaining term tends to the right side of
\eqref{eq:finite-obstruction-green}.  The integral on the left
converges absolutely.  This proves the sign and the factor $3/2$.
The calculation is made at one smoothly compactified filling, independent of a
Weyl bound.

\smallskip\noindent\emph{Surjectivity on invariant tensors.}
If the invariant range were proper, choose a nonzero
$\kappa\in K_G^{I\Gamma}$ annihilating it.  For any smooth
boundary tensor $\eta$, its group average is invariant, and
naturality of the formula gives the same pairing with $\kappa$.
Thus the right side of \eqref{eq:finite-obstruction-green} vanishes
for every smooth $\eta$.  Its smooth coefficient
$\kappa_{(3)}$ must vanish.  The strong expansion now gives
$|\kappa|_G=o(x^3)$.

Apply the boundary unique-continuation theorem for divergence-free
infinitesimal Einstein deformations:
\cite[Corollary 4.4]{AndersonHerzlich2008}, with the coordinate
construction corrected in \cite{AndersonHerzlich2010Erratum}.
Thus, $G$ is smoothly conformally compact,
$\delta_G\kappa=0$, $\kappa$ solves the linearized Einstein
equation by (OB.2) with zero Bianchi term, and
$\pi_1(B^4,S^3)=0$.  The conclusion is $\kappa=0$, a
contradiction.  Only the smooth case is used here, so the finite
regularity reserve in the erratum is automatic.  Equivalently this
is the precise smooth case of \cite[Proposition 3.1]{Anderson2010}.

Surjectivity onto a finite-dimensional space permits choosing a
finite set of smooth invariant preimages of a basis.  Adding these
directions makes the full derivative onto.  Since the metric
derivative is Fredholm of index zero, adding $\ell$ parameters
gives a Fredholm map of index $\ell$; its kernel is finite-dimensional
and complemented.  The finite-dimensional reduction lemma from the
harmonic-gauge section therefore applies.  
\end{proof}

There are two properties to check at the initial endpoint: the
hyperbolic filling is regular, as proved above, and it is the only
filling with the prescribed round boundary marking.  The
Li--Qing--Shi comparison and equality in Bishop--Gromov give the
second statement.

\begin{lemma}[Uniqueness of the marked hyperbolic seed]
\label{lem:finite-hyperbolic-seed}
The fiber with prescribed boundary metric $g_{\rm rd}$ has exactly
one marked PE class, the hyperbolic class.  It is a regular seed in
every invariant finite-parameter problem considered above.
\end{lemma}

\begin{proof}
Apply Proposition~\ref{prop:LQS} below to any smooth PE filling
with round conformal infinity, choosing any interior point.  No
group invariance is needed for this application.  Its positive
boundary test function satisfies
\[
 Y(S^3)\le\mathcal Q_{g_{\rm rd}}(w)
 \le6\left(\int_{S^3}w^6\,dV_{g_{\rm rd}}\right)^{2/3}
 =Y(S^3)\,\Theta^{2/3}\le Y(S^3).
\]
Thus $\Theta=1$.  Bishop--Gromov equality gives hyperbolic space.
The geodesic-ray calculation in the proof of
Theorem~\ref{thm:boundary-rigidity} gives a smooth local conformal
endpoint map from the prescribed round boundary to the ideal round
sphere. Its domain is compact, so this local diffeomorphism is a
covering of $S^3$; since $S^3$ is simply connected it is a
diffeomorphism. The explicit hyperboloid expression
$Z=(x^{-1}U+xV)/2$ also shows that the interior isometry
extends smoothly to this boundary map in ball coordinates.
Liouville makes the boundary map a M\"obius transformation,
and hence the boundary value of a hyperbolic isometry.  Compose with its inverse to obtain an
isometry preserving the given boundary marking.  Hence there is
exactly one marked class.  Its invariant weighted metric derivative
is invertible by Lemma~\ref{lem:finite-hyperbolic-linear-seed}; the
fixed-coordinate implicit-function theorem gives the unique regular
local seed branch.
\end{proof}

\subsection{Marked compactifications and weighted compactness}
\label{sec:continuation-compactness-bridge}

The preceding lemmas solve the local analytic problem.  To use
parity, however, we need \emph{one compact space of actual fillings}
containing all solutions over the prescribed path.  The geometric
compactness theorem works with adapted compactifications and
boundary-fixing diffeomorphisms.  The local equation works with
weighted tensors and harmonic-map gauges.  These are different
notions of convergence until their equivalence has been proved.

A convenient intermediary is a compact metric obtained
from each filling.  It will control boundary-marked isometries,
and hence the topology of the quotient.  We shall then show that
geometrically convergent fillings enter the weighted local charts.
This is why the following topology is constructed explicitly rather
than left implicit in the word ``moduli.''

Fix a smooth compact ball $\overline X$, its prescribed identification
$\partial X=S^3$, a smooth defining function $x$, and a smooth collar.
Fix $k\ge4$, $0<\beta<1$, and the weight
\[
 d=\frac52.
\]
The notation $C^{k,\beta}_{0}$ denotes the uniformly local H\"older
space in hyperbolic rescaling charts; its derivatives are generated by
$x\partial_x,x\partial_{y^a}$ and its tensor norm uses a fixed smooth
conformally compact reference metric.  Set
\[
 E_j=x^d C^{k+j,\beta}_{0}(S^2T^*X),\qquad j\ge0.
\]
These are complete Banach spaces with their displayed weighted norms.

Let $P_q=\operatorname{Ric}(q)-R_q q/4$ be the three-dimensional
Schouten tensor.  Choose a fixed smooth positive extension $B(q)$
with compactified collar expansion
\begin{equation}\label{eq:prepared-boundary-model}
 x^2B(q)=dx^2+q-x^2P_q+O(x^4).
\end{equation}
Positivity can be arranged for every $q$ by using
$q\exp(-x^2q^{-1}P_q)$ for its tangential part in the collar, then
patching with a fixed interior metric.  This construction is smooth
on every finite smooth family of boundary metrics.  A metric
$G=B(q)+h$, $h\in E_0$, has compactification
\[
 \overline g=x^2G=dx^2+q-x^2P_q+O(x^4)+\overline h,
 \qquad \overline h=O_0(x^d).
\]
The suffix $0$ denotes rescaled derivative estimates.  The second
boundary jet is now prescribed, which is the reason for using a
weight strictly between two and three.

For the global topological set, take smooth boundary metrics with
their $C^\infty$ topology and remainders in
$E_\infty=\bigcap_{j\ge0}E_j$ with the subspace topology from $E_0$.
This topological set is not called a Banach space.  The local
elliptic implicit-function theorem uses the complete space $E_0$;
its Einstein zeros in harmonic gauge acquire membership in
$E_\infty$.  Use interior-smooth, boundary-marked $C^{2,\alpha}$
equivalences, with some $\alpha>1/2$; the decay argument below
upgrades actual equivalences between conormal representatives to
weight $d$ at every order.  Every fixed such pullback is a
homeomorphism, so orbit projection is open.  

\subsubsection{The ordinary regularity supplied by the prepared weight}

The first observation translates decay in the complete metric into
ordinary regularity on the closed ball.  The vanishing boundary jets
are as important as the $C^2$ estimate: they make two nearby metrics
comparable in a common prepared affine model.

\begin{lemma}[Prepared weight gives a classical $C^2$ compactification]
\label{lem:prepared-weight-weak-compactification}
For every $0<\theta<d-2=1/2$, the map $h\mapsto x^2h$ is bounded
from $E_0$ to ordinary $C^{2,\theta}(\overline X)$, and its full
second boundary jet vanishes.  It is also bounded into $W^{2,p}$
for every finite $p$.  Convergence in $E_0$ implies convergence
in these weaker compactification norms.
\end{lemma}

\begin{proof}
In the covector frame $dx/x,dy^a/x$, write
$h=x^d h_{IJ}(dx^I/x)(dx^J/x)$, with uniformly bounded rescaled
coefficients and derivatives through order three.  The ordinary
components of $x^2h$ are $x^d h_{IJ}$ and satisfy
\[
 |D^j(x^d h_{IJ})|\le Cx^{d-j},\qquad 0\le j\le3.
\]
Thus derivatives through order two extend continuously, with zero
boundary values.  For the second-derivative H\"older estimate,
compare two points at Euclidean distance $r$.  If $r$ is at least
half the smaller boundary height, the size estimate gives the bound
$Cr^{d-2}$; if $r$ is smaller, the segment stays at comparable
height $x$, and the third-derivative estimate gives
$Cx^{d-3}r\le Cr^{d-2}$.  A finite collar cover and the compact
interior give the result, with any smaller exponent $\theta$.
Apply the same argument to differences.
\end{proof}

Arbitrary conormal derivative reserve still does not imply ordinary
smoothness: it is only the stated $C^{2,\theta}$ conclusion which
is used in the boundary regularity theorem.

\subsubsection{A canonical compact metric}

Different defining functions give different compact metrics for the
same Einstein filling.  For controlling marked isometries, we prefer
a compactification characterized by an intrinsic equation and a
fixed boundary value.  We choose scalar curvature $-1$.  Unlike the
scalar-flat equation, its Dirichlet linearization has a strictly
positive zeroth-order term, giving both uniqueness and smooth
parameter dependence.

\begin{lemma}[Negative-scalar compactification]
\label{lem:negative-scalar-compactification}
Let $\overline g$ be a uniformly positive $W^{2,p}$ metric,
$p>4$, on a smooth compact four-manifold with boundary. Assume every
connected component meets the boundary and its scalar curvature is bounded.  There is a unique positive solution
$U$, with $U|_{\partial X}=1$, of
\begin{equation}\label{eq:negative-scalar-dirichlet}
 -6\Delta_{\overline g}U+R_{\overline g}U+U^3=0.
\end{equation}
The metric $\widehat g=U^2\overline g$ has scalar curvature $-1$
and the same boundary metric.  Near every such metric the map
$\overline g\mapsto U$ is smooth from $W^{2,p}$ metrics to
$W^{2,p}$ functions.  If $\overline g$ is $C^{2,\theta}$,
$0<\theta<1$, then $U$ is $C^{2,\theta}$ and the same map is smooth
from the open set of $C^{2,\theta}$ metrics to $C^{2,\theta}$ functions.
The construction is natural under any boundary-marked conformal
equivalence that preserves the boundary representative.
\end{lemma}

\begin{proof}
\emph{Existence and positivity.}
For existence, $0$ is a subsolution and any constant
\[
 A\ge\max\{1,\|R_{\overline g}^{-}\|_\infty^{1/2}\}
\]
is a supersolution.  Add a sufficiently large positive zero-order
constant to $-6\Delta_{\overline g}$ and apply monotone iteration
on $[0,A]$, with Dirichlet value one.  The iterates remain ordered and
bounded.  Uniform elliptic $W^{2,p}$ estimates, followed by compact
embedding and passage to the equation, give a solution. On each connected component the nonzero boundary
value and the strong maximum principle give positivity in the interior; its boundary
value and continuity give a positive minimum on the compact closure.

\smallskip\noindent\emph{Uniqueness.}
For uniqueness, if $\widehat g_2=f^2\widehat g_1$ and both scalar
curvatures are $-1$, then $f=1$ on the boundary and
\[
 -6\Delta_{\widehat g_1}f=f-f^3.
\]
An interior maximum greater than one or minimum less than one is
impossible.  Thus $f=1$.

\smallskip\noindent\emph{Dependence on the compact metric.}
The equation defines a smooth map of the Banach variables
$(\overline g,U)\in W^{2,p}\times W^{2,p}$ into $L^p$;
$p>4$ guarantees the required products, inverse metrics and
$C^{1,\alpha}$ control.  For a variation $v=Uz$ with zero boundary
value, conformal covariance gives
\begin{equation}\label{eq:negative-scalar-linearization}
 ( -6\Delta_{\overline g}+R_{\overline g}+3U^2)v
 =U^3(-6\Delta_{\widehat g}+2)z.
\end{equation}
The operator on the right, with zero Dirichlet condition, is an
isomorphism $W^{2,p}\cap W^{1,p}_0\to L^p$: existence follows
from coercivity, and elliptic regularity and the maximum principle
give the stated inverse and uniqueness.  The Banach implicit-function
theorem proves smooth dependence.  This argument applies in a
$W^{2,p}$ neighborhood; bounded scalar curvature was used only for
the elementary initial existence construction.

\smallskip\noindent\emph{The classical H\"older topology.}
For the H\"older assertion, classical Schauder regularity first gives
$U\in C^{2,\theta}$ when $\overline g\in C^{2,\theta}$.
The nonlinear equation is a smooth map of the open set of uniformly
positive $C^{2,\theta}$ metrics and positive $C^{2,\theta}$ functions
with fixed boundary value into $C^{0,\theta}$: scalar curvature uses
two metric derivatives, and all its products and inverse metrics are
smooth in these Banach spaces.  Formula
\eqref{eq:negative-scalar-linearization} now identifies its derivative
with the conjugate of the Dirichlet isomorphism
\[
 -6\Delta_{\widehat g}+2:
 \{z\in C^{2,\theta}:z|_{\partial X}=0\}
 \longrightarrow C^{0,\theta}.
\]
The Schauder estimate, coercivity and maximum principle give its
bounded inverse.  The Banach implicit-function theorem therefore
proves smooth dependence in precisely the $C^{2,\theta}$ topology.
In particular convergence of the prepared compactifications in that
topology implies convergence of their negative-scalar compact
metrics in the same topology.  Naturality follows from the uniqueness
argument applied after pullback.
\end{proof}

Applied to $\overline g=x^2G$, this lemma supplies an intrinsic
compact metric $\widehat g(G,q)$ for the marked pair $(G,q)$.
If $F^*H=G$, $F|_{\partial X}=\mathrm{id}$, and both prescribed
boundary representatives are $q$, then
\begin{equation}\label{eq:negative-compactification-naturality}
 F^*\widehat g(H,q)=\widehat g(G,q).
\end{equation}
Indeed the two sides are conformal compactifications of $G$ with
scalar curvature $-1$ and boundary metric $q$, so their positive
conformal ratio has boundary value one and the preceding maximum
principle applies.  

\subsubsection{Smooth representatives and boundary regularity}

A weighted solution need not be ordinarily smooth at the boundary
in its initial coordinates.  Nevertheless, the $C^2$ estimate above
is strong enough to invoke the precise boundary-regularity theorem
for conformally compact Einstein metrics.  This supplies an
individual smooth representative of each marked filling.  The
remaining issue is whether the identifying map belongs to our
allowed class of boundary-fixing maps; the canonical compact metric
will provide the necessary improvement.

The boundary-regularity provided by Chru\'sciel--Delay--Lee--Skinner,
\emph{Boundary regularity of conformally compact Einstein metrics},
J. Differential Geom.\ 69 (2005), Theorem~A, arXiv:math/0401386v2,  applies here. 
It requires a $C^2$ conformal compactification,
a smooth prescribed conformal-infinity representative, and
$\operatorname{Ric}_G=-nG$. In dimension four the theorem gives a smooth
geodesic compactification after a boundary-identity collar
diffeomorphism of class $C^{1,\lambda}$, for any $0<\lambda<1$.

There is a small further regularity point because our equivalence
category requires more than the $C^{1,\lambda}$ collar map supplied by the theorem.  Choose the theorem's representative to be $q$
itself.  The initial compactification and the resulting smooth
geodesic compactification have the same boundary value $dx^2+q$.
Their conformal metric transformation therefore gives
$D\Phi=\mathrm{id}$ on the boundary.  Hence its displacement is
$O(x^{1+\lambda})$ and its derivative error is $O(x^\lambda)$.
In the product collar, write the tangential displacement with a
fixed smooth boundary exponential map.  Multiply this displacement
and its normal displacement by $\chi(x/r)$ and set the map equal
to the identity outside the collar.  The derivative error, including
the cutoff derivative, is $O(r^\lambda)$.  For sufficiently small
$r$ the resulting map is a global boundary-identity
$C^{1,\lambda}$ diffeomorphism, smooth in the interior and equal
to the collar map near the boundary.  The pulled-back PE
metric $H=\Phi^*G$
has a smooth compactification on the original smooth ball.  Its
scalar-curvature $-1$ compactification is smooth.  The
original $\overline g=x^2G$ is $C^{2,\theta}$ and
\eqref{eq:negative-scalar-dirichlet}, now by classical Schauder
regularity, has $U\in C^{2,\theta}$, so its compactification
is $C^{2,\theta}$.  Naturality and the conformal-ratio maximum
principle make the extended map an isometry between these two
compact metrics.  Explicitly, $\Phi^*\widehat g(G,q)$ and
$\widehat g(H,q)$ differ by a positive factor $f$, smooth in the
interior and continuous on the compact closure, with $f=1$ on the
boundary.  Both scalar curvatures equal $-1$ in the interior, so
$-6\Delta f=f-f^3$ gives $f=1$ by maximum and minimum principles.
This does not require prior boundary second derivatives of the
pulled-back compact metric.  Such an isometry is $C^{3,\theta}$: in harmonic
coordinates it is weakly harmonic and gains regularity by the
second-order elliptic equation; equivalently, after the first
$C^2$ improvement the connection identity
\eqref{eq:isometry-connection-control} bootstraps it.  Thus the
collar identification lies in the admissible marked
category.  This also retains the original smooth structure and
boundary marking of $B^4$.

Only individual smoothing of the intrinsic compactified geometry is
needed.  Finite-dimensional parameter smoothness is proved in the
weighted harmonic gauges; the negative-scalar compactification
has the $C^{2,\theta}$ dependence needed for marked properness by
Lemma~\ref{lem:negative-scalar-compactification}.

\subsubsection{Compactness of boundary-marked isometries}

Suppose two sequences of compactified metrics converge and are
pairwise isometric by maps fixing the boundary.  To pass to a
quotient, we need the isometries themselves to converge after a
subsequence.  This is a familiar consequence of the connection
transformation law once the compact metrics have controlled
$C^{2,\theta}$ regularity.  It is the source of both local uniqueness
and the Hausdorff property below.

\begin{proposition}[Compactness of the actual equivalences]
\label{prop:weak-marked-equivalence-properness}
Let $q_i$ be smooth boundary metrics converging to $q$ in $C^\infty$.
Let
\[
 G_i=B(q_i)+h_i\longrightarrow G=B(q)+h,
 \qquad H_i=B(q_i)+\ell_i\longrightarrow H=B(q)+\ell
\]
in the weight-$d$ $E_0$ topology.  Suppose $G_i,H_i$ are PE and
$F_i^*H_i=G_i$, with $F_i$ fixing the boundary pointwise.  Initially
allow equivalences that are smooth in the interior and $C^{2,\alpha}$
up to the boundary; smooth representatives and the harmonic gauges
constructed in the continuation argument satisfy this convention.
Then, after a subsequence, $F_i$ converge in $C^{3,\theta}$ for
every fixed $0<\theta<1/2$ to a boundary-marked equivalence
$F^*H=G$.  The same holds for the inverses.  In particular the
equivalence relation is closed.
\end{proposition}

\begin{proof}
We use the compactifications to bound the isometries,
and then differentiate the connection equation to gain compactness
through three derivatives.

Lemmas~\ref{lem:prepared-weight-weak-compactification} and
\ref{lem:negative-scalar-compactification} give
\[
 \widehat g_i\to\widehat g,
 \qquad \widehat h_i\to\widehat h
 \quad\hbox{in }W^{2,p},
\]
for any finite $p>4$.  Naturality makes $F_i$ isometries of these
compact metrics.  Uniform positivity gives a uniform bound on
$DF_i$ and $DF_i^{-1}$.  The connection transformation law is
\begin{equation}\label{eq:isometry-connection-control}
 \partial_a\partial_bF_i^c
 =\Gamma(\widehat g_i)^d_{ab}\partial_dF_i^c
  -\Gamma(\widehat h_i)^c_{ef}(F_i)
       \partial_aF_i^e\partial_bF_i^f.
\end{equation}
Since $W^{2,p}\subset C^{1,\theta}$, this first bounds $F_i$ in
$C^{2,\theta}$.  Differentiating
\eqref{eq:isometry-connection-control} once gives a $W^{3,p}$
bound: the derivatives of the two Christoffel symbols are in $L^p$,
and changing variables by $F_i$ preserves their bounds because the
Jacobians are uniformly bounded above and below.  The inverse maps
satisfy the same estimates.  Compact embedding gives the proposed
subsequence for $\alpha'<1-4/p$; $p$ can be arbitrarily large.
For the stronger asserted convergence, choose
$\theta<\epsilon<1/2$.  The prepared compactifications converge
in $C^{2,\epsilon}$, and the same scalar Dirichlet implicit-function
theorem and Schauder estimates give convergence of the 
compact metrics in $C^{2,\epsilon}$.  Their Christoffel symbols
are $C^{1,\epsilon}$.  Differentiating the connection identity once
therefore bounds $F_i$ and their inverses in $C^{3,\epsilon}$;
compact embedding gives $C^{3,\theta}$ convergence.
The two inverse limits compose to the identity.  Passing to
$F_i^*H_i=G_i$ on compact subsets of the interior proves $F^*H=G$,
and the boundary marking persists.
\end{proof}

\begin{lemma}[The third jet of a prepared marked isometry]
\label{lem:marked-isometry-second-jet}
If $G$ and $H$ have the same prepared second compactified jet
$dx^2+q-x^2P_q$, any $C^{3,\theta}$ boundary-identity equivalence
$F^*H=G$ has $j^3_{\partial X}F=j^3_{\partial X}\mathrm{id}$.
Consequently its displacement has complete-metric norm
$O(x^{2+\theta})$.
If $F_i\to\mathrm{id}$ in $C^{3,\theta}$, these displacements
tend to zero in weighted $C^0_\delta$ for every
$\delta<2+\theta$.
\end{lemma}

\begin{proof}
Write $F^x=ax+o(x)$ and $F^y=y+bx+o(x)$ at the boundary.
The tangential part of
\[
 x^2G=\left(\frac{x}{x\circ F}\right)^2F^*(x^2H)
\]
gives $a=1$, and its cross term gives $b=0$.  Thus $DF$ is the
identity at the boundary.  All tangential and mixed second
derivatives of $F-\mathrm{id}$ vanish.  Write the remaining terms
as $F^x=x+A(y)x^2+o(x^2)$,
$F^y=y+B(y)x^2+o(x^2)$.  In the displayed metric identity the
$xx$ coefficient at order $x$ is $2A$, and the $xa$ coefficient
is $2B_a$.  Both vanish because the compactified metrics have zero
first normal jet.  Thus $A=B=0$.  Next write
$F^x=x+C(y)x^3+o(x^3)$ and
$F^y=y+D(y)x^3+o(x^3)$.  At order $x^2$, the same metric
identity has extra $xx$ coefficient $4C$ and cross coefficient
$3D_a$.  The prescribed second jets on its two sides agree, so
$C=D=0$.  Tangential and mixed third derivatives vanish by
differentiating the second-jet identity along the boundary.
Taylor's formula with the $C^{3,\theta}$ remainder gives the
stated estimate, uniformly for
convergent families.  The corresponding exponential displacement
has the same estimate, since the complete metric is uniformly
equivalent to $x^{-2}(dx^2+q)$ and the displacement is small near
the boundary.
\end{proof}

Choose once and for all
\[
 2<\delta_-<2+\theta<d=\frac52.
\]
The preceding proposition and lemma give exactly the weak
$C^0_{\delta_-}$ control needed in the harmonic-map uniqueness
argument.  The nonlinear gauge construction then upgrades an
equivalence between gauged solutions to
$C^{k+1,\beta}_{\delta_-}$ by the equation, and uniqueness at
$\delta_-$ identifies it with the gauge constructed at weight $d$.

For completeness, the global asymptotic weight improvement has the
same mechanism.  Let $V$ be the exponential displacement of a marked
isometry near infinity and cut it off in the interior.  Its harmonic
map equation, compared to the identity and to the same target
metric, has the form
\[
 (\nabla_G^*\nabla_G+3)V
 =f_d+O_0(x^d)(V,\nabla V,\nabla^2V)
        +O_0(V^2,V\nabla V,|\nabla V|^2),
\]
where $f_d$ has weight $d$ because the two metrics share both $q$
and the prescribed second coefficient $-P_q$.  Interior cutoff terms are compactly
supported.  Starting with $V$ of weight $\delta_->2$, all terms
on the right have weight $d$, since $2\delta_->4>d$.  Invertibility
of the vector operator at weights $\delta_-$ and $d$, with the identical
$L^2$ kernel zero, promotes $V$ to weight $d$.  The local elliptic estimates
give the corresponding derivative reserve.  Thus a fixed equivalence
acts by a bounded pullback map on every weighted level on which it
is used. 

The marked weighted source is metrizable (it is a subspace of the
product of the smooth-boundary Fr\'echet space and $E_0$).
The properness proposition therefore makes its equivalence relation
closed.  Every fixed admissible pullback is a homeomorphism, so the
orbit projection is open.  Its product with itself is an open
quotient map; the inverse image of the diagonal is the closed
equivalence relation.  Consequently the marked weighted quotient is
Hausdorff.  This uses only fixed pullbacks and actual isometries.

\subsubsection{From geometric convergence to weighted convergence}

The essential bridge is now straightforward to formulate, although
keeping track of the derivatives requires care.  Geometric
compactness gives convergence of compactified metrics in ordinary
$C^K$.  After choosing common geodesic coordinates and subtracting
the prepared second boundary jet, the difference vanishes through
order two.  Taylor's theorem then provides precisely the factor
$x^{5/2}$ needed for the local weighted norm.  Thus the compactness
theorem produces \emph{actual points of the analytic local charts},
even at a degenerate filling.

\begin{proposition}[Weighted representative capture]
\label{prop:weighted-representative-capture}
Let $q_i$ range over a compact smooth finite-parameter family of PSC
metrics on $S^3$, and let $G_i$ be any sequence of marked PE metrics
on $B^4$ with conformal infinities $[q_i]$.  Theorem~\ref{thm:compactness} supplies, after a subsequence, a PE
limit $G_\infty$ and boundary-identity representatives which are
close to $G_\infty$ in every prescribed finite weight-$d$ norm
after placing the metrics in a common geodesic collar.  In particular
the sequence is captured by the finite local gauge chart at the
limit; this includes degenerate limits.
\end{proposition}

\begin{proof}
The main point is a finite derivative count.  The geometric theorem
provides high ordinary regularity; obtaining a common geodesic
collar consumes derivatives, and dividing out the first three Taylor
coefficients then gives the prepared weight.

If an initial representative is only weakly compactified, the
boundary-regularity result above first supplies a smooth representative
of the same marked geometric class.  Theorem~\ref{thm:compactness}
is then applied to the canonical scalar-flat adapted compactifications
$g_i=\rho_i^2G_i$ for the prescribed $q_i$.  For any preassigned
large integer $K$ it gives smooth boundary-identity diffeomorphisms
after which
\[
 g_i\longrightarrow g_\infty\quad\hbox{in }C^{K,\nu}.
\]
One chooses $K$ before applying the theorem; smoothness of the finite
boundary family supplies all its required boundary bounds. 

\smallskip\noindent\emph{Convergence of the adapted defining functions.}
The adapted identities give convergence of the defining functions.
Indeed $|d\rho_i|\le1$, the compact diameters are uniformly bounded,
$P_i|_{\partial X}=2H_i/3$, and
\[
 dP_i=\operatorname{Ric}(g_i)(\nabla\rho_i,\cdot),\qquad
 \nabla^2\rho_i=-\tfrac12\rho_i\operatorname{Ric}(g_i)
                         -\tfrac12P_i g_i.
\]
The first two observations bound $\rho_i,P_i$.  Repeated use of
these equations gives $\rho_i\to\rho_\infty$ in
$C^{K,\nu'}$ and $P_i\to P_\infty$ in
$C^{K-1,\nu'}$, after decreasing the exponent and taking a
subsequence.  The limiting eikonal identity rules out an interior
zero of $\rho_\infty$.

\smallskip\noindent\emph{A common geodesic collar.}
Acquire the geodesic defining functions for $q_i$ by the actual
noncharacteristic eikonal equation and its normal flow.  The fixed
finite derivative budget is as follows: coefficients of the
Hamiltonian are $C^{K-1,\nu'}$; its characteristic field is
$C^{K-2,\nu'}$; solving gives the defining function at least
$C^{K-2,\nu'}$; taking the normal field and pulling back by its
flow spends two additional derivatives.  Choosing
\[
 K\ge k+7,\qquad \beta<\nu'<\nu
\]
therefore gives convergence of the pulled-back geodesic
compactifications in $C^{k+3,\nu'}$.  The common collar has a fixed
positive width because the initial metrics, inverse metrics and
defining functions converge and the boundary is compact.  These
collar diffeomorphisms fix the boundary and are close to a fixed
limit; a cutoff collar isotopy extends them to the full ball, so no
extension of an arbitrary sphere diffeomorphism is required.

\smallskip\noindent\emph{Taylor division and weighted convergence.}
In the resulting fixed coordinates,
\[
 G_i=x^{-2}(dx^2+h_i(x)),\qquad
 h_i(0)=q_i,\quad h_i'(0)=0,\quad h_i''(0)=-2P_{q_i},
\]
and the compactified metrics converge in $C^{k+3,\nu'}$.
Choose the local background
\[
 B_*(q)=G_\infty+\chi(x)\bigl[x^{-2}(q-q_\infty)-(P_q-P_{q_\infty})\bigr],
\]
where $\chi=1$ on the common collar.  For
$r_i=G_i-B_*(q_i)$, the compactified tensor $T_i=x^2r_i$
has zero jets through order two and tends to zero in
$C^{k+3,\nu'}$.  Taylor's formula gives the exact identity
\begin{equation}\label{eq:double-Taylor-weight-capture}
 \frac{T_i(x,y)}{x^d}
 =\frac{x^{3-d}}2\int_0^1(1-t)^2\partial_x^3T_i(tx,y)\,dt.
\end{equation}
The integral on the right tends to zero in ordinary $C^{k,\nu'}$.
Multiplication by $x^{3-d}=x^{1/2}$ is bounded in every rescaled
$C^{k,\beta}_0$ norm.  Thus $T_i/x^d\to0$ in that norm on
the collar; the interior is compact.
Hence $\|r_i\|_{E_0}\to0$.

The local nonlinear gauge theorem now gives its unique near-identity
gauge and the finite local representative.  This step uses the
nondegenerate vector gauge operator, not invertibility of the tensor
Einstein operator, so nondegeneracy of $G_\infty$ is not assumed.
\end{proof}

For clarity, we use more derivatives of a \emph{smooth
compact finite family} to enter a fixed Banach norm.  

\paragraph{Smooth limits in the auxiliary quotient.}
A final regularity point must be addressed before discussing
compactness of the quotient itself.  Convergence in a fixed finite
$C^K$ topology does not automatically produce a smooth limit in the
original coordinates.  We use a finite reserve of derivatives to
choose a smooth representative of the limit, without choosing a
parameter-dependent smooth atlas.

There is one precise issue if the auxiliary geometric quotient is
defined using only ordinarily smooth adapted pairs with the integer
$C^K$ topology.  A limit in a fixed $C^K$ atlas must not simply be
declared smooth.  The following construction supplies representatives
whose limit lies in that source.  Start with convergence of adapted
pairs at order $K+4$, with some positive H\"older exponent; the
smooth compact boundary family permits this finite reserve.
Apply the already proved mixed Schouten and harmonic metric bootstrap
to the limit itself.  Its smooth boundary data give a smooth adapted
metric and defining function in the limiting boundary harmonic atlas.
The harmonic coordinate transition from the original fixed atlas is
$C^{K+5,\nu}$, after reducing $\nu$ if needed, by the scalar harmonic
coordinate equation.  Interior harmonic charts give the same
regularity and then smoothness of the coefficients.  Overlap maps
between these smooth harmonic charts are smooth by their elliptic
coordinate equations.  Their boundary restrictions use the prescribed
smooth boundary structure.

These are compatible $C^{K+5}$ smooth structures on the same compact
ball, with the same boundary marking.  The precise relative smoothing
fact used here is the following: if two compact smooth manifolds with
boundary are related by a $C^r$ diffeomorphism, $r\ge1$, whose boundary
restriction is a smooth diffeomorphism, that map can be approximated
arbitrarily closely in $C^r$ by smooth maps with exactly the same
boundary restriction.  Sufficiently close approximants are smooth
diffeomorphisms.  Relative smooth approximation in boundary charts,
with the normal component zero and the specified tangential components
on the boundary, gives the maps; smoothing the remaining components
and using a target tubular retraction gives the usual finite-chart
construction.  The openness of the diffeomorphism condition in $C^1$
on a compact domain gives the last assertion.  

Apply this fact to the identity from the original structure to the
limiting harmonic structure with $r=K+5$.  It gives a boundary-identity
diffeomorphism $\Psi$, smooth between the two structures and at least
$C^{K+5}$ as a map in the original coordinates.  Hence
$\Psi^*g_\infty,\Psi^*\rho_\infty$ are ordinarily smooth on the
original ball.

Approximate this \emph{one fixed} map $\Psi$ by smooth
boundary-identity diffeomorphisms $\Psi_j$ in the integer
$C^{K+1}$ norm.  Relative smooth approximation applies because
$\Psi$ has more derivatives than required; positivity of its
Jacobian and compactness make the diffeomorphism condition open in
$C^1$.  Then, on a subsequence,
\[
 \Psi_j^*g_{i_j}\longrightarrow\Psi^*g_\infty,
 \qquad
 \Psi_j^*\rho_{i_j}\longrightarrow\Psi^*\rho_\infty
 \quad\hbox{in }C^K.
\]
The pullback estimate follows from the uniform $C^{K+1}$ bounds
on the maps and continuity of integer-order composition; the
limiting tensors have spare derivatives in any event.  Every map
applied to the original sequence is an allowed smooth
boundary-identity diffeomorphism, and the limiting pair is smooth.
This establishes sequential compactness in the smooth-pair quotient.

We are ready to turn subsequential convergence of metrics into
compactness of the marked solution space.  The auxiliary quotient is
built from smooth adapted pairs in a fixed integer regularity; its
open orbit projection and properness of marked isometries provide
a Hausdorff, second-countable topology.  Geometric compactness
makes this auxiliary quotient compact, and weighted capture
identifies it with the desired path inverse image.

\begin{theorem}[Actual compactness of the path zero set]
\label{thm:actual-weighted-path-compactness}
In the preceding marked weighted topology, the full set of PE classes
over a fixed compact smooth PSC boundary path, with the path parameter
retained, is compact and metrizable.  The same assertion holds over
any compact finite-dimensional smooth PSC boundary family.
\end{theorem}

\begin{proof}
We use an auxiliary quotient of smooth adapted pairs.  Its topology
is elementary enough to make compactness precise; weighted capture
then shows that the auxiliary quotient and the actual marked path
inverse image describe the same objects with the same topology.

Let $T$ be the compact parameter set of the prescribed smooth boundary
family $q_t$ (in particular, $T=[0,1]$ for the path).  Form the
auxiliary source of smooth canonical adapted pairs $(g,\rho)$,
together with $t\in T$ and $g|_{TS^3}=q_t$, in the integer
$C^K\times C^K$ topology for the pair and the given topology for $t$.
Here $K\ge k+10$ is fixed.  Quotient by smooth boundary-identity
diffeomorphisms, keeping $t$ fixed.

This auxiliary quotient is second countable.  Indeed its source is a
subspace of a separable metric space, and its orbit projection is open
because each fixed smooth pullback is a homeomorphism.  The images of
a countable source base therefore give a countable quotient base.
It is Hausdorff: if two convergent sequences of smooth pairs are
related by marked diffeomorphisms, the isometry connection equation
first gives uniform Lipschitz bounds for those maps and their inverses,
then bounds and subsequential convergence of their derivatives.  The
limit is an isometry of the smooth limiting compact metrics and
preserves their defining functions and boundary marking.  The same
connection equation makes it smooth.  Thus the equivalence relation
is closed; openness of the orbit projection gives Hausdorffness as
above.

Theorem~\ref{thm:compactness}, the defining-function system
in the preceding capture proof, and the smooth-limit construction
in the preceding paragraph show sequential compactness in this exact
auxiliary source quotient.  Second countability gives the Lindel\"of
property.  Sequential compactness gives countable compactness, since
a countable open cover without a finite subcover would supply a
sequence without a convergent subsequence.  These two properties
prove actual compactness of the auxiliary quotient.  As a compact
Hausdorff second-countable space it is metrizable.

\smallskip\noindent\emph{Comparison with the weighted topology.}
There is a continuous map from this compact space to the weighted
marked solution space, preserving $t$ and the intrinsic PE metric
$G=\rho^{-2}g$.  To verify continuity, a sequence converging in the
auxiliary quotient has representatives converging in integer $C^K$:
images of successively smaller balls about one representative form a
neighborhood base because the orbit projection is open.  Choose
representatives in those balls.  The common geodesic collar
construction and triple Taylor division in
Proposition~\ref{prop:weighted-representative-capture}, with this
fixed finite derivative reserve, give convergence in $E_0$ after
subtracting the prepared boundary extension.  More explicitly, the
four-derivative collar budget and triple Taylor division spend seven
ordinary derivatives.  The integer choice $K\ge k+10$ leaves more
than one further derivative, so ordinary convergence implies the
required $C^{k,\beta}_0$ convergence.  This is distinct from the
$K\ge k+7$ budget for capture from $C^{K,\nu}$ convergence with
$\nu>\beta$, where the positive H\"older reserve is already present.  Convergence of the
retained parameter $t$ also gives $C^\infty$ convergence of $q_t$,
as required by the weighted source.  Local gauge capture gives
convergence in its quotient.  Sequential continuity suffices because
the auxiliary quotient is metrizable.

This map is onto the full weighted inverse image: classical boundary
regularity above supplies each weighted zero with a smooth intrinsic
representative, and the prescribed PSC boundary metric supplies its
canonical adapted compactification.  It is one-to-one: that adapted
compactification for fixed $(G,q_t)$ is unique.  Explicitly, the ratio
of two scalar-flat compactifications with common boundary metric is
positive, continuous up to the boundary with value one, and harmonic
in either compact scalar-flat metric by conformal covariance; the
maximum principle makes the ratio one.  Hence an equivalence of PE
metrics identifies the adapted pairs.  The target weighted quotient
is Hausdorff by marked properness.  A continuous bijection from a
compact space to a Hausdorff space is a homeomorphism.  Its image,
which is exactly the full inverse image $\mathcal K$, is therefore
compact and metrizable in the actual weighted topology.
\end{proof}

\subsubsection{Preservation of the group action under limits}

There is one more ingredient before continuation: the group action
must not be lost when a sequence of equivariant metrics converges.
The compactification just constructed is natural under isometries,
so the group elements themselves form compact families of maps.
Their limits retain the prescribed boundary transformations and
satisfy the same multiplication relations.

Suppose $q_i$ are $\Gamma$-invariant and $G_i$ carry coherent
isometric extensions $I_i(\gamma)$ of the prescribed boundary
action.  Naturality of the negative-scalar compactification makes
$I_i(\gamma)$ isometries of $\widehat g(G_i,q_i)$.
Proposition~\ref{prop:weak-marked-equivalence-properness} applies
with the fixed boundary map $\gamma$ in place of the identity.
For the finitely many $\gamma$, a common subsequence converges in
$C^{2,\alpha}$ to $I_\infty(\gamma)$.  The group law passes to
the limit, the boundary restriction remains $\gamma$, and
faithfulness follows from that restriction.  Boundary harmonic
regularity makes these intrinsic compactified isometries smooth.
Thus the equivariant locus is closed in the actual compactness
topology, with its prescribed action retained.

Near a fixed equivariant metric $G_*$, choose the harmonic-map
reference invariant under its coherent action $I_*(\Gamma)$.
The image of a nearby invariant class under its unique harmonic gauge
is invariant under this \emph{fixed} action.  Indeed, if $f$ is its
near-identity harmonic gauge and $I(\gamma)$ its coherent action,
then
\[
 I_*(\gamma)^{-1}\circ f\circ I(\gamma)
\]
is another harmonic map with boundary value the identity.  Properness
puts it in the weak $\delta_-$ neighborhood where uniqueness holds;
it equals $f$.  Therefore
$fI(\gamma)f^{-1}=I_*(\gamma)$.  This proves that nearby fixed
classes are represented by genuinely invariant tensors in the local
equivariant equation.  

\paragraph{Scope of the bridge.}
The lemmas above establish the admissibility, weak properness,
Hausdorffness, exact weighted capture, compactness, and symmetry
retention needed for finite-dimensional continuation.  The finite
solution manifolds and their smooth gauge transition maps are
constructed by the separate nonlinear harmonic-gauge argument, using
the two-weight uniqueness just identified.  Global gluing then uses
these actual marked classes and their compact topology.

\subsection{Finite-dimensional equivariant continuation}
\label{sec:finite-equivariant-continuation}

Everything needed for parity is now at hand.  The metric equation has
a finite-dimensional obstruction, detected by smooth invariant
boundary variations.  Its transverse finite-parameter zero sets are
smooth and have compatible overlaps.  Finally, geometric compactness
has been transferred to the marked weighted topology, and group
symmetry is retained under limits.

The remaining argument is topological.  We shall choose one finite
space of boundary variations that works near the entire compact
solution set, perturb the boundary path away from its endpoints,
and apply Sard's theorem.  We keep the boundary representatives,
not just their conformal classes, because the group action and the
perturbation parameters must be the same in every local chart.

We now specify a common topology for
all the local constructions.

Fix a smooth $\Gamma$-invariant defining function $x$, a smooth
$\Gamma$-equivariant collar, and a smooth complete
$\Gamma$-invariant reference metric on the fixed ball; a reference
metric with the prescribed invariant end can be made invariant in
the interior by finite averaging.  Fix $k\ge4$, $0<\beta<1/4$, and
$\delta_+=5/2$, and write
$E_j=\mathcal H^{k+j,\beta}_{\delta_+}(S^2T^*X)$ and
$E_\infty=\bigcap_{j\ge0}E_j$.  We use the prepared second jet
$q-x^2P_q$, where $P_q=\operatorname{Ric}_q-R_qq/4$.
A positive global extension, without a boundary-metric-dependent
collar width, is obtained from
\[
 B(q)=x^{-2}\bigl(dx^2+
          q(e^{-x^2P_q^\sharp}\,\cdot,\cdot)\bigr)
 \quad\hbox{on a smaller collar},\qquad P_q^\sharp=q^{-1}P_q.
\]
The endomorphism in the exponential is $q$-self-adjoint, so its
exponential is positive definite.  The second jet is exactly the
required one.  Patch this positive metric to the fixed reference
metric by a fixed convex interpolation.  The construction is
equivariant for the fixed linear boundary action and smooth on
every finite smooth family of boundary metrics.  Set
\[
 \mathscr A^+=\{(q,h):q\in\operatorname{Met}_{\rm PSC}^\infty(S^3),
    \ h\in E_\infty,\ G=B(q)+h>0,\ \operatorname{Ric}_G=-3G\}.
\]
Give this set the subspace topology from the smooth Fr\'echet
topology on $q$ and the single Banach norm $E_0$ on $h$.  The set
$E_\infty$ with this topology is not asserted to be complete; it
is not the Banach domain on which an implicit-function theorem is
used.  The local equations are defined on the complete $E_0$, and
their Einstein zeros acquire membership in $E_\infty$ by
Lemma~\ref{lem:hg-finite-reductions}.

Let $\mathscr D^\infty_{\delta_+}$ be the group of interior-smooth
boundary-identity diffeomorphisms whose displacement and inverse
displacement, near infinity and measured with the fixed reference
exponential map, belong to
$\mathcal H^\infty_{\delta_+}(TX)$.  Such maps extend
$C^{3,\theta}$ for every $\theta<\delta_+-2$, with the
identity boundary three-jet.  The rescaled chain rule proves closure
under composition and inversion.  Every fixed member acts by a
homeomorphism on $\mathscr A^+$: it acts boundedly on $E_0$,
preserves the prepared affine weight-$\delta_+$ condition, and
$q\mapsto\phi^*B(q)-B(q)$ is continuous into $E_0$.  This is
a statement about each fixed pullback, not about differentiability
of a moving group action.

The quotient
\[
 \mathscr Q^+=\mathscr A^+/\mathscr D^\infty_{\delta_+}
\]
is the marked solution space used here.  It agrees with the marked
geometric equivalence relation.  Indeed, boundary regularity gives
an ordinarily smooth representative of every Einstein class in
$\mathscr A^+$, and common geodesic collar coordinates put that
representative back into $\mathscr A^+$.  Here the prepared
weight $\delta_+>2$ gives a $C^{2,\theta}$ compactification and
permits the precise boundary-regularity theorem used in the bridge.
Conversely a marked isometry between two such normalized objects
is an isometry of their negative-scalar compactifications.
These have $C^{2,\theta}$ regularity, so the connection equation
upgrades it to $C^{3,\theta}$, with any fixed
$0<\theta<\delta_+-2$.  Equal prepared second metric jets imply
that its boundary three-jet is the identity.  After the first two
jets have been fixed, write
$F^x=x+A(y)x^3+o(x^3)$ and
$F^y=y+B(y)x^3+o(x^3)$.  In
$F^*(x^2H)=(x\circ F/x)^2x^2G$, the normal coefficient at order
$x^2$ gives $6A=2A$, and the mixed coefficient gives $3B=0$.
Thus $A=B=0$; all other third derivatives already vanish by
tangential differentiation of the first two jets.  Its displacement first
has weight $\delta\in(2,2+\theta)$.  Since every isometry is a harmonic
map into its target, the same expansion as in
Lemma~\ref{lem:hg-isometry-decay}, now using the second Einstein
metric as target, gives
\[
 (\nabla_G^*\nabla_G+3)V
 =f_{\delta_+}+E_{\delta_+}(V)+Q(V).
\]
The two metrics have the same boundary metric and differ by weight
$\delta_+$.  Hence $f_{\delta_+}$ and the coefficients of
$E_{\delta_+}$ have that weight, and $Q$ is quadratic.  Rescaled
elliptic estimates first give all derivatives at weight $\delta$;
since $2\delta>\delta_+$, the right side then has weight
$\delta_+$.  Invertibility of the vector
operator at both weights improves $V$ to weight $\delta_+$, with all
derivatives, exactly as in that lemma.  Thus this equivalence belongs
to $\mathscr D^\infty_{\delta_+}$.  This argument does not require either
identity map to be in a preselected reference gauge: the marked
isometry itself is harmonic.

This quotient is Hausdorff.  The source $\mathscr A^+$ is
metrizable, so it is enough to check closedness of the equivalence
relation on convergent sequences.  Weight-$\delta_+$ convergence gives
$C^{2,\theta}$ convergence of the negative-scalar
compactifications.  Proposition~\ref{prop:weak-marked-equivalence-properness}
gives a convergent marked isometry, and the preceding weight
improvement puts its limit in $\mathscr D^\infty_{\delta_+}$.
The proof of that proposition applies equally to an arbitrary
$C^\infty$-convergent sequence of boundary metrics: it uses their
$W^{2,p}$ convergence and uniform bounds, not a finite-dimensional
parametrization.  The relation is closed.  The quotient projection
is open because each fixed group element acts by a homeomorphism.
Its product is also an open quotient, so the closed relation implies
that the quotient diagonal is closed.  This proves Hausdorffness.

A local reference $H_G(q)$ used in the harmonic gauge differs from
$B(q)$ by an $E_\infty$ tensor, smoothly on each finite boundary
family.  Every finite-dimensional zero family from
Lemma~\ref{lem:hg-finite-reductions} therefore maps continuously to
$\mathscr A^+$ and hence to $\mathscr Q^+$.  This establishes
continuity in the actual global topology.

An element of $\mathscr Q^+$ is denoted $(q,[G])$.  Let
\[
 b:\mathscr Q^+\longrightarrow\operatorname{Met}_{\rm PSC}(S^3),
 \qquad b(q,[G])=q.
\]
The fixed linear extensions of the boundary group act continuously by
pullback on this space; replacing each pullback by the inverse pullback
if necessary gives a left action.  Its common fixed set is denoted
$\mathscr Q^+_\Gamma$.

\subsubsection{The local properties needed for continuation}
\label{subsec:finite-packet-interface}

We collect the local analytic properties and compactness results
used in the continuation argument.

\begin{enumerate}
\item At $z=(q,[G])\in\mathscr Q^+_\Gamma$, choose its coherent
isometric action $I^G_\gamma$ and an invariant gauge background.  For
every finite-dimensional, jointly smooth invariant boundary family
$q(a)$ with $q(a_0)=q$, the local gauged Einstein equation is a smooth
map
\[
 \Phi_G:A\times E_G^\Gamma\longrightarrow F_G^\Gamma.
\]
Here $E_G^\Gamma,F_G^\Gamma$ are complete weighted Banach spaces, the
metric represented by $(a,h)$ has boundary metric $q(a)$, and every
zero is an actual Einstein metric in the chosen gauge.  The derivative
in the metric variable at the central zero is a Fredholm operator
$L_G^\Gamma$ of index zero.

\item The obstruction map from smooth invariant boundary variations
is onto:
\[
 \mathcal O_G:C^\infty(S^2T^*S^3)^\Gamma
       \longrightarrow\operatorname{coker}L_G^\Gamma
       \cong K_G^\Gamma.
\]
The kernel of $L_G^\Gamma$ is also $K_G^\Gamma$.  Thus this is the
representation-level index calculation.

\item Whenever the full derivative of such a finite-parameter
equation is onto, its zero set near the central zero is a smooth
finite-dimensional manifold.  Its representatives have the spatial
regularity reserves proved in the gauge section.  The map from this
zero set to pairs consisting of a marked class and its finite
parameters is a homeomorphism onto a neighborhood in the corresponding
geometric solution set, as proved in
Lemma~\ref{lem:finite-topological-slice} below.  Every sufficiently close marked solution is
captured, including degenerate ones; two sufficiently close gauged
representatives of the same marked solution coincide.

\item For two such finite-dimensional solution families, the
boundary-fixing transformation placing one family in the other gauge
depends smoothly on the finite parameters.  This uses the
spatial reserves and the two-weight uniqueness argument from the
gauge and compactness-bridge sections. 

\item The full exact inverse image of any smooth compact PSC
boundary family is compact in the weighted marked topology.  This follows from Theorem~\ref{thm:compactness} and the
compactness bridge, and concerns arbitrary sequences of fillings.
\end{enumerate}

These local analytic properties and the compactness theorem provide
the ingredients for continuation.

\subsubsection{From a gauge slice to a neighborhood of marked fillings}

There is a distinction between constructing solutions of a gauged
equation and obtaining coordinates on the geometric solution space.
The first is an implicit-function theorem; the second requires
capture of every nearby filling and uniqueness of its representative
in the chosen gauge.  The following lemma makes that distinction
explicit, before transversality is imposed.

\begin{lemma}[Actual local quotient neighborhoods]
\label{lem:finite-topological-slice}
For a fixed finite smooth boundary family, the harmonic identity
slice is, locally, homeomorphic to an open neighborhood in the
marked solution set over that family.  Observe this assertion holds before
assuming transversality; it is simply about zero sets and their
topologies.  Whenever the invariant finite-parameter equation is
transverse, its finite-dimensional zero-set chart is consequently
an actual chart of the fixed marked solution set.
\end{lemma}

\begin{proof}
The gauge construction retracts nearby Einstein metrics to the
identity slice.  We must check continuity on the conormal-smooth
source and uniqueness of the representative of each nearby class.
The quotient topology then identifies the slice with an open
neighborhood of geometric solutions.

Work near a fixed smooth central representative and its smooth
reference family $H(a)$.  The nonlinear harmonic-map theorem
constructs, for every nearby $(a,h)\in\mathscr A^+$, a unique
small $F(a,h)$ placing $G=H(a)+h$ in harmonic identity gauge.
Its pushforward is Einstein and is a zero of the fixed-coordinate
equation, hence has all conormal spatial derivatives.  This
constructs a local gauge retraction $r$ onto the zero set in the
identity slice.  It fixes every point already in that slice.

This retraction is continuous in the topology actually used on
$\mathscr A^+$.  The vector solution is smooth into
$\mathcal H^{k+1,\beta}_{\delta_+}$.  To check pushforward, center at any
one conormal-smooth source $(a_0,h_0)$ and first compose by its
fixed gauge map, reducing its gauge to the identity.  For nearby
maps $F\to\operatorname{Id}$, decompose the remainder as
\[
 F_*h-h_0=F_*(h-h_0)+(F_*h_0-h_0).
\]
The first term is small in $E_0$ by the bounded composition estimate
on uniformly controlled diffeomorphisms.  The second is small
because this one fixed $h_0$ has an additional spatial derivative.
Similarly $F_*H(a)-H(a)$ tends to zero in weight $\delta_+$, since the
reference has every spatial derivative and $F-\operatorname{Id}$
has that weight.  Inversion of the maps is continuous here by the
same fixed-smooth-base decomposition.  These estimates prove
continuity of $r$ at each point in the conormal-smooth source. 

\smallskip\noindent\emph{Uniqueness of the representative.}
We next verify orbit uniqueness in a sufficiently small slice
neighborhood.  If it failed, two sequences of distinct gauged
representatives could approach the central one while being related
by marked isometries.  Properness of the negative-scalar
compactifications gives convergence of those isometries to the
identity in $C^{3,\theta}$, for every fixed
$\theta<\delta_+-2$.
Their boundary three-jets are the identity.  Choose
\[
 2<\delta_-<\delta_+,\qquad
 \max\{\beta,\delta_--2\}<\theta<\delta_+-2.
\]
The Taylor estimate gives convergence of their displacement to zero
in weighted $C^{2,\beta}_{\delta_-}$ in the rescaled charts.
Both comparison maps and the identity are harmonic from the first
gauged metric into the common smooth reference $H(a)$.  Subtracting
their equations and applying successive rescaled Schauder estimates
gives
\[
 \|V_i\|_{k+1,\beta,\delta_-}
 \le C\|V_i\|_{2,\beta,\delta_-}\longrightarrow0.
\]
The constants are uniform: the gauged Einstein equations bound each
needed conormal coefficient norm on the fixed low-order
neighborhood, and the smooth finite reference family has all these
bounds.  In coordinates the harmonic-map equation has fixed domain
principal coefficients and smooth target coefficients; the
difference equation has no term independent of $V_i$ because the
identity is also harmonic.  Differentiating successively therefore
retains the same decay, and the small first-derivative terms are
absorbed in the usual local Schauder estimate.  This explains the
displayed estimate and, in particular, its small right-hand side.
The uniqueness theorem at weight $\delta_-$ identifies the
comparison map with the identity.  This contradiction proves local
orbit uniqueness.

\smallskip\noindent\emph{The quotient topology.}
For completeness the quotient topology supplies openness, not just
a sequential representative statement.  Over the finite boundary
family let
\[
 \mathscr A_A=\{(a,(q,h))\in A\times\mathscr A^+:q=q(a)\}.
\]
Its projection to the marked family is open.  For a basic relatively
open set this follows from openness of the original orbit
projection, because a boundary-identity orbit preserves $q$ and
$a$.  Thus base change by $a\mapsto q(a)$ preserves this open
quotient property.

Shrink the local domain of $r$ to an open set $U$ which is stable
under $r$.  This is possible by first taking a smaller open set
$U_0$ and then replacing it by
$U_0\cap r^{-1}(U_0)$; since $r^2=r$, the new set is stable.
The slice $S=U\cap\operatorname{im}r$ maps bijectively to the
open marked neighborhood $\pi(U)$ by local orbit uniqueness.
For every relatively open $W\subset S$, the set $r^{-1}(W)$ is
open in $U$ and
\[
 \pi(r^{-1}(W))=\pi(W).
\]
The marked image $\pi(W)$ is therefore open.  The slice projection
is continuous and open, hence a homeomorphism.  Restricting this
statement to fixed classes and using
Lemma~\ref{lem:finite-local-invariance} gives the invariant
assertion.  A transverse invariant equation is a finite-dimensional
manifold by the Banach implicit-function theorem on its complete
spaces, so these are genuine finite-dimensional marked charts.
\end{proof}

\begin{lemma}[Closed fixed set and coherent action]
\label{lem:finite-fixed-closed}
The set $\mathscr Q^+_\Gamma$ is closed in $\mathscr Q^+$.  Every
element $(q,[G])$ in it carries a unique coherent isometric action
$I_\gamma$ extending the prescribed boundary action.
\end{lemma}

\begin{proof}
The common fixed set of finitely many continuous maps into a Hausdorff
space is closed.  If $[L_\gamma^*G]=[G]$, choose a boundary-fixing
equivalence $\phi_\gamma$ with
$\phi_\gamma^*L_\gamma^*G=G$ and put
$I_\gamma=L_\gamma\circ\phi_\gamma$.  These are isometries with
boundary values $\gamma$.  Both $I_\gamma I_\delta$ and
$I_{\gamma\delta}$ have the same boundary value; their quotient is a
boundary-identity isometry.  Boundary-identity rigidity makes it the
identity.  The same argument proves uniqueness, and faithfulness
follows from faithfulness of the boundary action.  
\end{proof}

\begin{lemma}[Invariant representatives in a local gauge]
\label{lem:finite-local-invariance}
Fix $z=(q,[G])\in\mathscr Q^+_\Gamma$ and its coherent action
$I_\gamma$.  For an invariant finite boundary family, sufficiently
nearby fixed marked classes are represented in the local gauge
precisely by the $I$-invariant zeros of its equation.
\end{lemma}

\begin{proof}
The background, extension, and gauge equation are invariant under
the fixed isometries $I_\gamma$.  Thus $I_\gamma^*G'$ is in the same
gauge whenever $G'$ is, and is close to $G$ if $G'$ is close to $G$.
If $[G']$ is a fixed marked class, $I_\gamma^*G'$ and $G'$ represent
the same class: $I_\gamma$ differs from $L_\gamma$ by a
boundary-fixing diffeomorphism.

The properness lemma for marked compactified isometries implies that
the equivalence between these two close representatives is close to
the identity.  Indeed, a contrary sequence would converge to a
nontrivial boundary-identity isometry of $G$, contradicting rigidity.
The local gauge uniqueness theorem therefore gives
$I_\gamma^*G'=G'$ for every $\gamma$.  Conversely, an invariant
representative gives a fixed marked class because
$I_\gamma L_\gamma^{-1}$ fixes the boundary.
\end{proof}

\begin{lemma}[Smooth finite-dimensional overlaps]
\label{lem:finite-overlaps}
Let two transverse finite-parameter local solution families use the
same actual boundary family $q(a)$ and be centered at fixed marked
classes.  Their images in the geometric solution set are open.  On
every nonempty overlap the change of finite-dimensional coordinates
is smooth and preserves $a$.  The changes obey the cocycle identities.
\end{lemma}

\begin{proof}
Openness and the homeomorphism assertion are the local capture
property.  Consider a point in the overlap.  Let $G_1^*,G_2^*$ be
its two representatives and let $\psi$ be their boundary-fixing
identification.  It is unique by boundary-identity rigidity.
Pull the first family back by this fixed $\psi$, and then apply the
smooth finite-family regauging theorem with the second background.
The resulting metric has the same actual boundary representative,
is in the second gauge, and represents the same marked solution.
Local uniqueness identifies it with the second representative.
Composing the smooth regauging map with the second local
finite-dimensional coordinate map proves smoothness of the
transition.  Reversing the roles proves smoothness of its inverse.

This also checks the equivariance issue.  If the fixed actions in the
two charts are $I^1$ and $I^2$, then $\psi$ intertwines them at the
overlap point.  Indeed, conjugating $I^1_\gamma$ by $\psi$ gives an
isometry of $G_2^*$ with boundary value $\gamma$, hence it equals
$I^2_\gamma$.  This is equality of fixed diffeomorphisms.  The entire
first family pulled back by $\psi$ is therefore $I^2$-invariant.
Uniqueness in the second gauge preserves this invariance.  

Finally the transition maps identify exactly the same marked metric
and the same boundary parameters.  Their compositions therefore
coincide on triple overlaps.  Equivalently, a discrepancy would be a
boundary-identity isometry.  This proves the cocycle identities.
\end{proof}

The preceding lemma constructs charts directly on a subspace of the
already Hausdorff marked solution space.  

\subsubsection{Choosing finitely many boundary directions}

We recall why a finite family suffices.  The obstruction space at a
filling is finite-dimensional.  The Green pairing shows that smooth
invariant boundary variations span it.  Thus we can select finitely
many such variations at each filling, and surjectivity persists in
a neighborhood.  Compactness will soon reduce these local choices
to one common finite family.

Proposition~\ref{prop:finite-invariant-obstruction} proves
surjectivity of $\mathcal O_G$ on smooth invariant boundary tensors,
using the strong boundary trace established in
Lemma~\ref{lem:strong-jacobi-boundary-trace}.  Concretely, if the
image were proper, choose a nonzero $\kappa\in K_G^\Gamma$
annihilating it.  Averaging any smooth boundary variation over the
prescribed action shows that $\kappa$ annihilates all smooth
variations.  The justified Green formula makes its smooth leading
boundary coefficient zero, and boundary unique continuation then
gives $\kappa=0$, a contradiction.  Finite-dimensionality of
$K_G^\Gamma$ permits choosing finitely many smooth invariant tensors
$\eta_1,\ldots,\eta_\ell$ with independent images spanning the
obstruction.  These are actual metric variations.

\subsubsection{Compactness precedes transversality}

We emphasize the order of the construction.  Compactness of the
\emph{unperturbed} path inverse image is established before any
finite-dimensional directions are selected. Only after compactness do we
take a finite cover of the solution set.

Theorem~\ref{thm:actual-weighted-path-compactness} establishes actual
compactness of the full exact inverse image in the weighted marked
topology specified here, before any finite-cover argument is used.
Its auxiliary geometric quotient has a separable integer-$C^K$
source and an open, Hausdorff quotient; this turns the
geometric sequential compactness into compactness.  The identity on
marked fillings maps that compact exact inverse image continuously
into the weighted quotient, with the retained finite boundary
parameter supplying the required $C^\infty$ convergence of the
actual boundary metrics.  The bridge takes $K\ge k+10$ in this
integer-topology argument, paying separately for the collar change,
third-order Taylor division, and an extra derivative controlling
the H\"older norm.  

\subsubsection{The continuation theorem}

\begin{theorem}[Finite-dimensional equivariant PE continuation]
\label{thm:finite-equivariant-continuation}
Let $\Gamma<SO(4)$ be finite and free on $S^3$.  Let
$q_s$, $0\le s\le1$, be a jointly smooth family of
$\Gamma$-invariant PSC metrics with $q_0=g_{\rm rd}$.
Then there is a complete PE metric $G$ on $B^4$ with
\[
 \operatorname{Ric}_G=-3G,
 \qquad\partial_\infty[G]=[q_1],
\]
and a coherent isometric action of $\Gamma$ extending the prescribed
boundary action.  The marked filling belongs to the connected
component of the hyperbolic filling in $\mathscr Q^+_\Gamma$ after
allowing the boundary representative to vary.  Degenerate fillings
are included throughout the construction.
\end{theorem}

\begin{proof}
The proof follows the mod-two degree principle: start with one
regular filling, enlarge the problem to make it transverse, and
observe that a compact one-manifold cannot have precisely one
boundary point.  The analytic work above is what allows that
simple principle to be used for actual Einstein fillings.

Let $X$ be the connected component of the hyperbolic marked filling
in $\mathscr Q^+_\Gamma$.  This component is closed. Suppose, for contradiction,
that there is no filling in $X$ with prescribed boundary metric
$q_1$.  Put
\[
 K=\{(z,s)\in X\times[0,1]:b(z)=q_s\}.
\]
The full inverse image over the path is compact in the actual
weighted quotient topology.  The fixed locus and the component
$X$ are closed, and the boundary map is continuous.  Therefore
$K$ is compact.  It is nonempty because it contains the hyperbolic
seed at time zero.

\paragraph{Protecting the endpoints.}
A perturbation will be introduced only in the interior of the
parameter interval.  We first isolate the unique hyperbolic seed
near $s=0$ and a putatively empty region near $s=1$.
Round-boundary rigidity gives exactly one marked filling at $s=0$,
the hyperbolic metric, and its invariant Jacobi space vanishes.
The fixed-gauge implicit-function theorem therefore gives a unique
regular local branch over $q_s$ near zero.  Local capture identifies
this with a neighborhood in the marked topology.  Compactness of
$K$ shows that every exact solution for all sufficiently small $s$
lies in this regular neighborhood: otherwise a sequence with
$s_i\to0$ would have a convergent subsequence to the unique seed,
contradicting the neighborhood condition.  The assumed absence of
time-one solutions, together with compactness, also gives an empty
terminal collar in $K$.

Choose $\varepsilon>0$ so that all points of $K$ with
$s\le3\varepsilon$ are regular and no points have
$s\ge1-3\varepsilon$.  Choose a smooth function
\[
 \chi:[0,1]\longrightarrow[0,1]
\]
which vanishes on neighborhoods of $0$ and $1$ and equals one on
$[2\varepsilon,1-2\varepsilon]$.

\paragraph{One finite family for the whole solution set.}
The obstruction directions chosen at different fillings must be
made into a single finite-dimensional parameter space.  This is
possible because the exact path inverse image is compact.  A
cutoff in $s$ keeps the two endpoint problems unchanged.
At every point of the compact middle part
\[
 K_{\rm mid}=K\cap\{2\varepsilon\le s\le1-2\varepsilon\},
\]
choose finitely many smooth invariant boundary tensors spanning its
invariant obstruction space.  Surjectivity of the coupled derivative
persists in a local gauge neighborhood, because it is a split
surjection with finite-dimensional kernel.  Finitely many such
neighborhoods cover $K_{\rm mid}$.  On each of these neighborhoods
the quantitative tensor estimates
\eqref{eq:uniform-gauge-quadratic}--\eqref{eq:uniform-gauge-derivative},
the finite-rank inverse \eqref{eq:uniform-augmented-inverse},
and the harmonic comparison estimate (HG.20) have constants fixed
at every prescribed finite order.  Shrinking the chosen finite cover
and taking the maximum gives one set of constants on it.  No
uniformity simultaneously in infinitely many derivative orders is
needed.  This is the exact point where compactness enters the
nonlinear estimates; the local estimates themselves do not prove
geometric compactness.

Let $\eta_1,\ldots,\eta_J$ span the union of their finite collections.
Use the single global boundary family
\begin{equation}\label{eq:finite-common-boundary-family}
 q_{s,p}=q_s+\chi(s)\sum_{j=1}^J p_j\eta_j,
 \qquad p\in P=\mathbb R^J.
\end{equation}
For a sufficiently small fixed parameter ball this is positive
definite and PSC: the original path is compact and PSC is open in
the $C^2$ topology.  It is jointly smooth and invariant.  The
perturbation vanishes on fixed collars of both endpoints.

At every point of $K$ the derivative of its local equation in
$(h,p)$ is onto.  On $K_{\rm mid}$ this follows from the selected
obstruction directions and $\chi=1$.  Elsewhere it follows from
invertibility of the invariant metric derivative, or there are no
points of $K$.  Since the metric derivative has index zero, the
full local equation in $(h,s,p)$ has index $J+1$.

\paragraph{The transverse solution manifold.}
The enlarged equation is surjective along the exact zero set.
The local implicit-function theorem gives finite-dimensional
solution manifolds, and the overlap theorem identifies them as
charts on one Hausdorff space of marked solutions.  Thus we obtain
an actual manifold.
Consider the set of actual solutions for the common family,
\[
 Z=\{(z,s,p)\in\mathscr Q^+_\Gamma\times[0,1]\times P:
                                  b(z)=q_{s,p}\}.
\]
At each point of $K\times\{0\}$, the local equation is a
submersion.  Its zero set is therefore a smooth
$(J+1)$-dimensional manifold.  Local capture gives a homeomorphism
onto an open neighborhood in $Z$; local invariant representatives
are justified by Lemma~\ref{lem:finite-local-invariance}.

Shrink each local finite-dimensional chart to a connected coordinate
ball (a half-ball at $s=0$).  Its projection to
$\mathscr Q^+_\Gamma$ is continuous and its image contains its
central point in $X$.  Therefore the entire image lies in the
connected component $X$. Extra finite boundary directions cause no problem:
the connected chart paths simply let the actual PSC boundary
metric vary within the same component.

By compactness choose finitely many of these chart images covering
$K\times\{0\}$, and let $\mathcal M$ be their union.  All charts
use the same family \eqref{eq:finite-common-boundary-family}.
Lemma~\ref{lem:finite-overlaps} gives smooth transitions and the
cocycle identities.  The topology is inherited from the Hausdorff
space $\mathscr Q^+\times[0,1]\times P$, so the resulting
manifold is Hausdorff.  A finite union of second-countable open
chart domains is second countable.  Thus $\mathcal M$ is an
ordinary smooth $(J+1)$-manifold, with boundary only at the genuine
time boundary.

At $s=0$ the local equation has invertible metric derivative and
$q_{0,p}=q_0$.  The boundary of $\mathcal M$ there consists, near
the exact fiber, of
\[
 \{(z_{\rm Hyp},0,p):p\text{ small}\}.
\]
There are no points with $s=1$: the boundary is unperturbed there
and the terminal fiber in $X$ was assumed empty.

\paragraph{Isolating the exact fiber.}
The compact solution set at zero auxiliary parameter has a
relatively compact neighborhood inside this finite-dimensional
manifold.  By choosing the auxiliary parameter sufficiently small,
we prevent new boundaries from appearing at the edge of that
neighborhood. 
Let
\[
 \varrho:\mathcal M\longrightarrow P,
 \qquad\varrho(z,s,p)=p.
\]
Because every chart lies in $X$ and the charts cover all exact
solutions in $X$ over the path,
\[
 \varrho^{-1}(0)=K\times\{0\}.
\]
Choose an open neighborhood $W$ of this compact fiber in
$\mathcal M$ whose closure is compact.  Such a neighborhood is
obtained by taking finitely many small coordinate balls with compact
closures.  The frontier
\[
 D=\overline W\setminus W
\]
is compact and disjoint from $\varrho^{-1}(0)$.  If $D$ is
nonempty, its continuous parameter norm has a positive minimum.
Choose $\delta>0$ so that $|\varrho|>2\delta$ on $D$ and so that
every $(z_{\rm Hyp},0,p)$ with $|p|<2\delta$ belongs to $W$.
If the frontier is empty only the latter condition is needed.

Observe compactness of the exact fiber and finite-dimensional
local compactness already isolate all sufficiently small parameter
fibers inside $W$.

\paragraph{Sard's theorem and the boundary count.}
A regular auxiliary parameter cuts the enlarged manifold down by
$J$ dimensions.  What remains is a compact one-manifold, with its
boundary controlled entirely by the two unperturbed endpoints.
When $J>0$, choose $p$ with $|p|<\delta$ which is a regular value
of $\varrho$ on a neighborhood of $\overline W$.
Ordinary finite-dimensional Sard applies: the dimension difference
is one, so $C^2$ regularity suffices.  On the boundary $s=0$ the
map is the identity parameter projection and is already a
submersion.  The set
\[
 L_p=\varrho^{-1}(p)\cap\overline W
\]
is a compact one-dimensional manifold.  It avoids $D$, hence has
no artificial boundary from $W$.  Its sole boundary point is
$(z_{\rm Hyp},0,p)$, and there are no time-one points.  This is
impossible: a compact one-dimensional manifold is a finite disjoint
union of circles and closed intervals, and its boundary has even
cardinality.  If $J=0$, the exact compact fiber itself is already
such a one-manifold with one boundary point, giving the same
contradiction without Sard.

The terminal fiber in $X$ is therefore nonempty.  For a terminal
representative $G$, membership in the fixed marked locus supplies
the coherent isometric action by
Lemma~\ref{lem:finite-fixed-closed}.  The original terminal boundary
problem is unchanged throughout: all finite perturbations vanish
at $s=1$.  
\end{proof}

The continuation argument rests on three ingredients: smooth
finite-dimensional solution charts, a Hausdorff topology of marked
Einstein fillings, and compactness of the inverse image of the
prescribed boundary path. Elliptic regularity makes the changes of
gauge smooth along the finite-dimensional families, while compactness
of marked isometries provides the global topological structure.
Together these properties yield the compact one-dimensional
cobordism used in the parity argument.

\section{Exact-path Fredholm parity}\label{sec:path-parity-proof}

For reference, we give a general form of the parity principle used
above.  Suppose a Fredholm map of index zero is followed along a
path in its target.  When the inverse image of the path is compact,
solutions cannot disappear without accounting for the boundary of
a one-dimensional cobordism.  Degenerate points can be handled by
adjoining finitely many auxiliary directions.

The compactness assumption is exactly
\eqref{eq:path-parity-compactness} below. The following theorem
states the abstract parity principle for a smooth Banach Fredholm map.
\begin{theorem}[Exact-path Fredholm parity]\label{thm:path-parity}
Let $X$ be a Hausdorff Banach manifold without boundary, let $B$ be a real Banach space, let $U\subset B$ be open, and let $F:X\to U$ be a $C^r$ Fredholm map of index zero with $r\ge2$.  Let $b:[0,1]\to U$ be $C^r$.  Assume
\begin{align}
F^{-1}(b([0,1]))&\text{ is compact},\label{eq:path-parity-compactness}\\
F^{-1}(b(0))&=\{x_1,\ldots,x_{2m+1}\},\qquad DF_{x_j}\text{ surjective for every }j.
\label{eq:path-parity-seed}
\end{align}
Then
\begin{equation}\label{eq:path-parity-conclusion}
F^{-1}(b(1))\neq\varnothing.
\end{equation}
\end{theorem}

If the path were already transverse to $F$, its inverse image would
be a compact one-dimensional manifold.  An odd number of regular
solutions at the initial endpoint and no solutions at the final
endpoint would give a compact one-manifold with an odd number of
boundary points, which is impossible.  The proof below makes this
work without assuming transversality: finite-dimensional auxiliary
directions produce the required manifold, and compact isolation
ensures that artificial boundary points are not introduced.

\begin{proof}[Proof of Theorem~\ref{thm:path-parity}]
We make the parity argument rigorous by adjoining finitely many
directions near the compact set of critical solutions, then
choosing a regular finite-dimensional parameter.

Assume for contradiction that
\[
F^{-1}(b(1))=\varnothing.
\]
Define the exact solution set
\[
K:=
\{(x,t)\in X\times[0,1]:F(x)=b(t)\}.
\]
It is closed in
\[
F^{-1}(b([0,1]))\times[0,1]
\]
and is therefore compact.

\medskip\noindent
\textbf{Step 1. A finite-dimensional transverse thickening.}
\par\smallskip

Let
\[
K_{\rm sing}
:=
\{(x,t)\in K:DF_x\text{ is not surjective}\}.
\]
Since surjectivity is open in the space of Fredholm operators,
$K_{\rm sing}$ is compact.  It is disjoint from $t=0$, by
\eqref{eq:path-parity-seed}, and from $t=1$, because the terminal fiber
is empty.  Hence its time projection is contained in a compact
subinterval of $(0,1)$.

For each $(x,t)\in K_{\rm sing}$ choose a finite-dimensional subspace
$V_{x,t}\subset B$ satisfying
\[
\im DF_x+V_{x,t}=B.
\]
Choose a linear map
\[
A_{x,t}:\mathbb R^{k_{x,t}}\longrightarrow B
\]
with image $V_{x,t}$.  Then
\[
(DF_x,A_{x,t}):
T_xX\oplus\mathbb R^{k_{x,t}}
\longrightarrow B
\]
is a surjective Fredholm operator.  Its kernel is finite-dimensional
and therefore complemented, so it has a bounded right inverse.
Split surjectivity consequently persists in a neighborhood of
$(x,t)$.

Choose a smooth function $\chi_{x,t}$ supported in $(0,1)$ and equal
to one on a neighborhood of $t$.  Shrinking the corresponding
neighborhood in $K_{\rm sing}$ if necessary, assume its time projection
lies in the region where $\chi_{x,t}=1$.  Compactness of
$K_{\rm sing}$ gives finitely many such choices, indexed by
$i=1,\ldots,N$.

Set
\[
P:=\bigoplus_{i=1}^N\mathbb R^{k_i}\cong\mathbb R^k
\]
and
\[
A(t)p:=\sum_{i=1}^N\chi_i(t)A_ip_i.
\]
All $\chi_i$ vanish on fixed collars of $t=0$ and $t=1$.  Define
\begin{equation}\label{eq:parity-splitting}
H:X\times[0,1]\times P\longrightarrow B,
\qquad
H(x,t,p):=F(x)-b(t)-A(t)p.
\end{equation}
At every point of $K\times\{0\}$, the derivative of $H$ in the
$(x,p)$ variables is split-surjective: at singular points this follows
from the construction above, and at regular points $DF_x$ is already
surjective.  Hence $DH$ is split-surjective there.  Since $F$ has
Fredholm index zero,
\[
\operatorname{ind}DH=k+1.
\]

\medskip\noindent
\textbf{Step 2. A finite-dimensional manifold near the exact fiber.}
\par\smallskip

By openness of split surjectivity and the Banach submersion theorem,
there is an open neighborhood
\[
O\subset X\times[0,1]\times P
\]
of $K\times\{0\}$ such that
\[
\mathcal M:=H^{-1}(0)\cap O
\]
is a $C^r$ finite-dimensional manifold of dimension $k+1$.

At $t=0$ the perturbation vanishes.  Moreover
$DF_{x_j}$ is surjective and has Fredholm index zero, hence is an
isomorphism.  The submersion theorem for manifolds with boundary
therefore shows that the only boundary of $\mathcal M$ near the exact
fiber occurs at $t=0$.  Since $A(t)=0$ on a collar of $t=1$ and
$F^{-1}(b(1))=\varnothing$, $\mathcal M$ has no $t=1$ boundary points.

If $K_{\rm sing}=\varnothing$, take $P=\{0\}$.  Then $K$ itself is a
compact one-dimensional manifold whose boundary consists of the
$2m+1$ points
\[
(x_1,0),\ldots,(x_{2m+1},0)
\]
and no time-one points, which is impossible.  Hence suppose
$k>0$.

\medskip\noindent
\textbf{Step 3. Isolating the compact zero-parameter fiber.}
\par\smallskip

Let
\[
\rho:\mathcal M\longrightarrow P,
\qquad
\rho(x,t,p)=p.
\]
By construction,
\begin{equation}\label{eq:parity-frontier}
\rho^{-1}(0)=K\times\{0\}.
\end{equation}
In particular this fiber is compact.

Since $\mathcal M$ is finite-dimensional and Hausdorff, choose an open
neighborhood
\[
K\times\{0\}\subset W\Subset\mathcal M
\]
with compact closure.  Let
\[
D:=\overline W\setminus W.
\]
Then $D$ is compact and
\[
D\cap\rho^{-1}(0)=\varnothing.
\]
Therefore, after choosing any norm on the finite-dimensional space $P$,
there is $\delta>0$ such that
\begin{equation}\label{eq:parity-isolation}
\|\rho(z)\|>2\delta
\qquad(z\in D).
\end{equation}

Because the initial seed set is finite and $W$ contains every
$(x_j,0,0)$, after decreasing $\delta$ if necessary,
\[
(x_j,0,p)\in W
\qquad
\text{whenever }\|p\|<2\delta
\]
for every $j$.

\medskip\noindent
\textbf{Step 4. Sard and parity.}
\par\smallskip

Apply finite-dimensional Sard's theorem to $\rho$ on an open
neighborhood of the compact set $\overline W$ in $\mathcal M$.
Here
\[
\dim\mathcal M=k+1,
\qquad
\dim P=k,
\]
so the assumption $r\ge2$ is sufficient.

On the boundary $t=0$, the map $\rho$ is already a submersion:
because $DF_{x_j}$ is an isomorphism and $A=0$ near $t=0$, the local
solution set is
\[
\{(x_j,0,p):p\text{ small}\},
\]
with $\rho$ equal to the parameter projection.

Choose a regular value
\[
p\in P,
\qquad
\|p\|<\delta.
\]
Then
\[
L_p:=\rho^{-1}(p)\cap\overline W
\]
is compact.  By \eqref{eq:parity-isolation} it is disjoint from $D$,
so it is in fact contained in $W$ and has no artificial boundary from
$\partial W$.  It is a compact $C^r$ one-dimensional manifold.

Its boundary consists exactly of its points on the genuine
time-boundary $t=0$.  Since the perturbation vanishes there, these are
precisely
\[
(x_1,0,p),\ldots,(x_{2m+1},0,p).
\]
There are no time-one boundary points.  Thus $L_p$ has an odd number of
boundary points.

Every compact one-dimensional manifold is a finite disjoint union of
circles and closed intervals, and hence has an even number of boundary
points.  This contradiction proves
\[
F^{-1}(b(1))\neq\varnothing.
\]
\end{proof}

\section{Pathwise continuation and equivariant filling}\label{sec:filling}

The continuation theorem,
Theorem~\ref{thm:finite-equivariant-continuation}, is formulated
for a path of $\Gamma$-invariant metrics on $S^3$.  We must now produce such a
path from a conformal class on the quotient.  The needed fact
comes from three-dimensional positive-scalar-curvature geometry:
any two positive-scalar-curvature metrics on the spherical space
form can be connected through metrics of the same kind.  Lifting
the path preserves the action of $\Gamma$ automatically.

This gives an equivariant filling.  The only further geometric
requirement for the Yamabe test is a common interior fixed point
of its finite isometry group.

We record the path construction explicitly because the continuation
theorem requires a jointly smooth path, not merely a continuous
one.

\begin{lemma}[Invariant PSC path]\label{lem:invariant-psc-path}
Let $C$ be a conformal class on $M=S^3/\Gamma$ with $Y(M,C)>0$.  There exists a smooth path $h_s$, $s\in[0,1]$, of positive-scalar-curvature metrics on $M$ such that $h_0$ is quotient-round and $h_1\in C$.  Consequently
\begin{equation}\label{eq:lifted-psc-path}
q_s:=\pi^*h_s
\end{equation}
is a smooth compact family of $\Gamma$-invariant PSC metrics on $S^3$ joining the round metric to a representative of $\pi^*C$.
\end{lemma}

\begin{proof}
The Yamabe problem first provides a positive-scalar-curvature
representative of the given class.  Contractibility of the
positive-scalar-curvature metric space provides a continuous path;
we then smooth that path without leaving the open PSC locus.

Since $Y(M,C)>0$, the Yamabe theorem gives a smooth representative
$h_1\in C$ with
\[
R_{h_1}>0.
\]
Let $h_0$ be the quotient-round metric.  By
\cite{BamlerKleiner2019}, the space
\[
\operatorname{Met}_{\rm PSC}(M)
\]
is contractible whenever it is nonempty.  Hence there is a continuous
path in $\operatorname{Met}_{\rm PSC}(M)$ joining $h_0$ to $h_1$.

The PSC condition is open in the $C^2$ topology.  More precisely, for
each metric $g$ in the compact image of the path, continuity of scalar
curvature as a map from $C^2$ metrics to $C^0$ functions, together
with $\min_M R_g>0$, gives a $C^2$-ball $U_g$ about $g$ whose closure
consists entirely of positive-definite metrics with positive scalar
curvature.  Shrinking the balls if necessary, we may take them to be
convex in the affine space of symmetric $2$-tensors.

A finite collection of these balls covers the compact image of the
path.  By the Lebesgue-number lemma, after subdividing $[0,1]$ into
finitely many intervals, the image of each subinterval lies in one of
these convex PSC neighborhoods.  On each subinterval, replace the
original path by the affine segment joining its endpoint metrics.
Convexity ensures that the resulting piecewise-smooth path remains in
the PSC locus.

Finally, smooth the finitely many parameter corners.  Since the PSC
locus is open, each corner has a convex $C^2$ neighborhood whose closure
still lies in the PSC locus; choosing the smoothing sufficiently
$C^2$-close to the piecewise-smooth path and leaving small neighborhoods
of $s=0$ and $s=1$ unchanged gives a smooth path
\[
h_s,\qquad 0\le s\le1,
\]
with the same endpoints and with
\[
R_{h_s}>0
\qquad\text{for every }s\in[0,1].
\]

Since $\pi:S^3\to M$ is a local isometry for the pullback metric,
\[
q_s:=\pi^*h_s
\]
is a smooth compact family of $\Gamma$-invariant PSC metrics on $S^3$,
joining the round metric to a representative of $\pi^*C$.
\end{proof}

The path is the last missing hypothesis in the continuation theorem.
We can therefore return from the analytic construction to the
geometric problem with which we began.

\begin{proposition}[Equivariant Poincar\'e--Einstein filling]
\label{prop:equivariant-filling}
Let $C$ be a smooth conformal class on $M=S^3/\Gamma$ with
$Y(M,C)>0$. There is a smooth complete PE metric $G$ on
$\operatorname{int}B^4$ such that
\begin{equation}\label{eq:equivariant-fill}
\Ric_G=-3G,\qquad \partial_\infty[G]=\pi^*C.
\end{equation}
It has a smoothly compactified representative and a coherent
isometric action
\begin{equation}\label{eq:coherent-action-map}
I:\Gamma\longrightarrow\operatorname{Isom}(B^4,G),
\qquad I_\gamma|_{S^3}=\gamma.
\end{equation}
The extension is unique for the marked filling. Its marked class
lies in the hyperbolic connected component of the fixed marked
solution space over PSC boundary metrics.
\end{proposition}

\begin{proof}
Lemma~\ref{lem:invariant-psc-path} gives the smooth invariant PSC
path $q_s=\pi^*h_s$, with $q_0=g_{\rm rd}$ and $h_1\in C$.
Theorem~\ref{thm:finite-equivariant-continuation} supplies a terminal
marked filling with boundary representative $q_1$. Its exact path
inverse image is compact in the topology used for the finite
continuation charts by
Theorem~\ref{thm:actual-weighted-path-compactness}. 

Lemma~\ref{lem:finite-fixed-closed} gives the coherent action.
Every group law and uniqueness statement follows from
Lemma~\ref{lem:boundary-identity}: two extensions of the same
boundary transformation differ by a boundary-identity isometry.
Classical boundary regularity, followed by the admissible global
collar extension, supplies the smoothly compactified representative.
The compactified isometries are smooth by their connection equation.
This proves all the stated assertions.
\end{proof}

The test function will be built from the distance to an interior
point.  To make it $\Gamma$-invariant, we need that point to be
fixed by every element of the group.  A fixed-point theorem for
finite actions on acyclic four-manifolds applies once the group
contains a nontrivial normal $p$-subgroup.  The following elementary
observation reduces our situation to that theorem.

\begin{lemma}[Normal $p$-subgroup]\label{lem:normal-p-subgroup}
Every nontrivial finite group $\Gamma<SO(4)$ acting freely on $S^3$ contains a nontrivial normal $p$-subgroup for some prime $p$.
\end{lemma}

\begin{proof}
If $2\mid|\Gamma|$, Cauchy's theorem provides $a\in\Gamma$ of order $2$.  Since $a\in SO(4)$ and $a^2=\Id$, all eigenvalues are $\pm1$.  Freeness on $S^3$ excludes a $+1$ eigenvalue; hence $a=-\Id$.  Thus $-\Id$ is the unique involution, so $\langle-\Id\rangle\cong C_2$ is central and normal.

Assume $|\Gamma|$ is odd.  By Feit--Thompson \cite{FeitThompson1963}, $\Gamma$ is solvable.  Let $A\neq1$ be a minimal normal subgroup.  Then $[A,A]\triangleleft\Gamma$ and $[A,A]\subsetneq A$ by solvability, so minimality gives $[A,A]=1$.  Hence $A$ is abelian.  Each primary component of $A$ is characteristic in $A$ and therefore normal in $\Gamma$; minimality implies that $A$ is a $p$-group for one prime $p$.  The subgroup $A[p]:=\{a\in A:a^p=1\}$ is nontrivial and characteristic in $A$, hence equals $A$ by minimality.  Therefore $A\cong(C_p)^r$ and is a nontrivial normal $p$-subgroup.
\end{proof}

The normal subgroup is purely algebraic.  It now combines
with the smooth extension of the group action to give the required
geometric fixed point.

\begin{proposition}[Common fixed point]\label{prop:common-fixed-point}
Let $G$ be the PE metric furnished by Proposition~\ref{prop:equivariant-filling}.  Then
\begin{equation}\label{eq:common-fixed-point}
\operatorname{Fix}_{B^4}(\Gamma):=\{x\in B^4:I_\gamma(x)=x\ \forall\gamma\in\Gamma\}\neq\varnothing.
\end{equation}
\end{proposition}

\begin{proof}
Choose the smooth $\Gamma$-invariant representative $q_1$
of the conformal infinity furnished by the construction in
Proposition~\ref{prop:equivariant-filling}, and let $x$
be its geodesic defining function. Boundary regularity gives
a smooth geodesic compactification. Uniqueness gives
$x\circ I_\gamma=x$, since $I_\gamma$ is a $G$-isometry
and its boundary action preserves $q_1$.
Hence each $I_\gamma$ is an isometry of this smooth
compactification in a collar and extends smoothly to
$\bar B^4$. It preserves orientation: its boundary
tangent map is the prescribed orientation-preserving
$\gamma\in SO(4)$, and its induced map on the inward
normal quotient is positive.

Restrict the action to
\[
\operatorname{int}B^4\cong\mathbb R^4.
\]
This is an acyclic smooth four-manifold.  By
Lemma~\ref{lem:normal-p-subgroup}, $\Gamma$ contains a nontrivial
normal $p$-subgroup.  We can therefore apply
\cite[Lemma~3(iii)]{GuazziMecchiaZimmermann2011} to give
 a global fixed point for the $\Gamma$-action on
$\operatorname{int}B^4$.  In particular,
\[
\operatorname{Fix}_{B^4}(\Gamma)\neq\varnothing.
\]
\end{proof}

\section{Proof of the sharp Yamabe bound}\label{sec:sharp-test}

\subsection{The Li--Qing--Shi boundary test}

The analytic and topological work is finished.  We have an Einstein
filling with the prescribed boundary action and a common fixed point
$p$ in its interior.  The remaining argument uses the relation
between distance, hyperbolic volume growth, and the Yamabe quotient
on the boundary.

The central construction is due to Li, Qing, and Shi
\cite{LiQingShi2017}.  It is worth noticing that the same
\emph{renormalized distance} used earlier to exclude a boundary
bubble now has a constructive role.  There it yielded rigidity
when the boundary was Euclidean.  Here its boundary value gives a
test function with the sharp constant.  Choosing the base point
$p$ to be fixed makes this function invariant automatically.

We state the comparison in the normalization used throughout the
paper, including the area ratio and the symmetry assertion.

\begin{proposition}[Li--Qing--Shi boundary test in the PE case]
\label{prop:LQS}
Let $(\operatorname{int}B^4,G)$ be a complete $C^3$ conformally
compact Poincar\'e--Einstein manifold,
\[
\Ric_G=-3G,
\]
with conformal infinity $(S^3,[q])$ of positive Yamabe type.
Choose any $p\in\operatorname{int}B^4$.

Let $x$ be the geodesic defining function associated with $q$, and put
\[
r:=-\log(x/2),
\qquad
t:=d_G(p,\cdot).
\]
The function
\[
e^{(r-t)/2}
\]
has a Lipschitz boundary trace; denote this trace by
\[
w:=\left.e^{(r-t)/2}\right|_{S^3}.
\]
Then
\[
0\neq w\in H^1(S^3),
\]
and, if a finite group of $G$-isometries preserves $q$ and fixes $p$,
$w$ is invariant under its boundary action. If
\[
A:=\int_{S^3}w^6\,dV_q
\]
and
\[
\Theta
:=
\lim_{T\to\infty}
\frac{\Area_G S_G(p,T)}
{2\pi^2\sinh^3T},
\]
then
\begin{equation}\label{eq:LQS}
\mathcal Q_q(w)\le6A^{2/3},
\qquad
A=2\pi^2\Theta,
\qquad
0<\Theta\le1.
\end{equation}
Consequently
\[
\mathcal Q_q(w)\le\Lambda.
\]
\end{proposition}

\begin{proof}
\emph{The boundary conformal factor.}
A Poincar\'e--Einstein metric satisfies
\[
\Ric_G=-3G,
\qquad
R_G=-12,
\]
so the curvature assumptions of the relative-volume theorem
\cite[Theorem~1.5]{LiQingShi2017} hold with equality.

We use the boundary conformal factor constructed in the proof of that
theorem.  With the normalization
\[
r=-\log(x/2),
\qquad
t=d_G(p,\cdot),
\]
their construction considers the interior function
\[
\psi:=e^{(r-t)/2}.
\]
It extends Lipschitzly to the conformal boundary, and its boundary trace
is precisely the function $w$ defined above.  The corresponding
boundary metric is
\[
\widetilde q=w^4q.
\]
Hence
\[
\Vol(S^3,\widetilde q)
=
\int_{S^3}w^6\,dV_q
=
A.
\]

The boundary Yamabe comparison in the proof of
\cite[Theorem~1.5]{LiQingShi2017}, beginning with their equation
(4.7) and completed by the estimates (4.10)--(4.12), gives in boundary
dimension three
\[
\int_{S^3}
\left(
8|dw|_q^2+R_qw^2
\right)dV_q
\le
6A.
\]
Dividing by
\[
\left(\int_{S^3}w^6\,dV_q\right)^{1/3}
=
A^{1/3}
\]
gives
\[
\mathcal Q_q(w)\le6A^{2/3}.
\]

\smallskip\noindent\emph{Volume comparison.}
The relative-area identity in the same proof gives
\[
\Theta
=
\frac{\Vol(S^3,\widetilde q)}
{\Vol(S^3,g_{\rm rd})}
=
\frac{A}{2\pi^2},
\]
and therefore
\[
A=2\pi^2\Theta.
\]
Since $w\not\equiv0$,
\[
A=\int_{S^3}w^6\,dV_q>0,
\]
and hence
\[
\Theta=\frac{A}{2\pi^2}>0.
\]
Bishop--Gromov comparison gives the upper bound
\[
\Theta\le1.
\]
Therefore
\[
0<\Theta\le1.
\]
It follows that
\[
\mathcal Q_q(w)
\le
6(2\pi^2)^{2/3}\Theta^{2/3}
=
\Lambda\Theta^{2/3}
\le
\Lambda.
\]

\smallskip\noindent\emph{Invariance under the group action.}
For the invariance assertion suppose the stated group action is present. Let
\[
I_\gamma\in\operatorname{Isom}(B^4,G)
\]
be the coherent extension of $\gamma\in\Gamma$.  Since $q$ is
$\Gamma$-invariant, $x\circ I_\gamma$ is a geodesic defining function
inducing the same boundary representative $q$.  Uniqueness of the
geodesic defining function gives
\[
x\circ I_\gamma=x.
\]
Since $p$ is fixed,
\[
t\circ I_\gamma
=
d_G(p,I_\gamma(\cdot))
=
d_G(I_\gamma^{-1}p,\cdot)
=
t.
\]
Hence
\[
w\circ\gamma=w
\]
on $S^3$.  Thus $w$ is $\Gamma$-invariant.
\end{proof}

We now return to the theorem stated at the beginning of the paper.
We have shown that the invariant test function reduces the remaining argument to the covering formula.

\begin{proof}[Proof of Theorem~\ref{thm:main}]
Fix a smooth conformal class $C$ on $M$.

If $Y(M,C)\le0$, then \eqref{eq:main-bound} follows from $N^{-2/3}\Lambda>0$.  Assume henceforth that
\begin{equation}\label{eq:positive-case}
Y(M,C)>0.
\end{equation}
Choose $h\in C$ with $R_h>0$, set
\begin{equation}\label{eq:lifted-boundary-metric}
q:=\pi^*h,
\end{equation}
and let $(B^4,G)$ be the equivariant PE filling given by Proposition~\ref{prop:equivariant-filling}.  By Proposition~\ref{prop:common-fixed-point}, choose
\begin{equation}\label{eq:fixed-basepoint}
p\in\operatorname{Fix}_{B^4}(\Gamma).
\end{equation}
Since $\Ric_G=-3G$, one has
\begin{equation}\label{eq:PE-curvature-exact}
\Ric_G=-3G,
\qquad
R_G=-12.
\end{equation}
Let $x$ be the geodesic defining function associated with the $\Gamma$-invariant representative $q$, and set
\begin{equation}\label{eq:r-t-definitions}
r:=-\log(x/2),
\qquad
t:=d_G(p,\cdot).
\end{equation}
For every $\gamma\in\Gamma$, uniqueness of the geodesic defining function and \eqref{eq:fixed-basepoint} imply
\begin{equation}\label{eq:r-t-invariance}
r\circ I_\gamma=r,
\qquad
t\circ I_\gamma=t.
\end{equation}
Thus Proposition~\ref{prop:LQS} gives a nonzero $\Gamma$-invariant function
\begin{equation}\label{eq:w-test}
w:=\left.e^{(r-t)/2}\right|_{S^3}\in H^1(S^3),
\qquad
\mathcal Q_q(w)\le\Lambda.
\end{equation}
Since $w\in H^1(S^3)$ is $\Gamma$-invariant, it descends through the
finite smooth covering
\[
\pi:S^3\longrightarrow M
\]
to a unique function
\[
\bar w\in H^1(M)
\]
satisfying
\begin{equation}\label{eq:w-descends}
w=\pi^*\bar w.
\end{equation}
Indeed, locally on an evenly covered coordinate neighborhood of $M$,
$\Gamma$-invariance identifies the Sobolev representatives on the
different sheets; these local representatives patch to $\bar w$.
Equivalently,
\[
H^1(M)\cong H^1(S^3)^\Gamma
\]
under pullback by the finite covering.
Using $q=\pi^*h$ and $\deg\pi=N$,
\begin{align}
\int_{S^3}\left(8|dw|_q^2+R_qw^2\right)dV_q
&=N\int_M\left(8|d\bar w|_h^2+R_h\bar w^2\right)dV_h,
\label{eq:numerator-covering}\\
\int_{S^3}|w|^6dV_q
&=N\int_M|\bar w|^6dV_h.
\label{eq:denominator-covering}
\end{align}
This is the point at which the order of the group enters the
estimate: the energy integral is multiplied by $N$, while the
normalizing $L^6$ integral is multiplied by $N^{1/3}$.  The quotient
therefore acquires the factor $N^{2/3}$.

Equations \eqref{eq:numerator-covering}--\eqref{eq:denominator-covering} imply
\begin{equation}\label{eq:covering-scaling}
\mathcal Q_q(w)=N^{2/3}\mathcal Q_h(\bar w).
\end{equation}
Therefore, by the definition of the full Yamabe infimum,
\begin{equation}\label{eq:terminal-chain}
Y(M,C)
\le\mathcal Q_h(\bar w)
=N^{-2/3}\mathcal Q_q(w)
\le N^{-2/3}\Lambda.
\end{equation}
The invariant test function constructed on the covering sphere
descends to an admissible function for the full Yamabe infimum on $M$.

\smallskip\noindent\emph{Equality in the quotient-round class.}
It remains to prove \eqref{eq:round-equality}.  Let $h_{\rm rd}$ be the quotient-round metric.  Since $R_{h_{\rm rd}}=6$ and $\Vol(M,h_{\rm rd})=\omega_3/N$, the constant test gives
\begin{equation}\label{eq:round-upper}
Y(M,[h_{\rm rd}])\le6(\omega_3/N)^{2/3}=N^{-2/3}\Lambda.
\end{equation}
Conversely, if $0\neq u\in H^1(M)$ and $v:=\pi^*u$, then
\begin{equation}\label{eq:round-lift-scaling}
\mathcal Q_{g_{\rm rd}}(v)=N^{2/3}\mathcal Q_{h_{\rm rd}}(u).
\end{equation}
Since $\mathcal Q_{g_{\rm rd}}(v)\ge\Lambda$, \eqref{eq:round-lift-scaling} implies $\mathcal Q_{h_{\rm rd}}(u)\ge N^{-2/3}\Lambda$.  Taking the infimum over $u\neq0$ and combining with \eqref{eq:round-upper} gives \eqref{eq:round-equality}.
Taking the supremum over all smooth conformal classes in
\eqref{eq:main-bound} and using the attained value
\eqref{eq:round-equality} gives \eqref{eq:main-sigma}.
\end{proof}

\providecommand{\bysame}{\leavevmode\hbox to3em{\hrulefill}\thinspace}
\providecommand{\MR}{\relax\ifhmode\unskip\space\fi MR }
\providecommand{\MRhref}[2]{%
  \href{http://www.ams.org/mathscinet-getitem?mr=#1}{#2}}
\providecommand{\href}[2]{#2}

\section*{Acknowledgments}

The author would like to express his gratitude to his advisor, Professor Xianzhe Dai, for suggesting this problem and for his mentorship and encouragement.

\end{document}